\documentclass[review,sort&compress,3p]{elsarticle}

\usepackage{lineno,hyperref}
\usepackage{graphicx}
\usepackage{slashbox}
\usepackage{amsmath, amssymb, amsfonts, amsthm}
\usepackage{mathrsfs}
\usepackage{url}
\usepackage{latexsym}
\usepackage[ruled,vlined]{algorithm2e}
\usepackage{color}
\usepackage{subfigure}
\usepackage{multirow}
\definecolor{darkblue}{rgb}{0,0,0.4}
\newtheorem{thm}{Theorem}[section]

\newtheorem{lem}[thm]{Lemma}
\theoremstyle{remark}
\newtheorem{rem}[thm]{Remark}

\newcommand{\ie}{{\it i.e.}}
\newcommand{\eg}{{\it e.g.}}

\graphicspath{{./figs/}} 

\def\bx{{\mathbf x}}
\def\by{{\mathbf y}}

\journal{}

\begin{document}

\begin{frontmatter}

\title{An efficient iterative method for reconstructing surface from point clouds}

\author[a]{Dong Wang \corref{cor1}}
\ead{dwang@math.utah.edu}

\address[a]{Department of Mathematics, University of Utah, Salt Lake City, UT, 84112 USA.}

\cortext[cor1]{Corresponding author}

\begin{abstract}
Surface reconstruction from point clouds is a fundamental step in many applications in computer vision. In this paper, we develop an efficient iterative method on a variational model for the surface reconstruction from point clouds. The surface is implicitly represented by indicator functions and the energy functional is then approximated based on such representations using heat kernel convolutions. We then develop a novel iterative  method to minimize the approximate energy and prove the energy decaying property during each iteration. We then use asymptotic expansion to give a connection between the proposed algorithm and interface dynamics. Extensive numerical experiments are performed in both 2- and 3- dimensional Euclidean spaces to show that the proposed method is simple, efficient, and accurate.
\end{abstract}


\end{frontmatter}

\linenumbers

\section{Introduction}
The problem of reconstructing surfaces from point clouds has attracted tremendous attention for the past decades \cite{Bolle_1991}. Point clouds are usually obtained using optical measuring devices such as laser scanners. It is a fundamental step in many applications such as computer graphics \cite{Wang_1991,Calakli_2011}, medical imaging \cite{Khan_2018},  manufacturing applications \cite{Bi_2010}, and many others \cite{Berger_2016,Bolle_1991}.

In this paper, we consider the reconstruction of $n-1$-dimensional manifold from a point cloud $\mathcal C \in \mathbb R^n$ (for example, a curve in $\mathbb R^2$ or a surface in $\mathbb R^3$).  To be specific, motivated from \cite{Zhao_2000}, we consider the following optimization problem:\begin{align}
\Gamma^* = \min_{\Gamma} E(\Gamma) \colon = \left(\int_{\Gamma} |d|^p \ ds \right)^{\frac{1}{p}}\label{pro:1}
\end{align}
where $d(\bx) = \min_{\by \in \mathcal C}{|\bx-\by|}$ is the distance from any point $\bx \in \mathbb R^n$ to the point cloud $\mathcal C$, $p$ is a positive number, $ds$ is the line/surface integral element, and $\Gamma^*$ is the surface to be reconstructed.

The goal of \eqref{pro:1} is to find an optimal surface in the sense of minimizing the $p$-norm of the distance function on the surface. In the continuous limit, when $d(x)$ is the distance function to a smooth surface $\Gamma_0$, it is easy to see that there are two global minimizers ({\it i.e.}; $\Gamma = \Gamma_0$ and $\Gamma = \emptyset$). However, in general and practical situations, $\mathcal C$ is a discrete set and may have noise or missing data. The problem is then interesting and complicated. In this case, the only one global minimizer is the trivial solution $\Gamma = \emptyset$ and the nontrivial local minimizer is more interesting. 

There have been many successful developments along the direction on the surface reconstruction from point cloud using a variational approach: mainly on different objective functionals and numerical methods; for instance, the Poisson surface reconstruction method \cite{kazhdan2006poisson}, moving least square projections  \cite{ztireli_2009}, reconstructing surfaces using anisotropic radius basis functions \cite{dinh2001reconstructing}, polygonal surface reconstruction \cite{Nan_2017},  and reconstruction using image segmentation formulations \cite{Liang_2012}. Besides these, there are many data driven approaches developed for surface reconstruction with priors for considering sampling density, level of noise, missing alignment, local surface smoothness, volumetric smoothness, absence of boundaries, symmetries, shape
primitives, or global regularity. We refer to \cite{Berger_2016} and references therein for a detailed survey on data-driven approaches.

In this paper, we use the indicator function to implicitly represent the surface and approximate the energy in \eqref{pro:1} using the indicator functions. Based on the new approximation, we derive an unconditional stable and efficient method to minimize the approximate energy to approximate the optimized solution. 

The method is motivated by the threshold dynamics method \cite{merriman1992diffusion,MBO1993,merriman1994motion} for simulating the motion by mean curvature. Recently, the method is interpreted as a minimizing movement scheme of a Lyapunov functional of indicator functions in \cite{esedoglu2015threshold}. The novel derivation and interpretation in \cite{esedoglu2015threshold} can then be directly generated to multiphase mean curvature motion with arbitrary surface tensions. The method has attracted much attention due to its simplicity and unconditional stability.  It has subsequently been extended to many problems, including the problem of area or volume preserving interface motion \cite{ruuth2003simple,jacobsauction},  wetting dynamics \cite{xu2016efficient,Wang_2019}, image processing \cite{esedog2006threshold,merkurjev2013mbo,wang2016efficient,wang2019iterative}, target-value harmonic maps \cite{osting2017generalized,wang2019interface,wang2018diffusion,Osting_2020},  high-order geometric motions \cite{Esedoglu_2008}, and so on. 

Threshold dynamics method can also be generalized and extended to modeling anisotropic interface motions via constructing a specific kernel (instead of heat kernel) (see \cite{merriman2000convolution,ruuth2001convolution,Elsey_2017} for more details). It is obvious that the gradient flow of \eqref{pro:1} is also an anisotropic interface motion where the anisotropy comes from varying $|d|^p$ in the domain. However, when the anisotropy is fixed in the computational domain, approaches from constructing kernels would fail.  

In this work, we derive a novel threshold dynamics type method for the application in reconstructing surface from point clouds. We understand this as an iterative approach for minimizing a given surface energy. Also, we perform asymptotic expansions to formally analysis the anisotropic dynamics of the interface during each iteration, to build a connection between the method and interface dynamics. 

The rest of this paper is organized as follows. In Section~\ref{sec:der}, we introduce new approximations of the energy, derive the numerical method based on the approximation, and prove the unconditional stability property of the method. We discuss some interpretations of dynamics of the interface in Section~\ref{sec:connection}. In Section~\ref{sec:accele}, we provide an accelerated version of the proposed algorithms.  In Section~\ref{sec:num}, we describe the numerical implementation and illustrate the performance of the method using extensive numerical experiments. We draw some conclusions and discussions in Section~\ref{sec:con}. 

\section{Derivation of the method}\label{sec:der}
\subsection{Approximation of the energy in \eqref{pro:1}.}
In this paper, since we focus on the codimension 1 interface ({\it e.g.}, a closed curve in  $\mathbb R^2$, a surface in $\mathbb R^3$, or higher dimensions),  we use indicator functions to implicitly represent the interface. That is, we denote  
\begin{equation} u(x) = \begin{cases} 1  & \textrm{if} \  \bx  \in \Omega_\Gamma \\
0 & \textrm{otherwise} \end{cases} 
\end{equation}
where $\Omega_\Gamma$ is the domain bounded by an interface $\Gamma$.

Under this representation, as shown in \cite{Miranda_2007}, as $\tau \searrow 0$, the boundary integral $\int_{\Gamma} |d|^p \ ds$ is approximated by a short time heat flow (\ie; Gaussian convolution):
\begin{align}
\int_{\Gamma} |d|^p \ ds \approx \sqrt{\frac{\pi}{\tau}} \int_{\mathbb R^n}  |d|^p \ u \ G_\tau * (1-u) \ d\bx,
\end{align}
or 
\begin{align}
\int_{\Gamma} |d|^p \ ds \approx \sqrt{\frac{\pi}{\tau}} \int_{\mathbb R^n}  |d|^p \ (1-u) \ G_\tau * u \ d\bx
\end{align}
where $$G_\tau(\bx) = \frac{1}{(4\pi\tau)^{n/2}} \exp\left(-\frac{|\bx|^2}{4\tau}\right),$$ $*$ denotes the convolution, and $\tau$ is a free parameter. To keep the symmetry of the formula with respect to $u$ and $1-u$, we approximate $\int_{\Gamma} |d|^p \ ds$ by $E^\tau(u)$:
\begin{align}
 E^\tau(u) \colon =\frac{1}{2}\sqrt{\frac{\pi}{\tau}} \left(\int_{\mathbb R^n}  |d|^p \ u \ G_\tau * (1-u) \ d\bx+\int_{\mathbb R^n}  |d|^p \ (1-u) \ G_\tau * u\ d\bx\right)\label{eq:approx}
\end{align}
or \begin{align}
 E^\tau(u) \colon =\sqrt{\frac{\pi}{\tau}}\int_{\mathbb R^n}  |d|^{\frac{p}{2}} u \ G_\tau * \left(|d|^{\frac{p}{2}} (1-u)\right) \ d\bx.\label{eq:approx1}
\end{align}
We note that in the special case when $|d|^p = 1$, the formula reduces to the perimeter or the surface area. It can be used to model multiphase motion with arbitrary surface tensions \cite{esedoglu2015threshold}. The convergence of \eqref{eq:approx1} as $\tau\searrow 0 $ is rigorously proved in \cite{Hu2020} when they study wetting dynamics with surfactant.

Now, we arrive at the following problem: finding $u^{\tau,\star}$ such that
\begin{align}
u^{\tau, \star} = \arg\min_{u \in \mathcal{B}} E^\tau(u)  \label{pro:2}
\end{align} where 
$$\mathcal{B} \colon = \{u\in BV(\Omega,\mathbb{R}) \ |  \ u =\{0,1\} \} $$
and $BV(\Omega,\mathbb{R})$ denotes the bounded-variation functional space. 

\subsection{Derivation of the method based on \eqref{eq:approx}.}
In this section, we iteratively solve \eqref{pro:2}. We first note that problem \eqref{pro:2} is a minimization problem of a functional on a nonconvex set $\mathcal B$. Using the relaxation approach in \cite{esedoglu2015threshold}, we relax this problem to an equivalent problem: finding $u^{\tau,\star}$ such that
\begin{align}
u^{\tau, \star} = \arg\min_{u \in \mathcal{K}} E^\tau(u)  \label{pro:3}
\end{align} where 
$$\mathcal{K} \colon = \{u\in BV(\Omega,\mathbb{R}) \ |  \ u \in [0,1] \}.$$ 
The equivalence between \eqref{pro:2} and \eqref{pro:3} is guaranteed in the following lemma.

\begin{lem}
Problem~\ref{pro:2} is equivalent to problem~\ref{pro:3}. That is, 
\[\arg\min_{u \in \mathcal B}E^\tau(u) =  \arg\min_{u \in \mathcal K}E^\tau(u).\]
\end{lem}
\begin{proof}
It is easy to see that \[\min\limits_{u \in \mathcal K}E^\tau(u)\leq \min\limits_{u \in 
\mathcal B}E^\tau(u)\] from the fact that $\mathcal B \subsetneqq \mathcal K$. 

To finish the proof, we need only to prove $$\arg\min\limits_{u\in\mathcal K} E^\tau(u) \in \mathcal B.$$ Assume it is not true and the minimizer is $u^*$, then there exists a set $A\subset \Omega$ with nonzero measure and $c>0$ such that 
\[u^*(x) \in (c,1-c), \ \ \ \ \ \forall x \in A.\]
Denote $u^t = u^*+t \chi_A$ where $\chi_A$ is the indicator function of $A$, we have $u^t \in \mathcal K$ for any $|t|<c$. Direct computation yields 
\begin{align*}
\frac{d^2 E^\tau(u^t)}{dt^2} = -2\sqrt{\frac{\pi}{\tau}}\int_{\mathbb R^n}  |d|^p \chi_A G_\tau * \chi_A \ d\bx.
\end{align*}
Because $|d|^p \geq 0$ and  $|d|^p = 0$ only on a set with zero measure, we have $\frac{d^2 E^\tau(u^t)}{dt^2}<0$, especially at $t=0$ (\ie, $u^*$).  This contradicts with the assumption that  $u^*$ is a minimizer. 

Therefore, we have that the minimizer of $\min\limits_{u\in\mathcal K} E^\tau(u)$ must be attained in $\mathcal B$.
\end{proof}

Now, we use an iterative method to solve \eqref{pro:3}. Without loss of generality, assume the $k$-th iteration is known, we find $k+1$-th iteration as follows. At the $k$-th iteration $u^k$, we compute the linearization of $E^\tau(u)$ at $u^k$:
 \begin{align}
 L^\tau(u,u^k) =  \frac{1}{2}\sqrt{\frac{\pi}{\tau}}\int_{\mathbb R^n} u \varphi^k \ d \bx 
 \end{align}
 where \[ \varphi^k =  |d|^p G_\tau * (1-2u^k) + G_\tau * \left(|d|^p (1-2u^k)\right) .\]
 
Based on the sequential linear programming, we then compute the $k+1$-th iteration $u^{k+1}$ by solving the following linearized problem:
 \begin{align}
u^{k+1} = \arg\min_{u \in \mathcal{K}} L^\tau(u,u^k).  \label{pro:4}
\end{align}
Since $u$ only takes value in $[0,1]$ (\ie, a bounded set), the minimization problem \eqref{pro:4} can be solved point-wisely. That is, at each point $\bx$,
\[u^{k+1}(\bx) = \arg\min_{u(\bx) \in [0,1]} u(\bx)\varphi^k(\bx) .\]
This is exactly solved via the following thresholding step:
\[u^{k+1}(\bx) = \begin{cases} 1, & \textrm{if} \  \varphi^k(\bx) < 0, \\
0, & \textrm{otherwise}. \end{cases}
\]

The algorithm is summarized into Algorithm~\ref{a:MBO}.

\begin{algorithm}[ht!]
\DontPrintSemicolon
 \KwIn{$\Omega$: computational domain; $d$: distance function to the point cloud; $\tau > 0$; and $u^0 \in \mathcal{B}$.}
 \KwOut{$u^\star \in \mathcal{B}$.}
 \While{not converged}{
{\bf 1.} For the fixed $u^k$, compute 
\[ \varphi^k(\bx) = |d|^p G_\tau * (1-2u^k) + G_\tau * \left(|d|^p (1-2u^k)\right). \]
{\bf 2.} Set
\[u^{k+1}(\bx) = \begin{cases}1 \ \  \textrm{if}  \ \phi^k(\bx)\leq 0 ,   \\
0 \ \  \textrm{otherwise}.  \end{cases}\]
 }
\caption{The iterative method for approximating minimizers of \eqref{eq:approx}. }
\label{a:MBO}
\end{algorithm}

\begin{rem}
The criteria for ``{\it convergence}'' in all proposed algorithms is that $u^{k+1} (\bar \bx) = u^k (\bar \bx)$ on each grid point $\bar \bx$ in the discretized domain. In other words, the value of $u$ at no grid point is changing (from $1$ to $0$ or $0$ to $1$).
\end{rem}

\subsection{Derivation of the method based on \eqref{eq:approx1}.}\label{sec:algorithm2}
Similar to the derivation in Algorithm~\ref{a:MBO}, we use the same relaxation and linearization approach to derive another unconditional stable method in Algorithm~\ref{a:MBO2}. The details are omitted here.

\begin{algorithm}[ht!]
\DontPrintSemicolon
 \KwIn{$\Omega$: computational domain; $d$: distance function to the point cloud; $\tau > 0$; and $u^0 \in \mathcal{B}$.}
 \KwOut{$u^\star \in \mathcal{B}$.}
 \While{not converged}{
{\bf 1.} For the fixed $u^k$, compute 
\[ \varphi^k(\bx) = G_\tau * \left(|d|^{\frac{p}{2}} (1-2u^k)\right). \]
{\bf 2.} Set
\[u^{k+1}(\bx) = \begin{cases}1 \ \  \textrm{if}  \ \phi^k(\bx)\leq 0 ,   \\
0 \ \  \textrm{otherwise}.  \end{cases}\]
 }
\caption{The iterative method for approximating minimizers of \eqref{eq:approx1}. }
\label{a:MBO2}
\end{algorithm}

\begin{rem}
We remark here that, at each iteration, the computational complexity of Algorithm~\ref{a:MBO2} is about the half of the computational complexity of Algorithm~\ref{a:MBO} because only one convolution needs to be computed. 
\end{rem}

As for Algorithm~\ref{a:MBO2}, we discuss the stability of the method in the sense of the monotonicity of the approximate energy \eqref{eq:approx1} (\ie; $E^\tau(u^{k+1}) \leq E^\tau(u^k)$).  

\begin{thm}
Suppose $u^k$ ($k = 1, 2, \cdots$) are computed from Algorithm~\ref{a:MBO2}, we have 
\[ E^\tau(u^{k+1}) \leq E^\tau(u^k)\]
with $E^\tau(u)$ being defined in \eqref{eq:approx1}.
\end{thm}
\begin{proof}
As for $E^\tau(u)$ defined in \eqref{eq:approx1}, the linearization of $E^\tau(u)$ at $u^k$ is defined by:
$$L^\tau(u,u^k)  = \sqrt{\frac{\pi}{\tau}} \int_{\mathbb R^n} |d|^{\frac{p}{2}} u G_\tau * (1-2u^k) \ d\bx.$$
Note that we have \[E^\tau(u^k) =  L^\tau(u^k,u^k)+ \sqrt{\frac{\pi}{\tau}}\int_{\mathbb R^n} |d|^{\frac{p}{2}}  u^k G_\tau * \left(|d|^{\frac{p}{2}} u^k\right)\ d \bx\] and 
\begin{align*}
E^\tau(u^{k+1}) =  & L^\tau(u^{k+1},u^k) + 2\sqrt{\frac{\pi}{\tau}}\int_{\mathbb R^n} |d|^{\frac{p}{2}}  u^{k+1} G_\tau * \left(|d|^{\frac{p}{2}} u^k\right) \ d \bx  -\sqrt{\frac{\pi}{\tau}}\int_{\mathbb R^n} |d|^{\frac{p}{2}}  u^{k+1} G_\tau * \left(|d|^{\frac{p}{2}} u^{k+1}\right) \ d \bx.  \end{align*}
Because $u^{k+1}$ is the solution from the sequential linear programming, we have $L^\tau(u^{k+1},u^k) \leq L^\tau(u^k,u^k)$. Then, we compute 
\begin{align*}
E^\tau(u^{k+1})  - E^\tau(u^k) = L^\tau(u^{k+1},u^k) - L^\tau(u^k,u^k) + \mathcal L
\end{align*}
where \begin{align*}
\mathcal L = & - \sqrt{\frac{\pi}{\tau}}\int_{\mathbb R^n}  |d|^{\frac{p}{2}}  (u^{k+1}-u^k)  G_\tau * \left(|d|^{\frac{p}{2}}  (u^{k+1}-u^k) \right)  \ d \bx \\
=& - \sqrt{\frac{\pi}{\tau}} \int_{\mathbb R^n}  \left[G_{\tau/2} * \left(|d|^{\frac{p}{2}}  (u^{k+1}-u^k) \right) \right]^2  \leq 0 .\end{align*}
Therefore, we are led that 
$E^\tau(u^{k+1})  - E^\tau(u^k) \leq 0$.\end{proof}

\section{Interpretations of interface dynamics.}\label{sec:connection}
Note that, at each iteration, both Algorithm~\ref{a:MBO} and Algorithm~\ref{a:MBO2} start from an indicator function and end by another indicator function, which implicitly determines a motion of a front. The free parameter $\tau$ can be interpreted as the time step in the dynamics of the interface.

When $|d|^p = 1, \forall \bx \in \Omega$, it is easy to see that both algorithms reduce to the original MBO method \cite{merriman1992diffusion} where the front evolves along its normal direction by mean curvature. In this case, if there is no other constraint (\eg, volume constraint), the front evolves to the unique minimizer (\ie, the global minimizer $\emptyset$). 

However, when $d(\bx)$ is varying in the space, we perform the following asymptotic expansions to compute the motion law of the front during the iterations. For the convenience, we write $\psi = |d|^{\frac{p}{2}}$ and use an asymptotic expansion to expand $\varphi^k(\bx)$ in Algorithm~\ref{a:MBO2} with respect to a small parameter $\tau$. Without loss of generality, we set up the interface as that in Figure~\ref{fig:diag}. Specifically, we assume that  $u^k$ is the indicator function for the region where $x_2 \geq g(x_1)$ in Figure~\ref{fig:diag}. Furthermore, we assume the point of interest is the origin and $g'(0) = 0$.

\begin{figure}[t!]
\centering
\includegraphics[width = 0.8 \textwidth, clip, trim = 4cm 2cm 4cm 0cm]{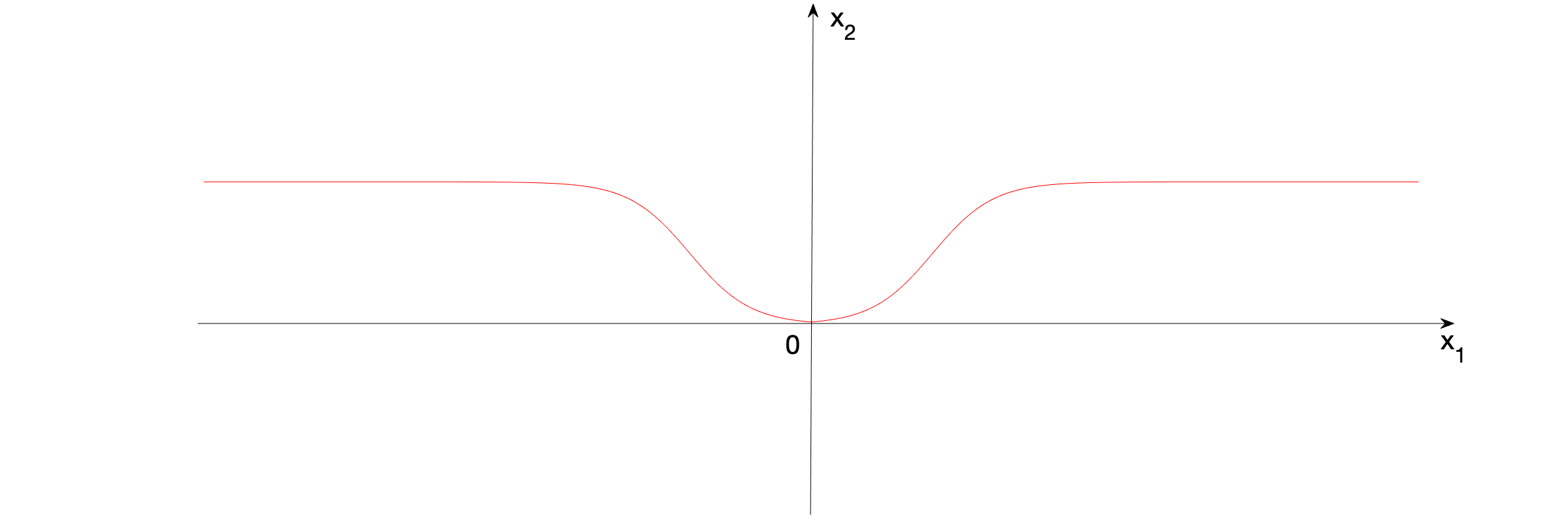}
\caption{A diagram for the set up for the asymptotic analysis. See Section~\ref{sec:connection}.} \label{fig:diag}
\end{figure}

In the follows, we focus on expanding $\varphi^k$ into a series with respect $\tau$ and find zero level set of $\phi^k(\bx)$ which corresponds to the new position of the front (according to Algorithm~\ref{a:MBO2}):
\begin{align}
\varphi^k(\bx)& = G_\tau * \left(\psi (1-2u^k)\right)  \nonumber  \\
& =\frac{1}{4\pi\tau} \int_{-\infty}^{\infty} \int_{-\infty}^{\infty}  \exp\left(-\frac{|\bx-\by|^2}{4\tau}\right)\psi(\by)(1-2u^k(\by))  \ dy_2dy_1 \\
& = \frac{1}{4\pi\tau} \int_{-\infty}^{\infty} \int_{-\infty}^{\infty}  \exp\left(-\frac{|\bx-\by|^2}{4\tau}\right)\psi(\by) \ dy_2 dy_1 - \frac{1}{2\pi\tau} \int_{-\infty}^{\infty} \int_{g(y_1)}^{\infty}  \exp\left(-\frac{|\bx-\by|^2}{4\tau}\right)\psi(\by) \ dy_2dy_1. \nonumber 
\end{align}
Evaluating $\varphi^k(\bx)$ at $(0,x_2)$ (\ie; the front moves along its normal direction), we have
\begin{align}
\varphi^k(0,x_2) = I_1-I_2
\end{align}
where 
\[I_1 =  \frac{1}{4\pi\tau} \int_{-\infty}^{\infty} \int_{-\infty}^{\infty}  \exp\left(-\frac{y_1^2+(x_2-y_2)^2}{4\tau}\right)\psi(y_1,y_2) \ dy_2dy_1 \]
and
\[I_2 = \frac{1}{2\pi\tau} \int_{-\infty}^{\infty} \int_{g(y_1)}^{\infty} \exp\left(-\frac{y_1^2+(x_2-y_2)^2}{4\tau}\right)\psi(y_1,y_2) \ dy_2dy_1. \]
Assume $\psi(y_1,y_2)$ is smooth almost everywhere and expand $\psi(y_1,y_2)$ into its Taylor's series around $(0, x_2)$:
\begin{align}
\psi(y_1,y_2) = \psi_{0,0}+y_1 \psi_{1,0}+(y_2-x_2) \psi_{0,1}+\frac{y_1^2}{2} \psi_{2,0}+\frac{(y_2-x_2)^2}{2} \psi_{0,2} + y_1(y_2-x_2)\psi_{1,1}+ \cdots
\end{align}
where $\psi_{m,n}$ denotes the mixed partial derivative at $(0,x_2)$ with $m$-th order derivative with respect to $y_1$ and $n$-th order derivative with respect to $y_2$. 

As for $I_1$, direct computation yields
\begin{align}
I_1  &= \frac{1}{4\pi\tau} \int_{-\infty}^{\infty} \int_{-\infty}^{\infty}  \exp\left(-\frac{y_1^2+y_2^2}{4\tau}\right)\psi(y_1,y_2+x_2) \ dy_2dy_1 \nonumber \\
& =  \sum_{m,n=0}^\infty \psi_{m,n} \frac{1}{4\pi\tau} \int_{-\infty}^{\infty} \int_{-\infty}^{\infty}  \exp\left(-\frac{y_1^2+y_2^2}{4\tau}\right) y_1^m y_2^n
 \ dy_2dy_1 \\
 & =  \sum_{m,n=0}^\infty \psi_{m,n} \frac{1}{4\pi\tau} \int_{-\infty}^{\infty}  \exp\left(-\frac{y_1^2}{4\tau}\right) y_1^m \ dy_1 \int_{-\infty}^{\infty} \exp\left(-\frac{y_2^2}{4\tau}\right) y_2^n \ dy_2 \nonumber \\
 & = \sum_{m,n=0}^\infty \psi_{m,n} \eta_{m,n} \nonumber
\end{align}
where $\eta_{m,n} = \xi_m \xi_n$ and $$\xi_m = \frac{1}{2\sqrt{\pi\tau}} \int_{-\infty}^\infty e^{-\frac{y^2}{4\tau}} y^m \ dy$$ which can be explicitly computed:
\[ \xi_m = \begin{cases} 0  & \ \textrm{if}  \  \ m \  \  \textrm{is odd}, \\
(m-1)!! \ 2^{m/2} \tau^{m/2} & \ \textrm{if}  \  \ m \  \  \textrm{is positive and even}, \\
1 & \ \textrm{if}  \  \ m =0.
\end{cases} 
\]
Therefore, 
\[I_1 =  \psi_{0,0}+2(\psi_{2,0}+\psi_{0,2})\tau+4(\psi_{2,2}+3\psi_{4,0}+3\psi_{0,4}) \tau^2+o(\tau^2).\]

As for $I_2$, because $g(0) = g'(0) = 0$, we expand $g(y_1)$ into its Taylor's series around $0$ by
\begin{align}
g(y_1) = \frac{g^{(2)}}{2}y_1^2 + \frac{g^{(3)}}{6}y_1^3+ \frac{g^{(4)}}{24}y_1^4 +\cdots   
\end{align} 
where $g^{(n)}$ denotes the $n$-th order derivative at $0$. After changing of variables, $I_2$ is then written into 
\begin{align}
I_2 = II_1-II_2
\end{align}
where \[II_1 = \frac{1}{2\pi\tau} \int_{-\infty}^{\infty} \int_{0}^{\infty} \exp\left(-\frac{y_1^2+y_2^2}{4\tau}\right)\psi(y_1,y_2+x_2) \ dy_2dy_1\]
and 
\[II_2 =\frac{1}{2\pi\tau} \int_{-\infty}^{\infty} \int_{0}^{-x_2+\frac{g^{(2)}}{2}y_1^2 + \frac{g^{(3)}}{6}y_1^3+ \frac{g^{(4)}}{24}y_1^4 +\cdots} \exp\left(-\frac{y_1^2+y_2^2}{4\tau}\right)\psi(y_1,y_2+x_2) \ dy_2dy_1.\]
For $II_1$, we have
\begin{align*}
II_1 = \sum_{m,n=0}^\infty \psi_{m,n} \eta'_{m,n} 
\end{align*}
where $\eta'_{m,n} = \xi_m \xi'_n$ and $$\xi'_n = \frac{1}{\sqrt{\pi\tau}} \int_0^\infty e^{-\frac{y^2}{4\tau}} y^n \ dy$$ 
which can be explicitly computed as follows:
\[\xi'_n = \begin{cases}
\frac{1}{\sqrt{\pi\tau}}(n-1)!!(2\tau)^{(n+1)/2} & \ \textrm{if}  \  \ n \ \  \textrm{is odd},  \\
\xi_n & \ \textrm{if}  \  \ n \ \ \textrm{is even}.
\end{cases}
\]
Therefore, 
\[II_1 = \psi_{0,0}+ \frac{2\psi_{0,1}}{\sqrt{\pi}}\tau^{1/2} + 2(\psi_{2,0}+\psi_{0,2} )\tau+ \frac{4(2\psi_{0,3}+\psi_{2,1})}{\sqrt{\pi}} \tau^{3/2}+4(\psi_{2,2}+3\psi_{4,0}+3\psi_{0,4})  \tau^2+o(\tau^2).  \]
For the computation of $II_2$, we further expand $e^{-\frac{y_2^2}{4\tau}}$ into its Taylor's series and compute:
\begin{align}
II_2  = \sum_{m,n=0}^\infty   \psi_{m,n} \zeta_{m,n}
\end{align}
where 
\begin{align}
\zeta_{m,n} = 
 \frac{1}{2\pi\tau}  \int_{-\infty}^\infty e^{-\frac{y_1^2}{4\tau}}y_1^m \int_{0}^{-x_2+\frac{g^{(2)}}{2}y_1^2 + \frac{g^{(3)}}{6}y_1^3+ \frac{g^{(4)}}{24}y_1^4 +\cdots} \left(y_2^n - \frac{1}{4\tau} y_2^{n+2} +\frac{1}{32\tau^2} y_2^{n+4}+ \cdots \right) \ dy_2 dy_1 .
\end{align}
Assume $x_2 \sim O(\tau)$ and write $x_2 = \tau \tilde x_2$ with $\tilde x_2 \sim O(1)$, then the leading orders in $\zeta_{m,n}$ are computed:
\[\begin{cases}
\zeta_{0,0} = & \frac{-\tilde x_2}{\sqrt{\pi}}\tau^{1/2} + \frac{ g^{(2)}}{\sqrt{\pi}}  \tau^{1/2}+ \left( \frac{g^{(4)}}{2\sqrt{\pi}}+\frac{\tilde x_2^3}{12 \sqrt{\pi}}- \frac{g^{(2)}\tilde x_2}{4\sqrt{\pi}}\right) \tau^{3/2} +o(\tau^{3/2}) \\
\zeta_{1,0} = & \frac{2g^{(3)}}{\sqrt{\pi}} \tau^{3/2}+o(\tau^{3/2}) \\ 
\zeta_{0,1} = & \left(\frac{\tilde x_2^2}{2\sqrt{\pi}}- \frac{\tilde x_2 g^{(2)}}{\sqrt{\pi}} +\frac{3(g^{(2)})^2}{2\sqrt{\pi}} \right)\tau^{3/2} +o(\tau^{3/2}) \\
\zeta_{2,0}  = &\left(\frac{-2 \tilde x_2}{\sqrt{\pi}}+ \frac{6g^{(2)}}{\sqrt{\pi}} \right)\tau^{3/2} + o(\tau^{3/2}) \\
\zeta_{0,2} = &  o(\tau^{3/2})\\
\zeta_{1,1} = & o(\tau^{3/2})
\end{cases}\]
Therefore, 
\[II_2 = \psi_{0,0} (\frac{-\tilde x_2}{\sqrt{\pi}} + \frac{ g^{(2)}}{\sqrt{\pi}})  \tau^{1/2} +O(\tau^{3/2}) .\]
Collecting all leading terms in $I_1$, $II_1$, and $II_2$ yields
\[\phi^k(0,\tilde x_2) = -\frac{2\psi_{0,1}}{\sqrt{\pi}} \tau^{1/2} + \psi_{0,0} (\frac{-\tilde x_2}{\sqrt{\pi}} + \frac{ g^{(2)}}{\sqrt{\pi}})  \tau^{1/2} + O(\tau^{3/2}).\]
Set $\phi^k(0,\tilde x_2) = 0$, we then have 
\begin{align}
\psi_{0,0} \tilde x_2 = \psi_{0,0} g^{(2)} -2\psi_{0,1}. \label{eq:motionlaw}
\end{align}
Note that above analysis is only valid when the origin $(0,0)$ is not on the zero level set of $\psi$ because $\psi$ is not differentiable at points where $\psi = 0$. If $\psi$ is constant along the normal direction of the front (\ie, $\psi_{0,1} = 0$), the leading order behavior of the front motion is $g^{(2)}$ which denotes mean curvature. This reduces to the formal analysis to the original MBO method \cite{pierre92}. 

Now, we consider the case where the front is perturbed from the zero level set of $\psi$ and assume $\psi_{0,0} \sim O(\varepsilon)$ and $\psi_{0,1} \sim O(1)$. Then, the value of $\tilde x_2$ (the velocity of the front) is dominated by $-2\psi_{0,1}$. Therefore, every point on the front moves to the direction which decreases the value of $\psi$ (see, for example, a diagram in Figure~\ref{fig:velocity_diag}). Since $\psi = |d|^{p/2}$, in the leading order, the front evolves to $\psi = 0$. 

\begin{figure}[ht!]
\centering
\includegraphics[width = 0.5\textwidth, clip, trim = 4cm 5cm 2cm 4cm]{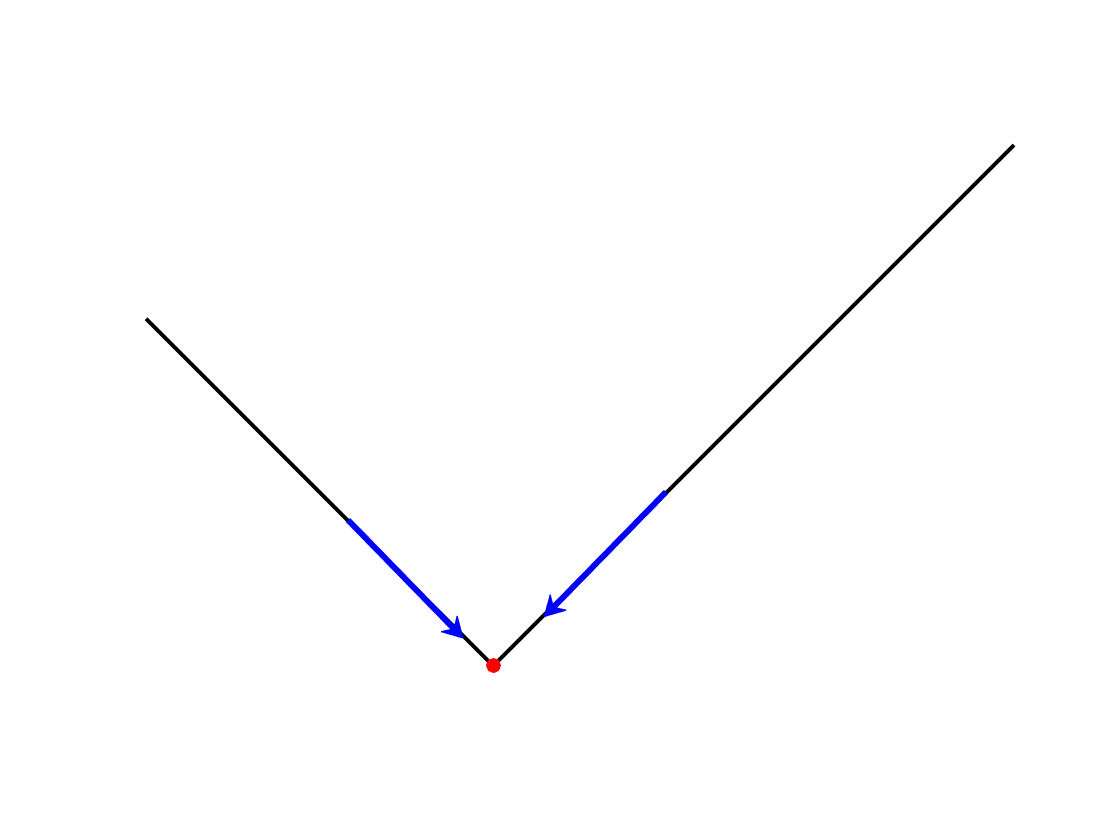}
\caption{A diagram for the direction of motion of the front. See Section~\ref{sec:connection}.} \label{fig:velocity_diag}
\end{figure}

\begin{rem}
We remark here that the above analysis only gives the formal leading order velocity of motion of the front. In \eqref{eq:motionlaw}, there is still one term regarding the mean curvature which may play an important role when $\psi_{0,0}$ is not small (\ie, the front is far away from the point cloud). This observation gives several suggestions on the setting of the algorithms: 1.) The initial guess should contain the point cloud and 2.) the initial choice $\tau$ should not be very large, otherwise, the front evolves to the global minimizer $\emptyset$ except the local minimizer.
\end{rem}

\begin{rem}
All above computation can be done for Algorithm~\ref{a:MBO} similarly.
\end{rem}

\section{Accelerations on Algorithms~\ref{a:MBO} and \ref{a:MBO2}.} \label{sec:accele}

As we mentioned in Section~\ref{sec:connection}, both Algorithm~\ref{a:MBO} and Algorithm~\ref{a:MBO2} are related to interface dynamics and $\tau$ is a free parameter which can be interpreted as the time step. Because the method is unconditional stable as shown in Section~\ref{sec:algorithm2}, the algorithms work for any large $\tau$. However, $\tau$ is also related to the approximation of the energy functional. We do need to use small $\tau$ to obtain a good approximation. In the discrete case, when the domain $\Omega$ is discretized by a uniform mesh,  the interface may stuck with a small $\tau$, where the indicator functions between two iterations are exactly same.  The whole algorithm is then stuck.

Based on above observations, we propose to accelerate Algorithms~\ref{a:MBO} and \ref{a:MBO2} with building a sequence $\tau_1 > \tau_2 >\tau_3>\cdots>\tau_N$. We start with solving Algorithms~\ref{a:MBO} and \ref{a:MBO2} with $\tau_1$ until convergence and then we use the obtained solution to build an initial guess of the Algorithms with $\tau_2$. This is repeated until we have obtained two exactly same solutions with $\tau_\ell$ and $\tau_{\ell+1}$. The algorithm is summarized into Algorithm~\ref{a:MBO3}.

\begin{algorithm}[ht!]
\DontPrintSemicolon
 \KwIn{$\Omega$: computational domain; $d(\bx)$: distance function to the point cloud; a sequence $\tau_1>\tau_2>\tau_3>\cdots >\tau_N$; $\tau =\tau_1$; $s=1$; and $u^0 \in \mathcal{B}$.}
 \KwOut{$u^\star \in \mathcal{B}$.}
 \While{not converged}{   
Run Algorithm~\ref{a:MBO} or \ref{a:MBO2} to obtain the stationary solution $u^\star$ at current $\tau$. \\
$s = s+1$ \\ 
$\tau = \tau_s$}
\caption{The accelerated version of Algorithm~\ref{a:MBO} or \ref{a:MBO2}. }
\label{a:MBO3}
\end{algorithm}
\begin{rem}
We remark here that $\tau_1, \tau_2, \cdots$ are not required to be chosen specifically or initially. One could simply choose $\tau_{s+1} = \frac{\tau_{s}}{2}$ each time after the algorithms with $\tau_s$ converges.
\end{rem}

The advantages of Algorithm~\ref{a:MBO3} can be understood as follows:
\begin{enumerate}
\item It avoids the solution to be stuck. 
\item The initial large $\tau$ makes the solution converge faster.
\item Algorithm~\ref{a:MBO3} converges at the a relatively small $\tau$, which guarantees the accuracy of the solution. 
\end{enumerate}

\section{Numerical implementation and experiments} \label{sec:num}
\subsection{Numerical implementation}\label{sec:implementation}
In this section, we discuss the implementation of the whole algorithm. When the point cloud is fixed, $d(\bx)$ is fixed and only needs to be evaluated once. That is, from the point cloud $\mathcal C$, we expect to solve 
\begin{equation} \label{eq:distance}
\begin{cases}
|\nabla d| = 1, \ \ \bx \in \Omega, \\
d(\bx) = 0,  \ \  \bx \in \mathcal C.
\end{cases}\end{equation}
Following \cite{he2019fast}, we simply choose a first order Lax--Friedrich scheme to discretize the relaxed dynamical equation of \eqref{eq:distance}:
\begin{equation}
d_{i,j}^{n+1} = \frac{1}{2} \left(1-|\nabla d_{i,j}^n|+\frac{d_{i+1,j}^n+d_{i-1,j}^n}{2}+\frac{d_{i,j+1}^n+d_{i,j-1}^n}{2}\right)
\end{equation}
and use a fast sweeping method \cite{Kao_2004} to solve it with a linear complexity.

Because the point cloud is discrete, if one plots the $\varepsilon$-level set (relatively small $\varepsilon$) of the computed $d(\bx)$, the contours are small circles around each points in the cloud (see Figure~\ref{fig:distance_diag}). Figure~\ref{fig:distance_diag} tells that a distance function computed from discrete point cloud does not imply the surface reconstruction from point cloud, especially in high dimensional cases where the ordering of points are much more complicated than 2-dimensional cases. That's the reason that we consider minimizing the objective function to obtain the reconstructed surface instead of direct interpolation from the computed distance function.

\begin{figure}[ht!]
\centering
\includegraphics[width = 0.3 \textwidth, clip, trim = 13cm 8cm 11cm 7cm]{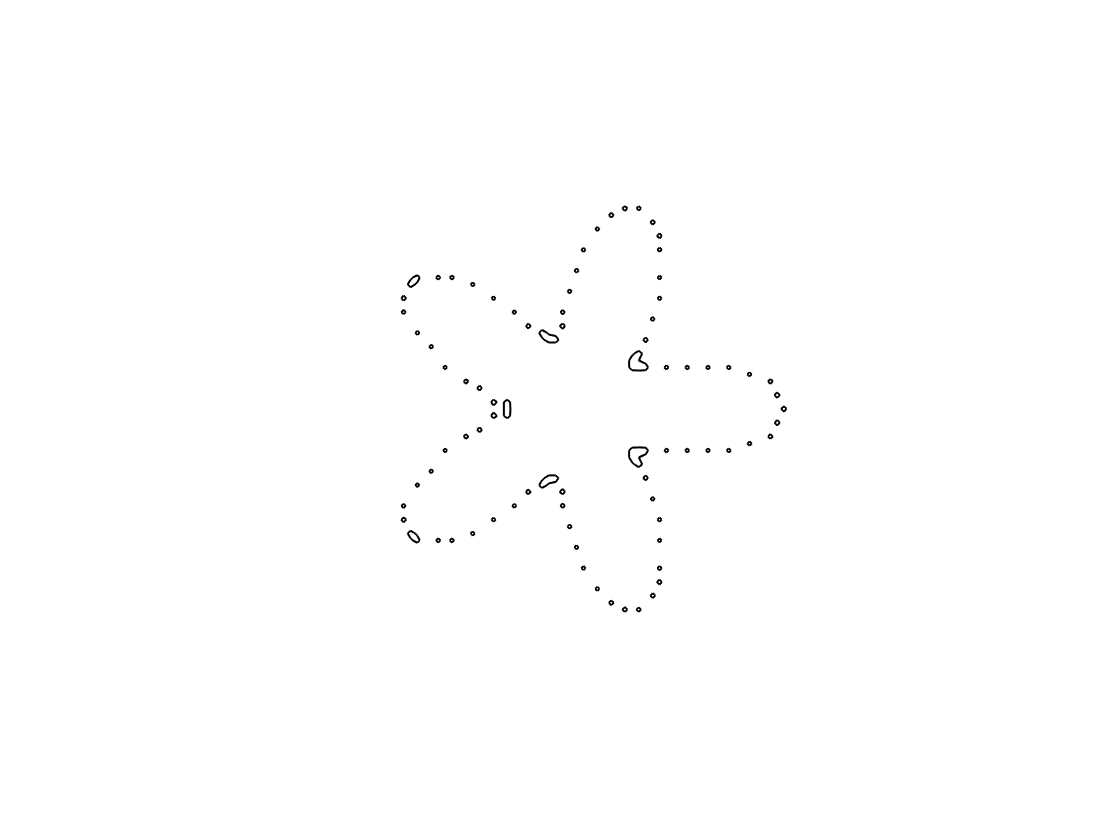}
\caption{A diagram for a $\epsilon$-level set of the distance function computed from \eqref{eq:distance}.  See Section~\ref{sec:implementation}.}\label{fig:distance_diag}
\end{figure}

Besides above, we only need to compute convolutions in each iteration.  They are efficiently computed using fast Fourier transform (FFT) based on a uniformly discretized computational domain. 

\subsection{Numerical experiments.} \label{sec:num_experiments}
In this section, we provide a variety of numerical experiments to show the performance of the proposed algorithm. We implemented all algorithms in MATLAB.  All reported results were obtained on a laptop with a 2.7GHz Intel Core i5 processor and 8GB of RAM.  In all experiments, the domain $\Omega = [-\pi,\pi]^n$ is discretized with uniform meshes.  If there is no other statement, the CPU time we report is the total CPU time for iterations after $d(\bx)$ is computed and $p=2$. For the convenience, we denote Algorithm~\ref{a:MBO3}\_\ref{a:MBO} (or \ref{a:MBO3}\_\ref{a:MBO2}) as Algorithm~\ref{a:MBO3} cooperated with Algorithm~\ref{a:MBO} (or \ref{a:MBO2}).

\subsubsection{Comparisons between Algorithms~\ref{a:MBO}-~\ref{a:MBO3}.} \label{sec:propertiescheck}
In this section, we perform several experiments to show the comparisons among Algorithms~\ref{a:MBO}-\ref{a:MBO3}. We use two point clouds in this experiment as displayed in Figure~\ref{fig:ex1}. The 2-dimensional point cloud is generated using $N=200$ uniform grids $\theta_i$ ($i\in [N]$) in $[0,2\pi]$:
\[\begin{cases}
x_i = r_i\cos(\theta_i),\\
y_i = r_i\sin(\theta_i)
\end{cases}\]
where $r_i= 1+0.5\cos(5(\theta_i-\pi/2))$. The 3-dimensional point cloud is generated from $N = 2000$ random points $(u_i,v_i)$, $i \in [N]$ (i.i.d. from uniform distribution) in $[0,2\pi]^2$ by:
\[\begin{cases}
x_i = (1+0.5\cos(u_i))\cos(v_i), \\
y_i = (1+0.5\cos(u_i))\sin(v_i), \\
z_i = 0.5\sin(u_i).
\end{cases}
\]

\begin{figure}[ht!]
\centering
\includegraphics[width = 0.3\textwidth, clip, trim = 2cm 2cm 2cm 1cm]{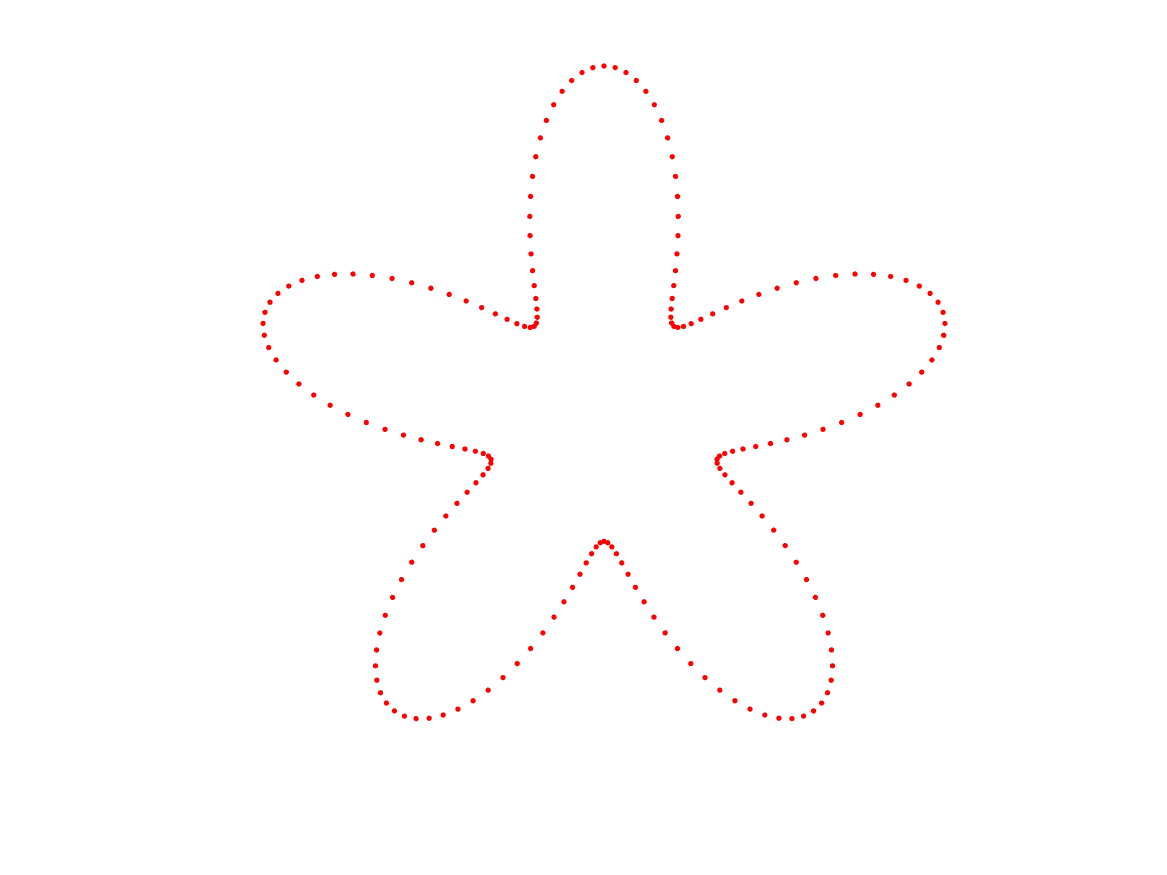}\qquad \qquad  
\includegraphics[width = 0.3\textwidth, clip, trim = 2cm 2cm 2cm 1cm]{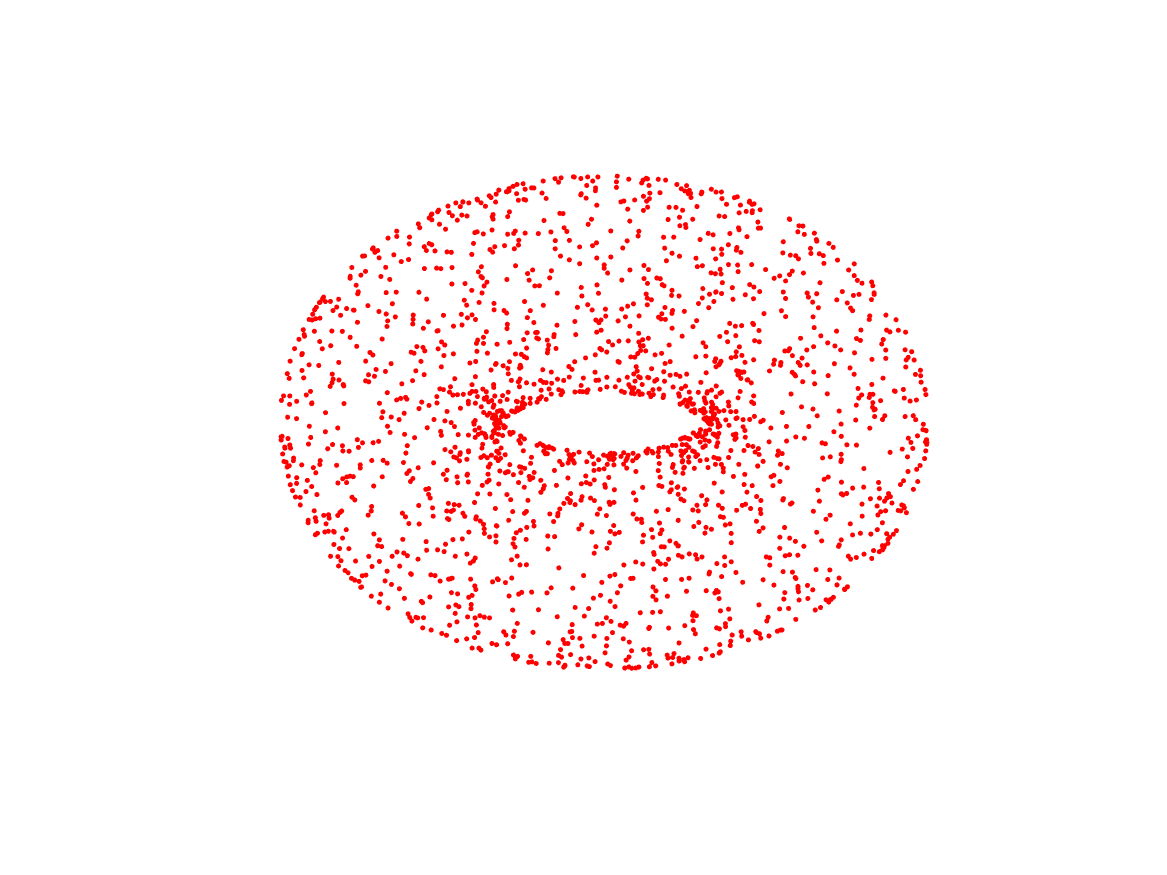}
\caption{{\bf Left:} Point cloud of five folds. {\bf Right:} Point cloud of a torus. See Section~\ref{sec:propertiescheck}.} \label{fig:ex1}
\end{figure}

Figure~\ref{fig:ex1_comparison1} displays results from Algorithms~\ref{a:MBO} and \ref{a:MBO2} with different values of $\tau$ for the 2-dimensional point cloud. Figure~\ref{fig:ex1_adaptive} displays the results obtained using Algorithms~\ref{a:MBO3}\_\ref{a:MBO} and \ref{a:MBO3}\_\ref{a:MBO2}. In Figure~\ref{fig:ex1_comparison1}, for a small $\tau = 0.0025$, the solution is stuck at an incorrect solution which is just because of the spatial discretization. In Algorithm~\ref{a:MBO3}, this is easily avoided by using the adaptive in time strategy. In all experiments, the computational domain $[-\pi,\pi]^2$ is discretized by $128^2$ uniform grids. 

\begin{figure}[ht!]
\begin{center}
\begin{tabular}{c|c|c|c|c|c}
& $\tau = 0.02$ & $\tau = 0.015$ & $\tau = 0.01$ & $\tau = 0.005$ & $\tau = 0.0025$   \\
 Algorithm~\ref{a:MBO}& $\includegraphics[width = 0.15\textwidth,clip, trim = 6.5cm 4cm 4cm 3.5cm ]{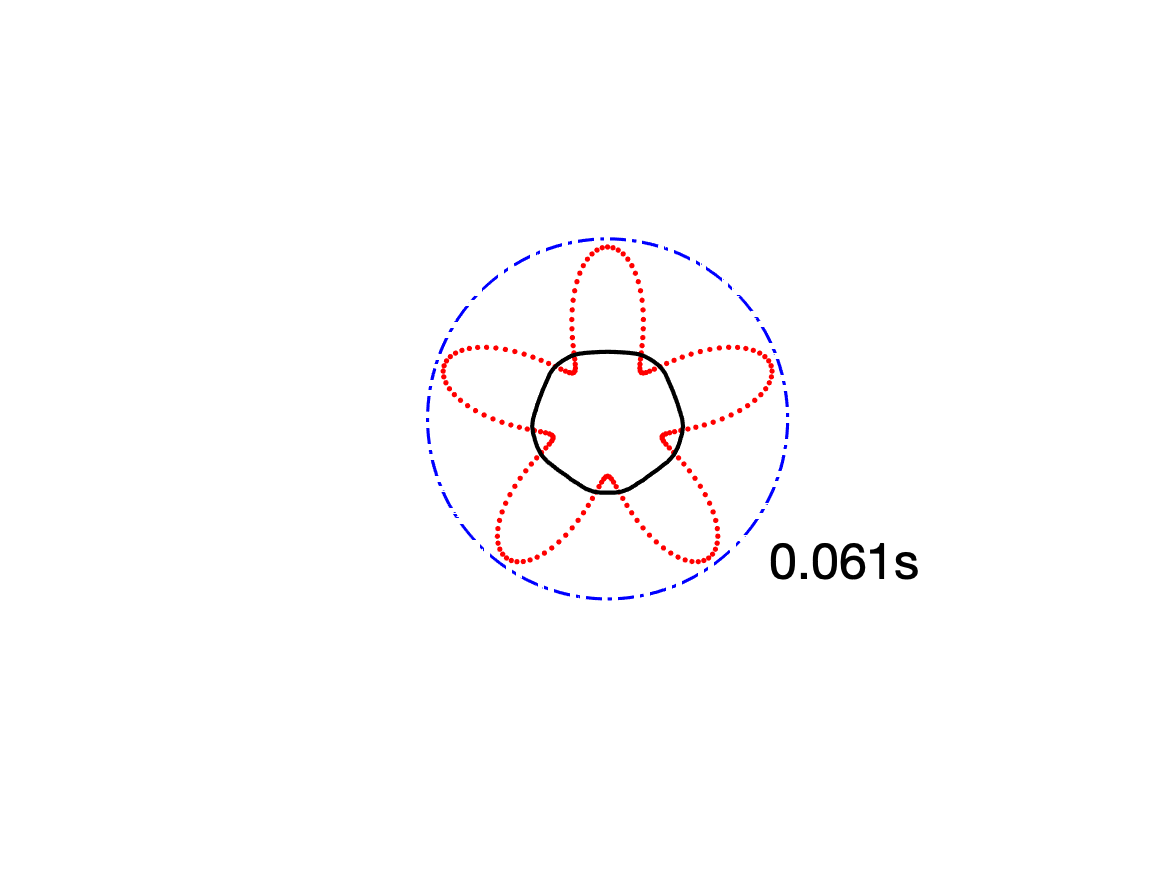}$ & $\includegraphics[width = 0.15\textwidth,clip, trim = 6.5cm 4cm 4cm 3.5cm ]{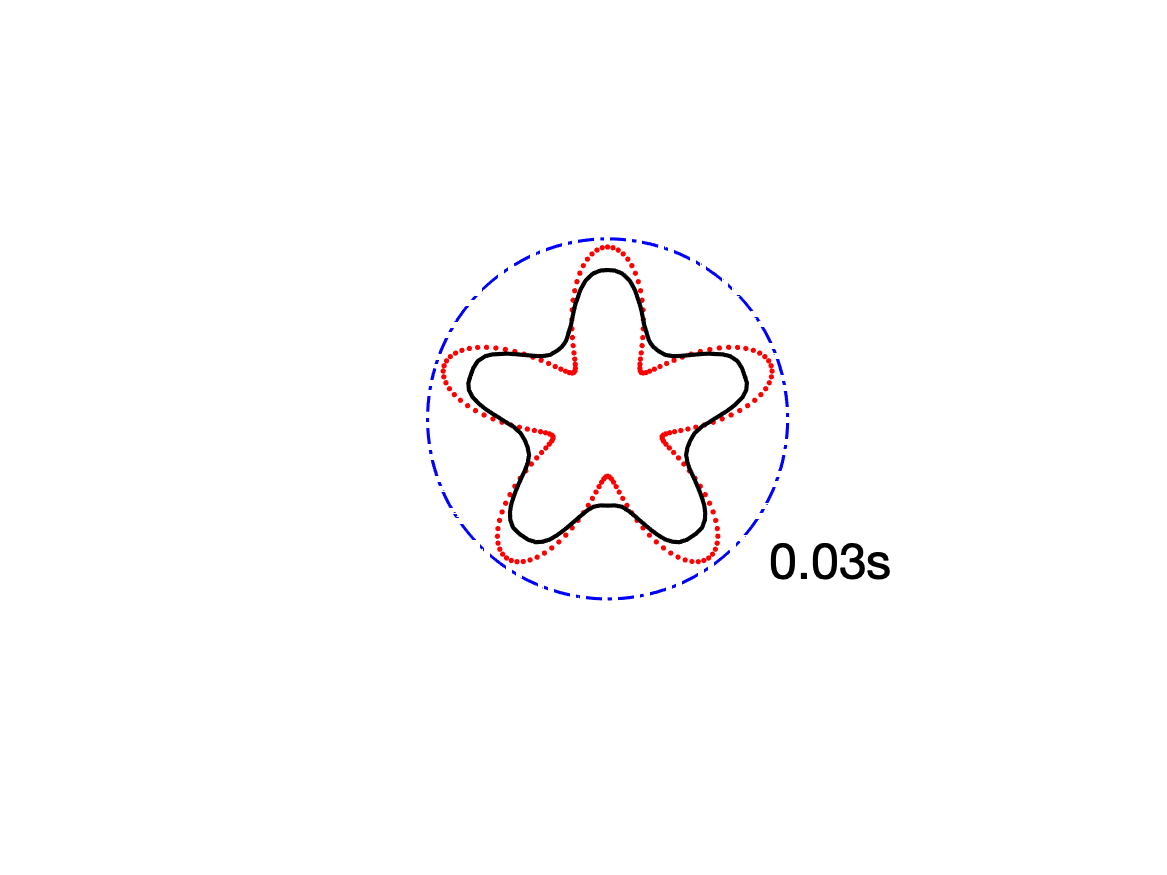}$ & $\includegraphics[width = 0.15\textwidth,clip, trim = 6.5cm 4cm 4cm 3.5cm ]{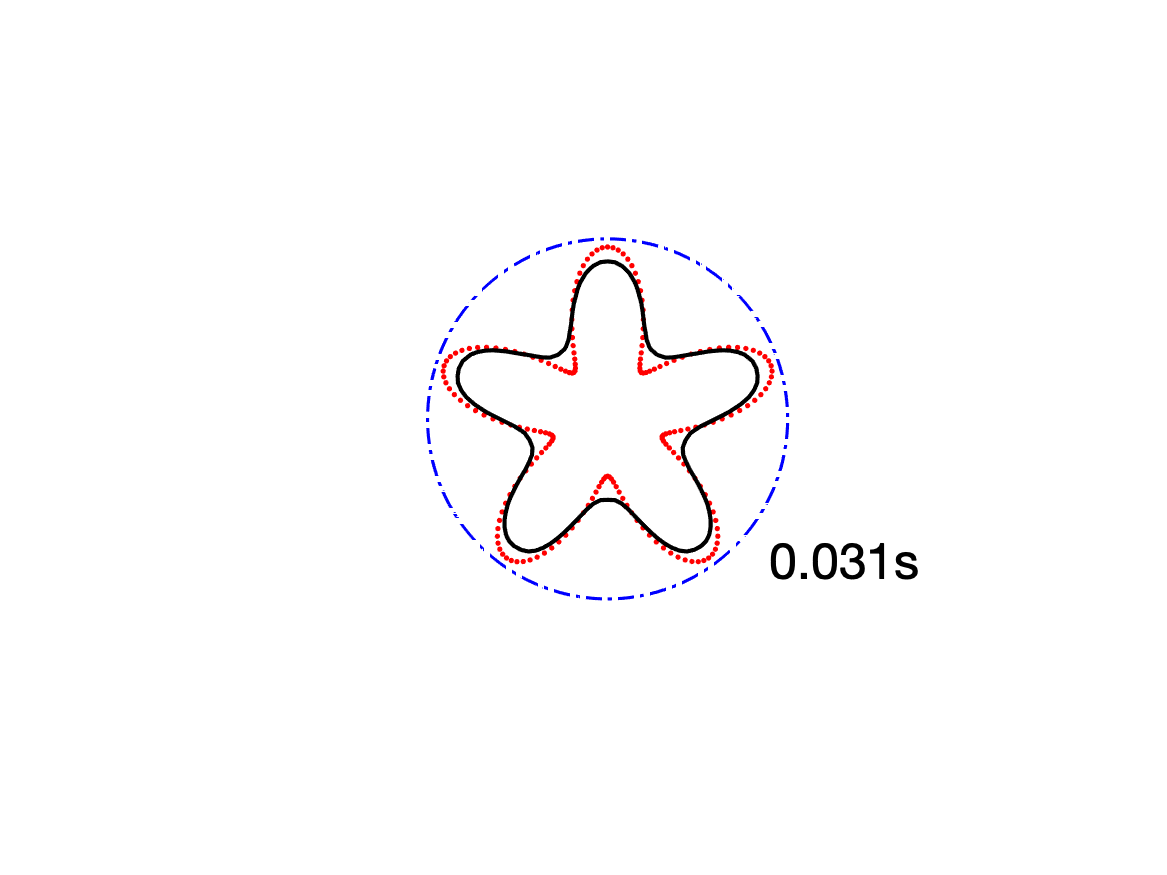}$ & $\includegraphics[width = 0.15\textwidth,clip, trim = 6.5cm 4cm 4cm 3.5cm ]{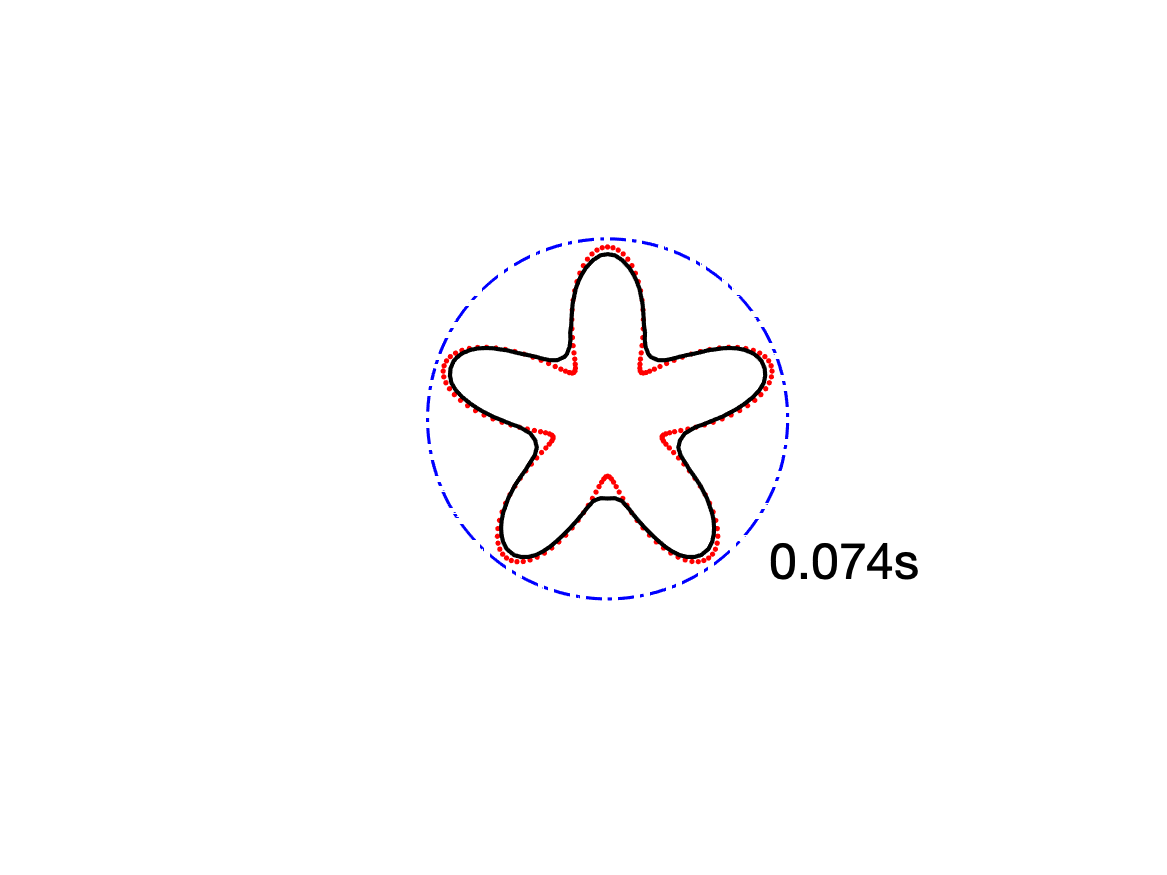}$ & $\includegraphics[width = 0.15\textwidth,clip, trim = 6.5cm 4cm 4cm 3.5cm ]{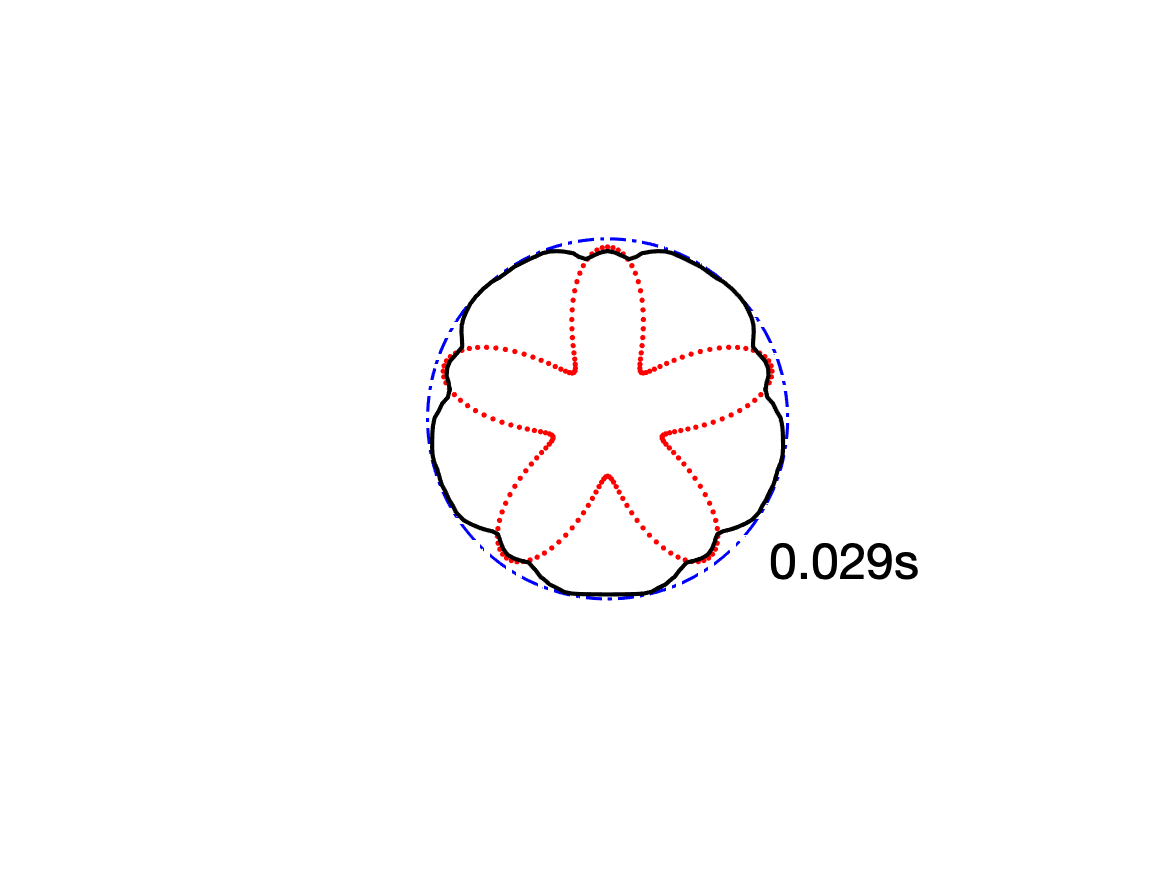}$ \\
Algorithm~\ref{a:MBO2}& $\includegraphics[width = 0.15\textwidth,clip, trim = 6.5cm 4cm 4cm 3.5cm ]{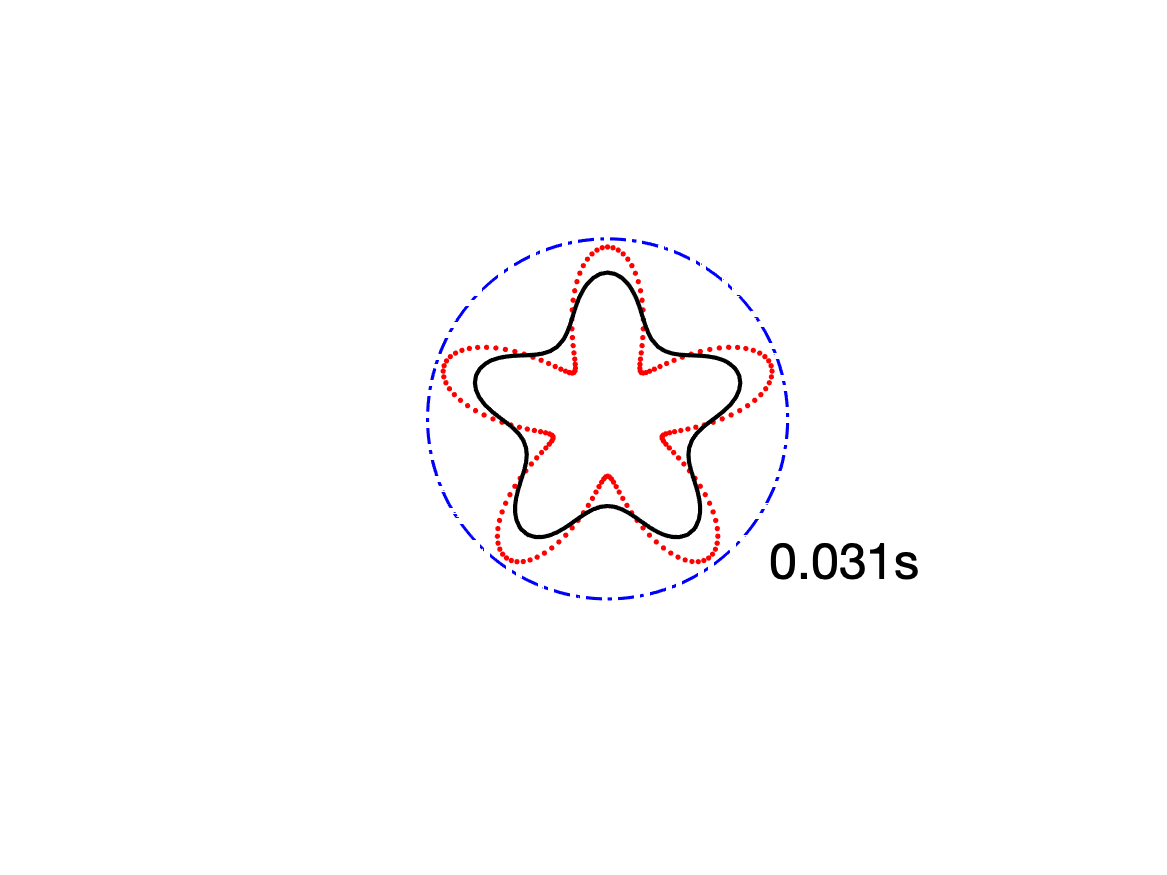}$ &  $\includegraphics[width = 0.15\textwidth,clip, trim = 6.5cm 4cm 4cm 3.5cm ]{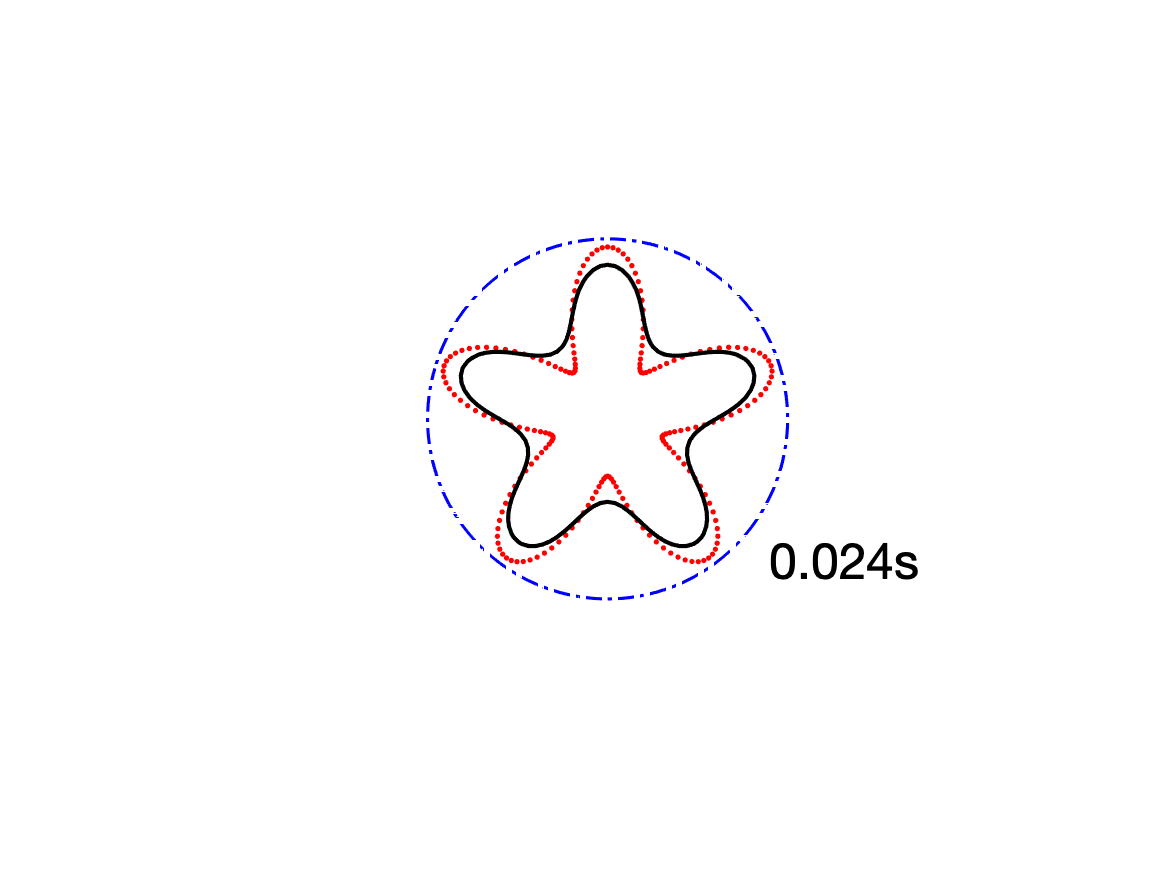}$  & $\includegraphics[width = 0.15\textwidth,clip, trim = 6.5cm 4cm 4cm 3.5cm ]{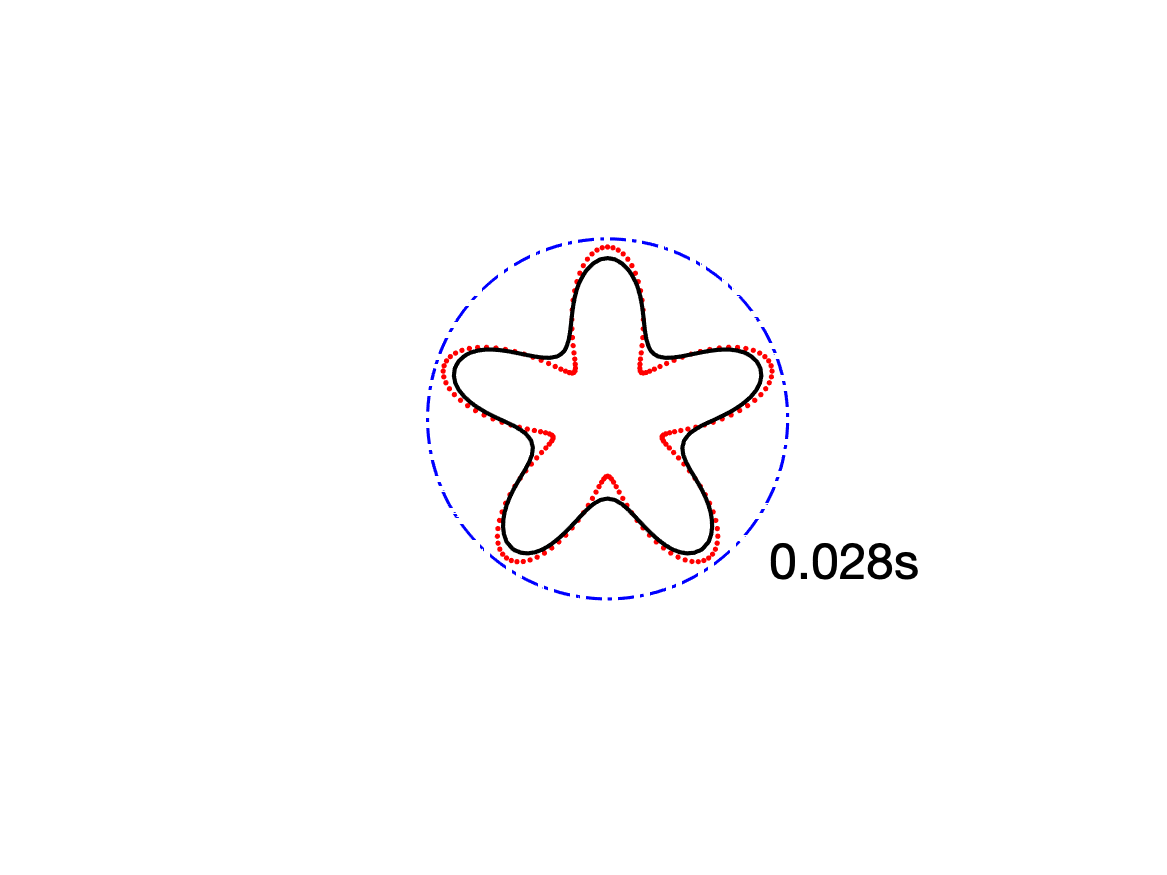}$ & $\includegraphics[width = 0.15\textwidth,clip, trim = 6.5cm 4cm 4cm 3.5cm ]{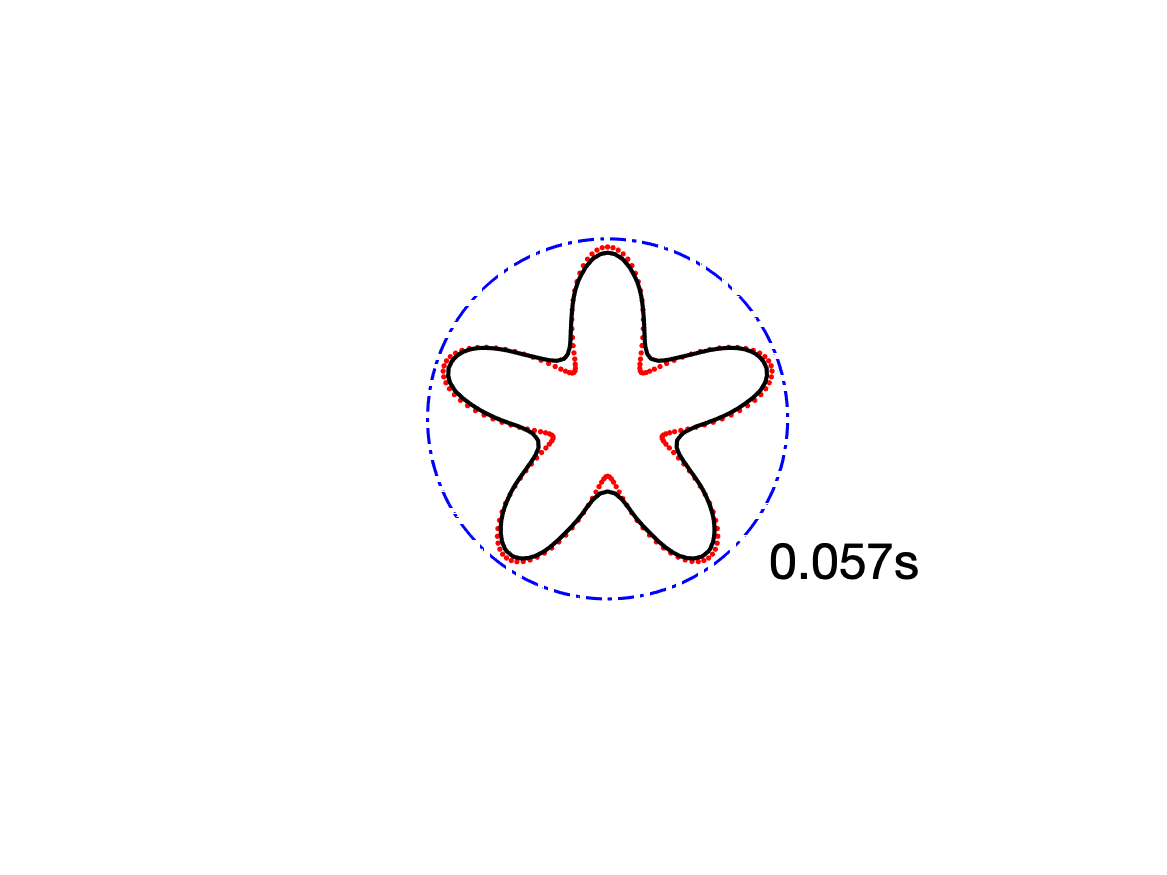}$  & $\includegraphics[width = 0.15\textwidth,clip, trim = 6.5cm 4cm 4cm 3.5cm ]{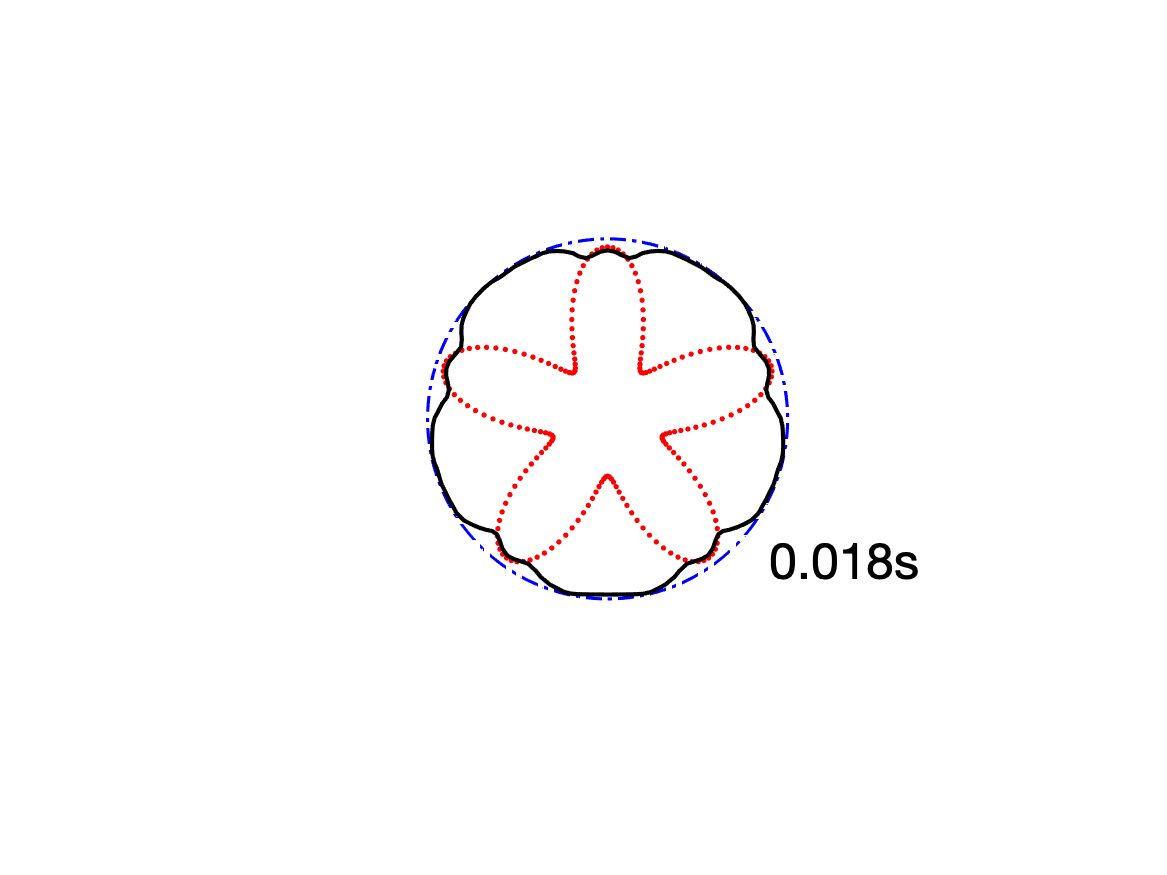}$ 
\end{tabular}
\end{center}
\caption{{\bf First row:} Computed results from Algorithm~\ref{a:MBO} with different $\tau$. {\bf Second row:} Computed results from Algorithm~\ref{a:MBO2} with different $\tau$. {\bf Blue curves:} Initial guess. {\bf Black curves:} Reconstructed surface.  See Section~\ref{sec:propertiescheck}.} \label{fig:ex1_comparison1}
\end{figure}

\begin{figure}[ht!]
\centering
 \includegraphics[width = 0.3\textwidth,clip, trim = 6.5cm 4cm 4cm 3.5cm ]{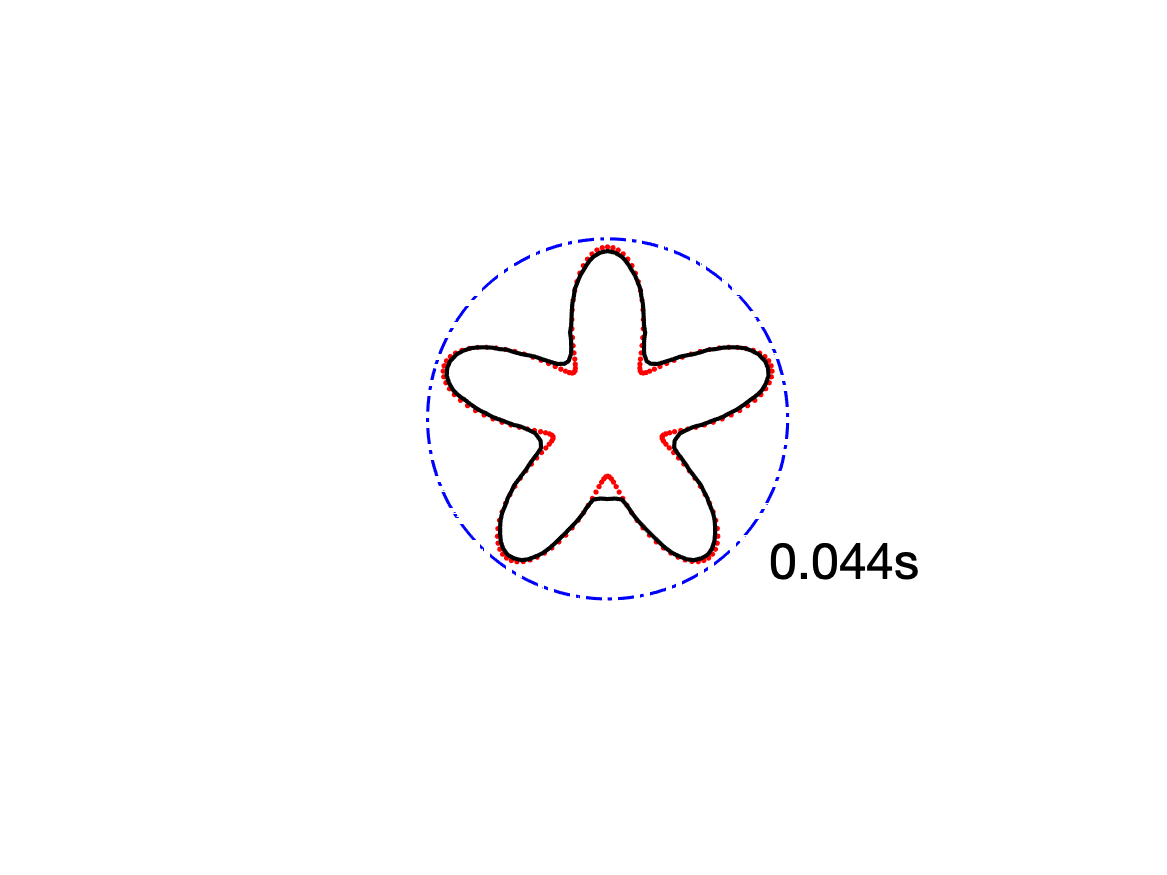} \qquad \qquad  
  \includegraphics[width = 0.3\textwidth,clip, trim = 6.5cm 4cm 4cm 3.5cm ]{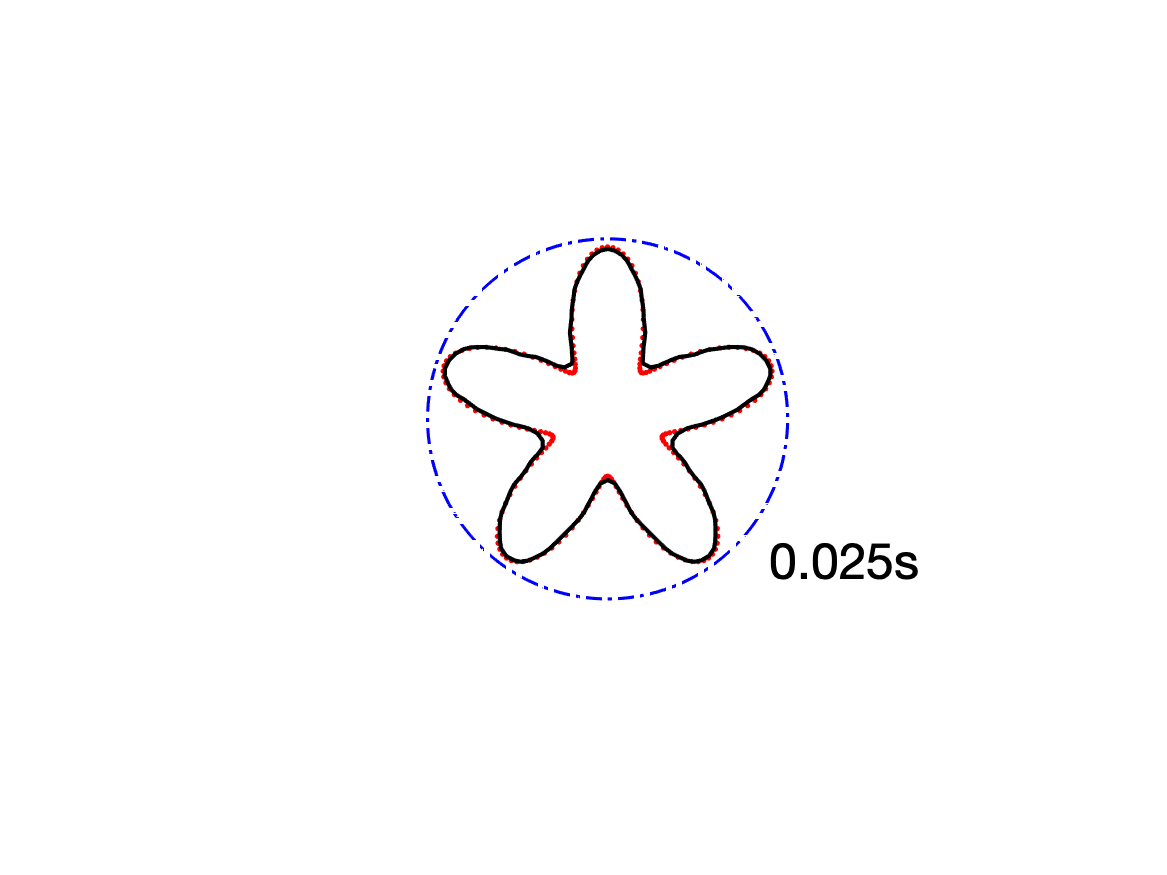}
  \caption{{\bf Left:} The result computed with Algorithm~\ref{a:MBO3}\_\ref{a:MBO} . {\bf Right:} The result computed with  Algorithm~\ref{a:MBO3}\_\ref{a:MBO2}. The $\tau_i$ used in both experiments are $0.02, 0.01, 0.005, 0.0025$, and  $0.00125$; respectively. See Section~\ref{sec:propertiescheck}.} \label{fig:ex1_adaptive}
\end{figure}

Figure~\ref{fig:ex1_comparison2} displays the 3-dimensional results obtained from Algorithms~\ref{a:MBO} and \ref{a:MBO2} with different choices of $\tau$. In all experiments, the initial guess is simply set to be the indicator function of a rectangular box:
\[\{(x,y,z) \colon |x|<1.6, \ |y|<1.6, \ \textrm{and}  \ |z|<0.6\}.\] The computational domain $[-\pi,\pi]^3$ is discretized with $128^3$ uniform grids.
It's easy to see that small $\tau$ makes both Algorithms stuck at some local minimizers of the discretized problem. However, smaller $\tau$ does give a closer surface to the point cloud. In Figure~\ref{fig:ex1_adaptive2}, we list the results obtained from Algorithm~\ref{a:MBO3}\_\ref{a:MBO} and \ref{a:MBO2}. Visually, the surfaces reconstructed from Algorithm~\ref{a:MBO3} are much better than the results in Figure~\ref{fig:ex1_comparison2}, respectively. From above two experiments, we observe that Algorithm~\ref{a:MBO3} is more accurate and robust than Algorithms~\ref{a:MBO} and \ref{a:MBO2}.

\begin{figure}[ht!]
\begin{center}
\begin{tabular}{c|c|c|c|c|c}
& $\tau = 0.02$ & $\tau = 0.018$ & $\tau = 0.016$ & $\tau = 0.01$ & $\tau = 0.005$   \\
 Algorithm~\ref{a:MBO}& $\includegraphics[width = 0.15\textwidth,clip, trim = 3.5cm 3cm 2cm 2.5cm]{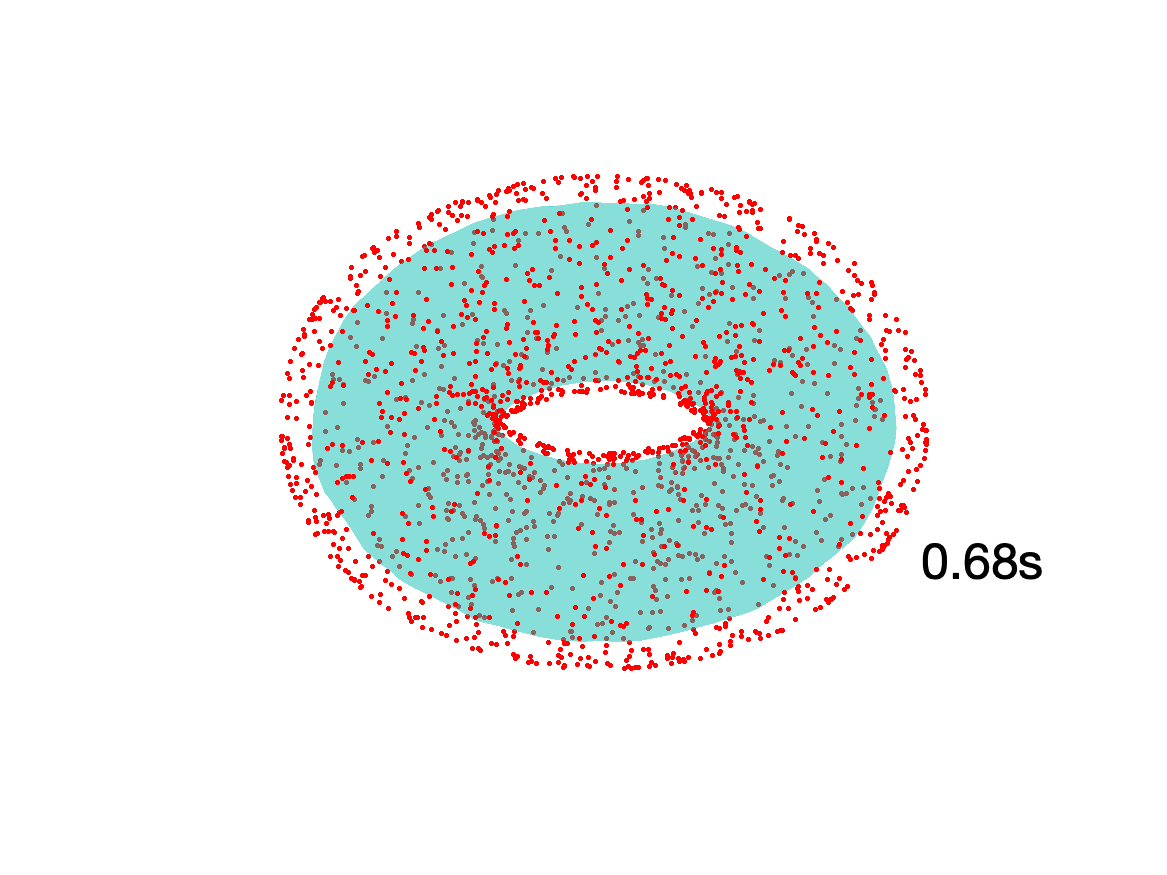}$ & $\includegraphics[width = 0.15\textwidth,clip, trim = 3.5cm 3cm 2cm 2.5cm]{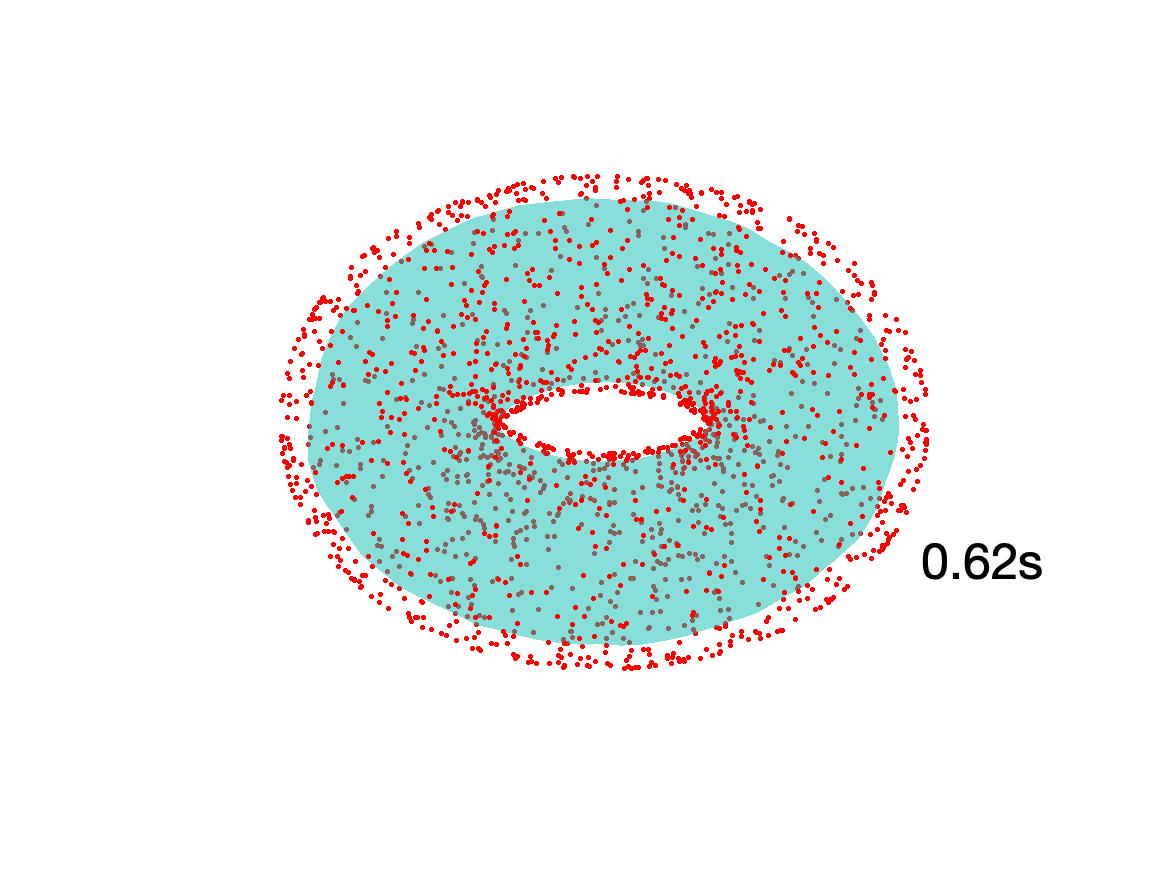}$ & $\includegraphics[width = 0.15\textwidth,clip, trim = 3.5cm 3cm 2cm 2.5cm ]{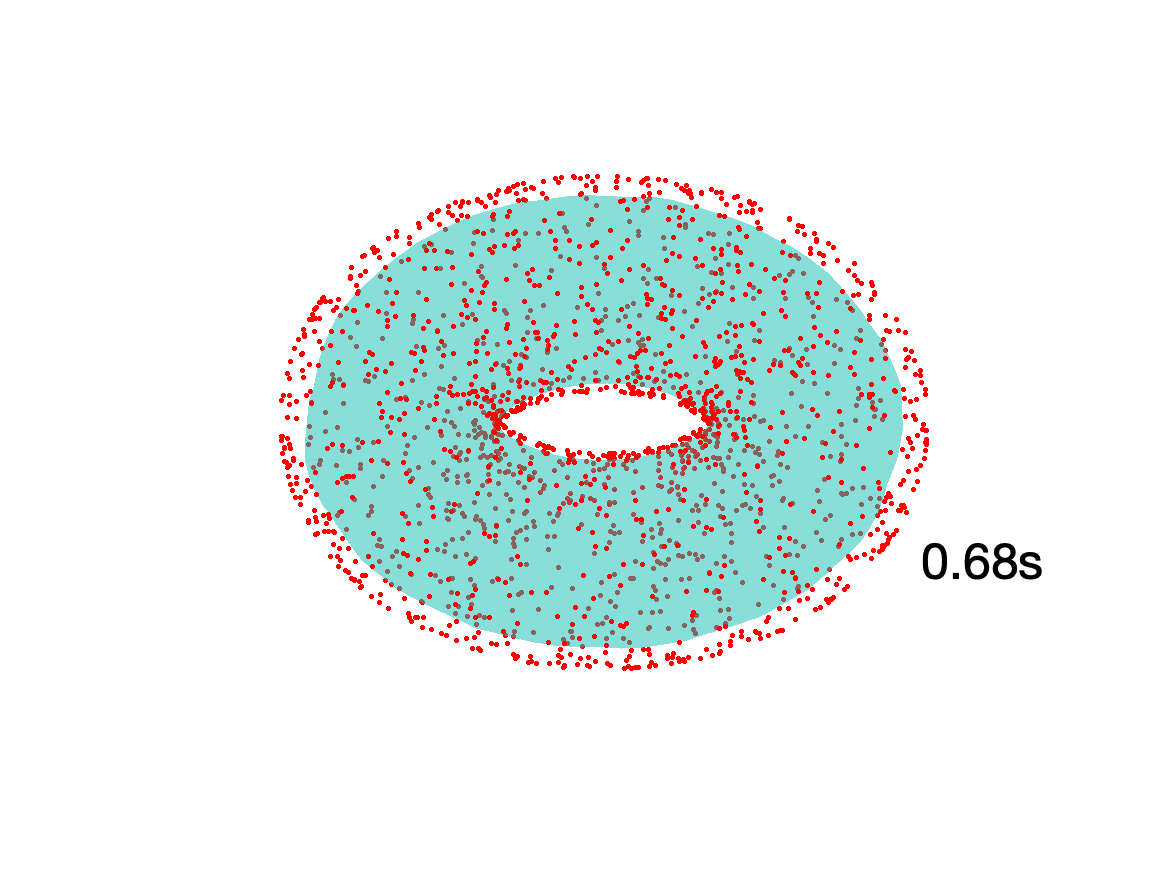}$ & $\includegraphics[width = 0.15\textwidth,clip, trim = 3.5cm 3cm 2cm 2.5cm]{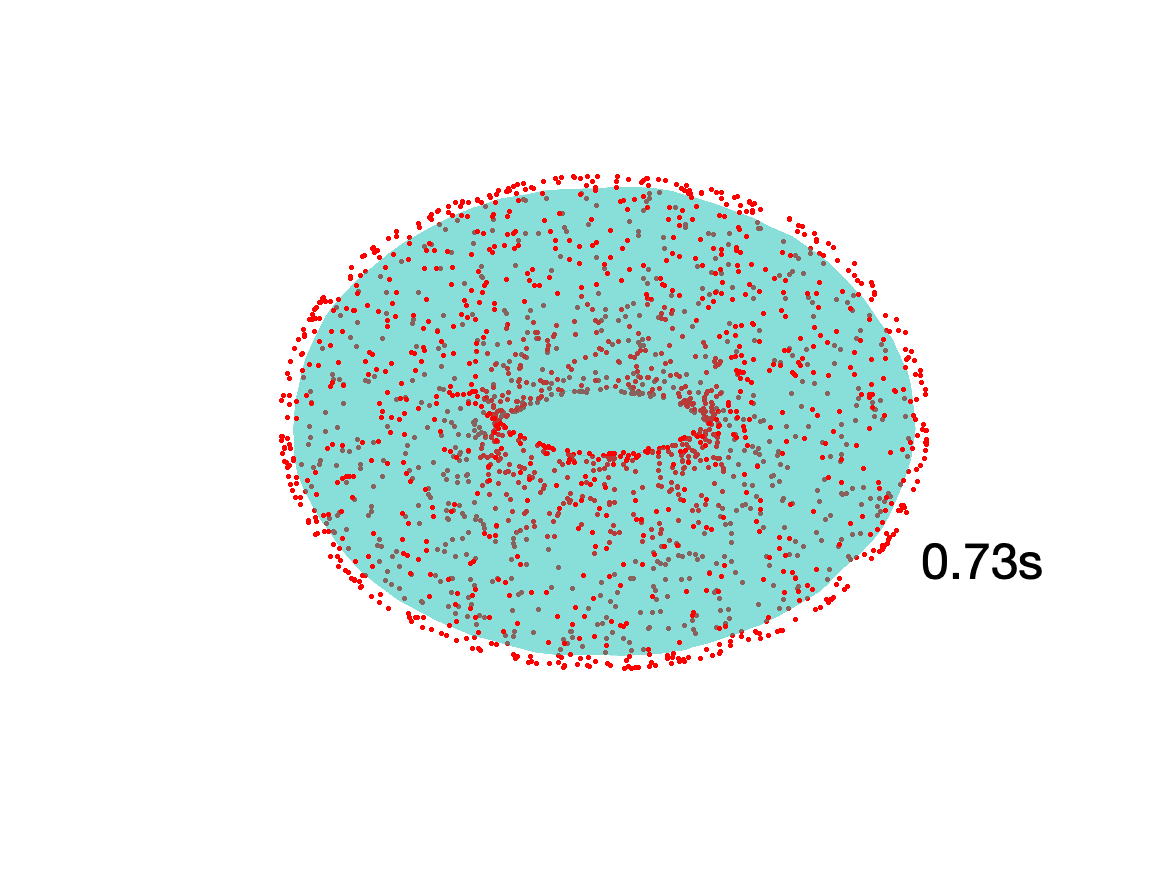}$ & $\includegraphics[width = 0.15\textwidth,clip, trim = 3.5cm 3cm 2cm 2.5cm]{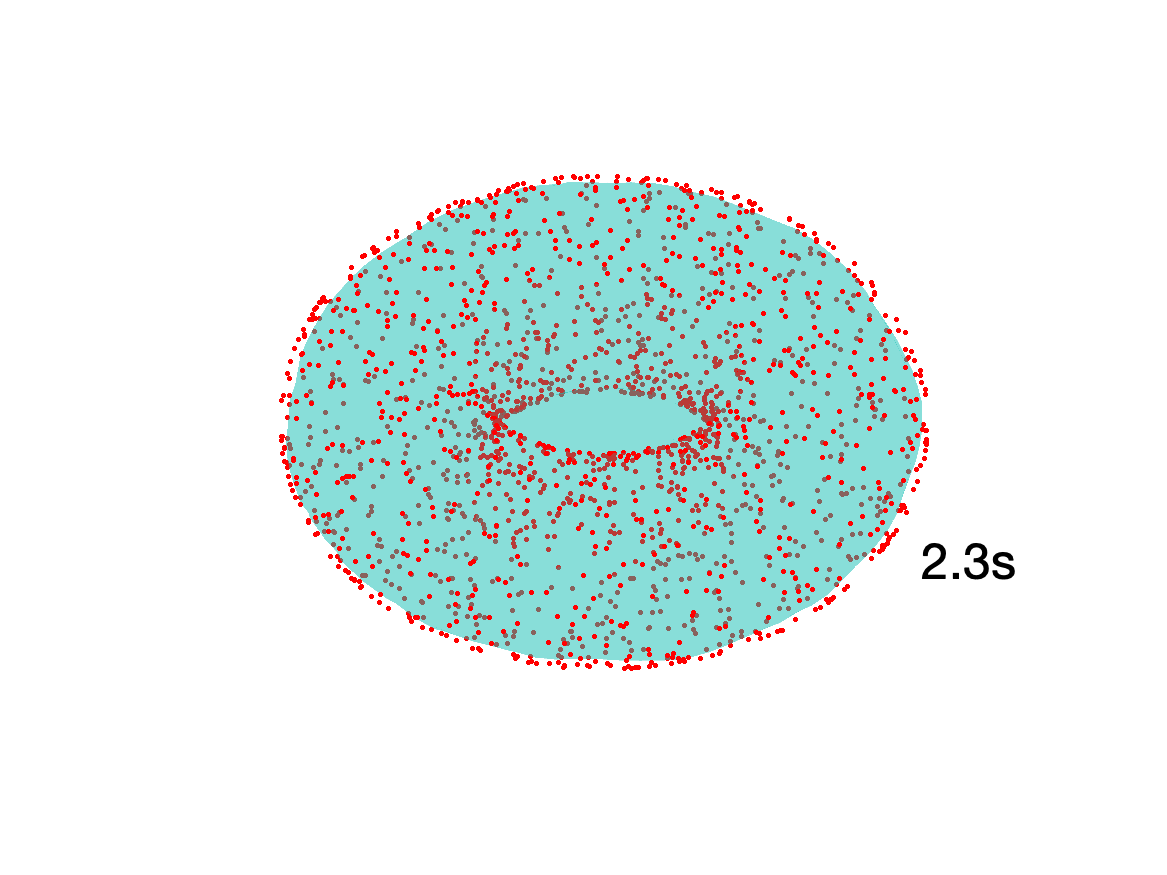}$ \\
Algorithm~\ref{a:MBO2}& $\includegraphics[width = 0.15\textwidth,clip, trim = 3.5cm 3cm 2cm 2.5cm]{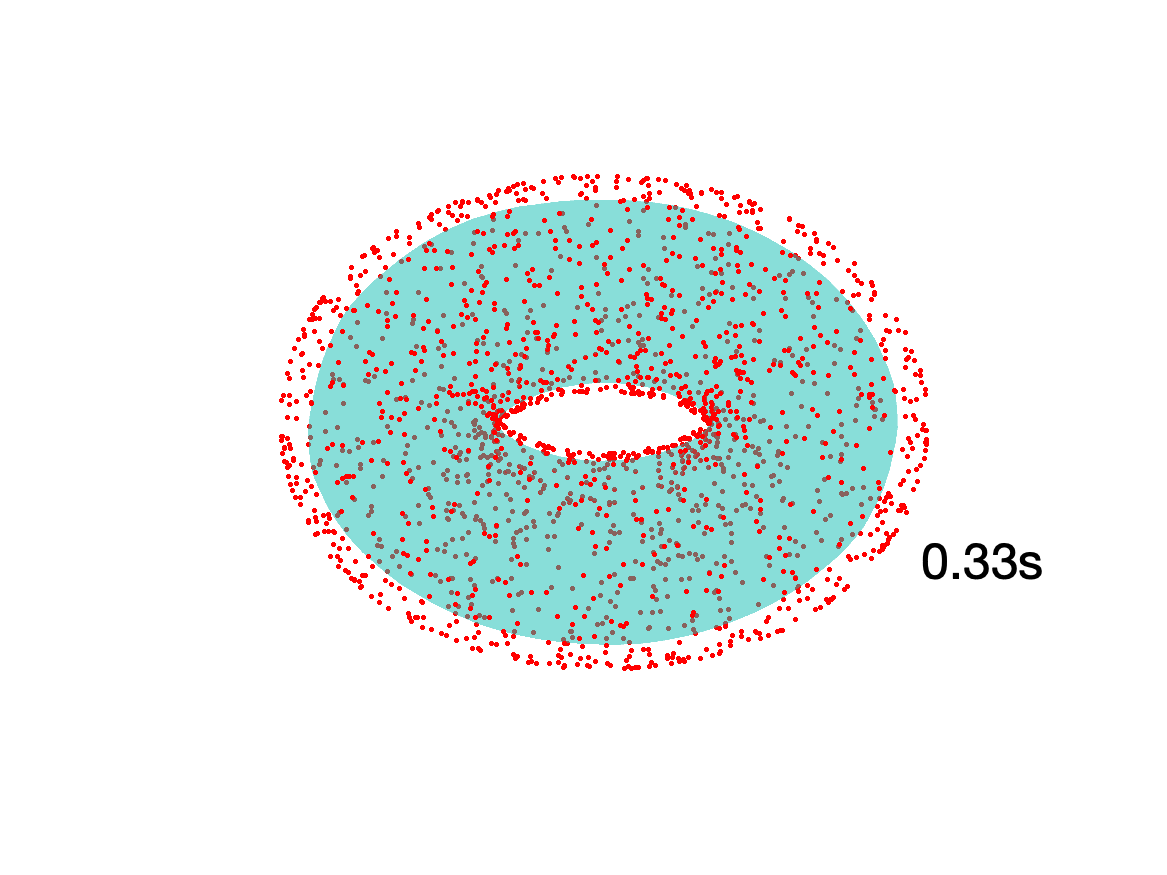}$ &  $\includegraphics[width = 0.15\textwidth,clip, trim = 3.5cm 3cm 2cm 2.5cm]{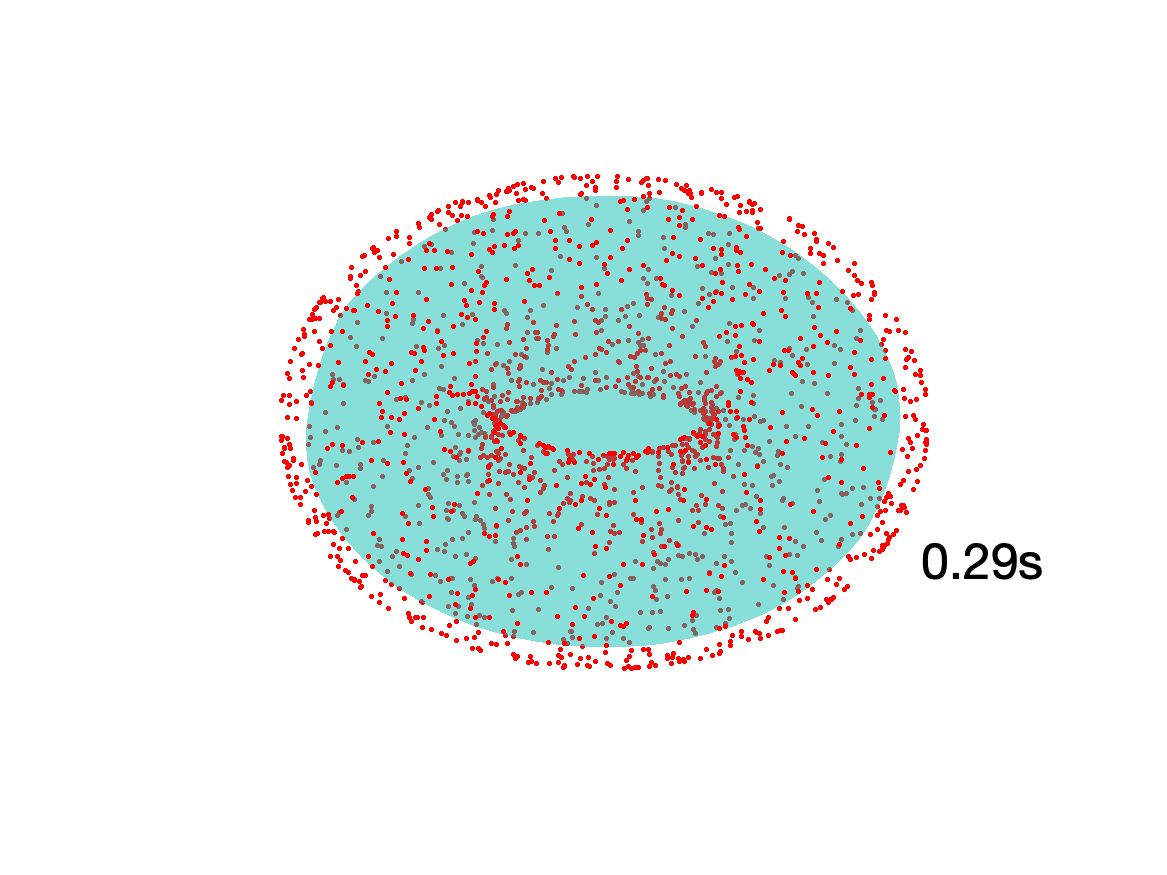}$  & $\includegraphics[width = 0.15\textwidth,clip, trim = 3.5cm 3cm 2cm 2.5cm]{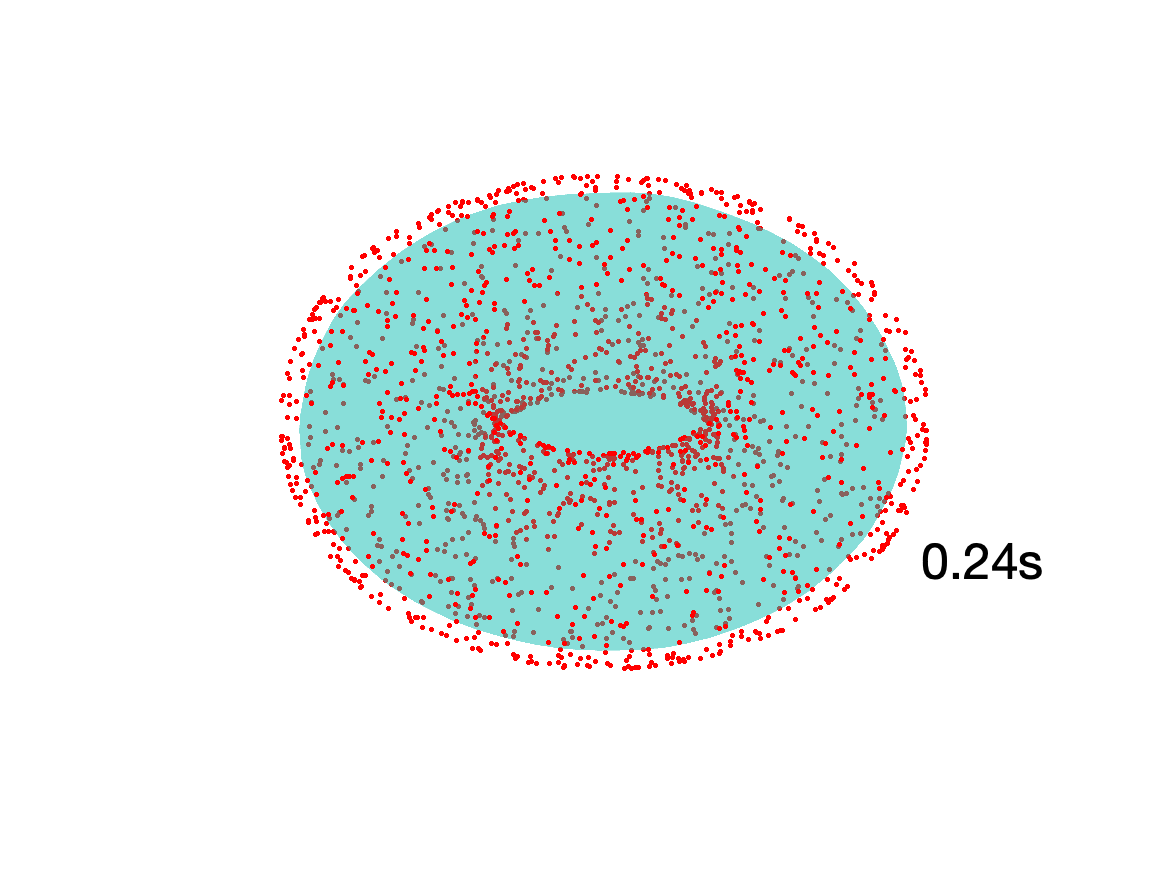}$ & $\includegraphics[width = 0.15\textwidth,clip, trim = 3.5cm 3cm 2cm 2.5cm ]{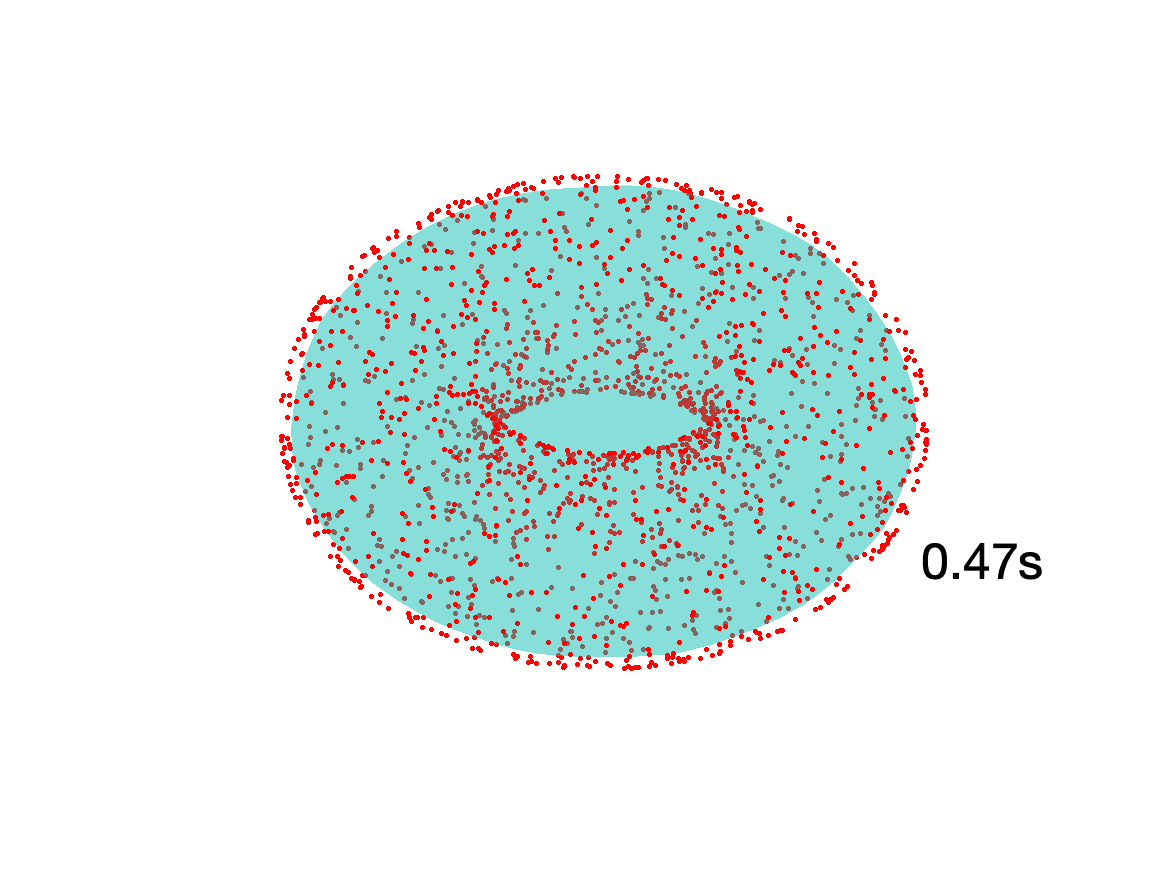}$  & $\includegraphics[width = 0.15\textwidth,clip, trim = 3.5cm 3cm 2cm 2.5cm ]{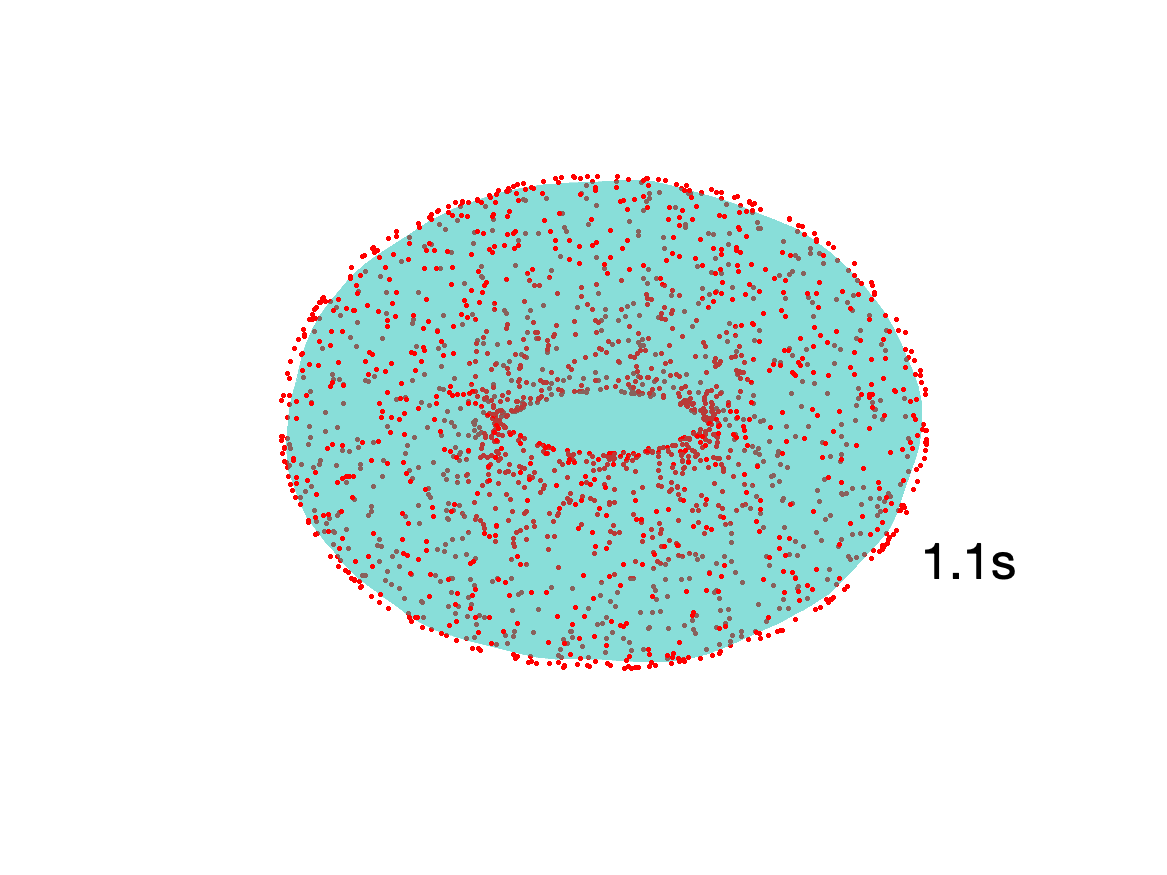}$ 
\end{tabular}
\end{center}
\caption{{\bf 1st row:} Computed results from Algorithm~\ref{a:MBO} with different $\tau$. {\bf 2nd row:} Computed results from Algorithm~\ref{a:MBO2} with different $\tau$. See Section~\ref{sec:propertiescheck}.} \label{fig:ex1_comparison2}
\end{figure}

\begin{figure}[ht!]
\centering
 \includegraphics[width = 0.3\textwidth,clip, trim = 3.5cm 3cm 2cm 2.5cm ]{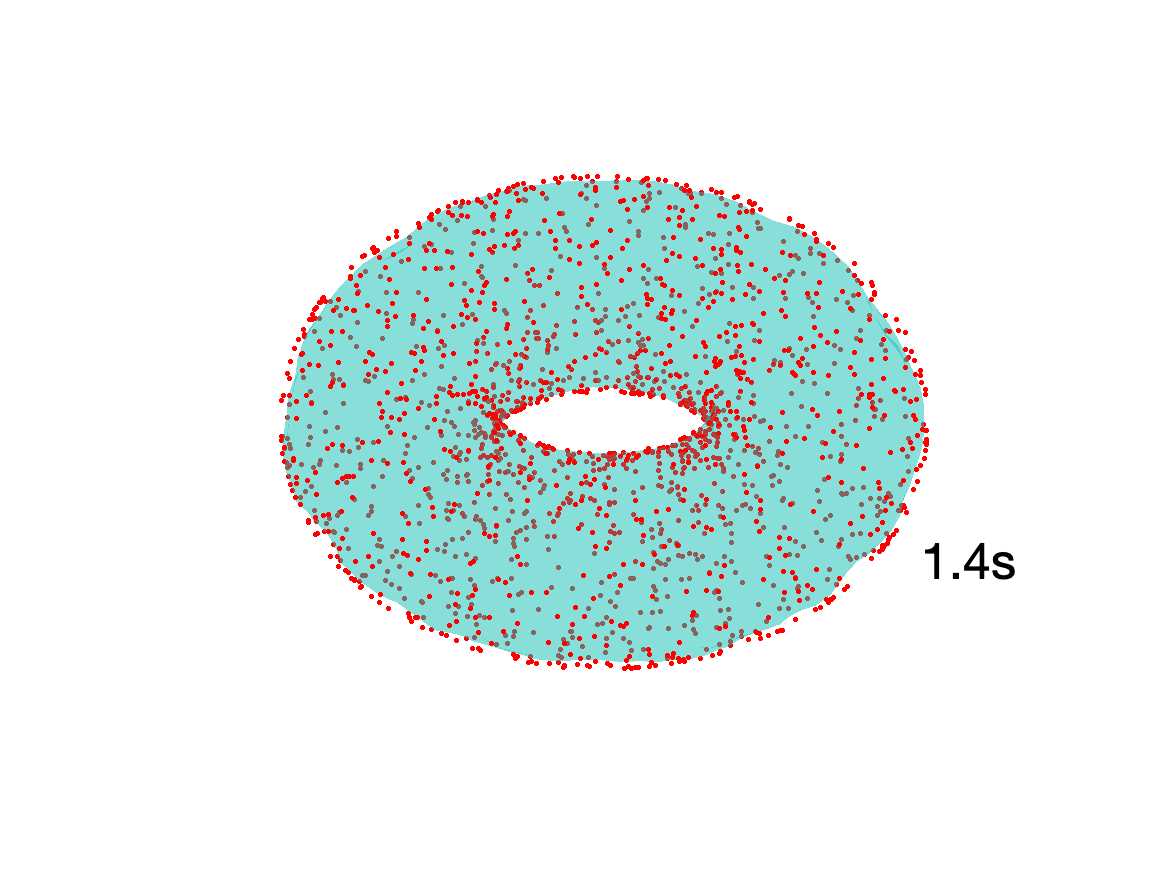} \qquad \qquad  
  \includegraphics[width = 0.3\textwidth,clip, trim = 3.5cm 3cm 2cm 2.5cm ]{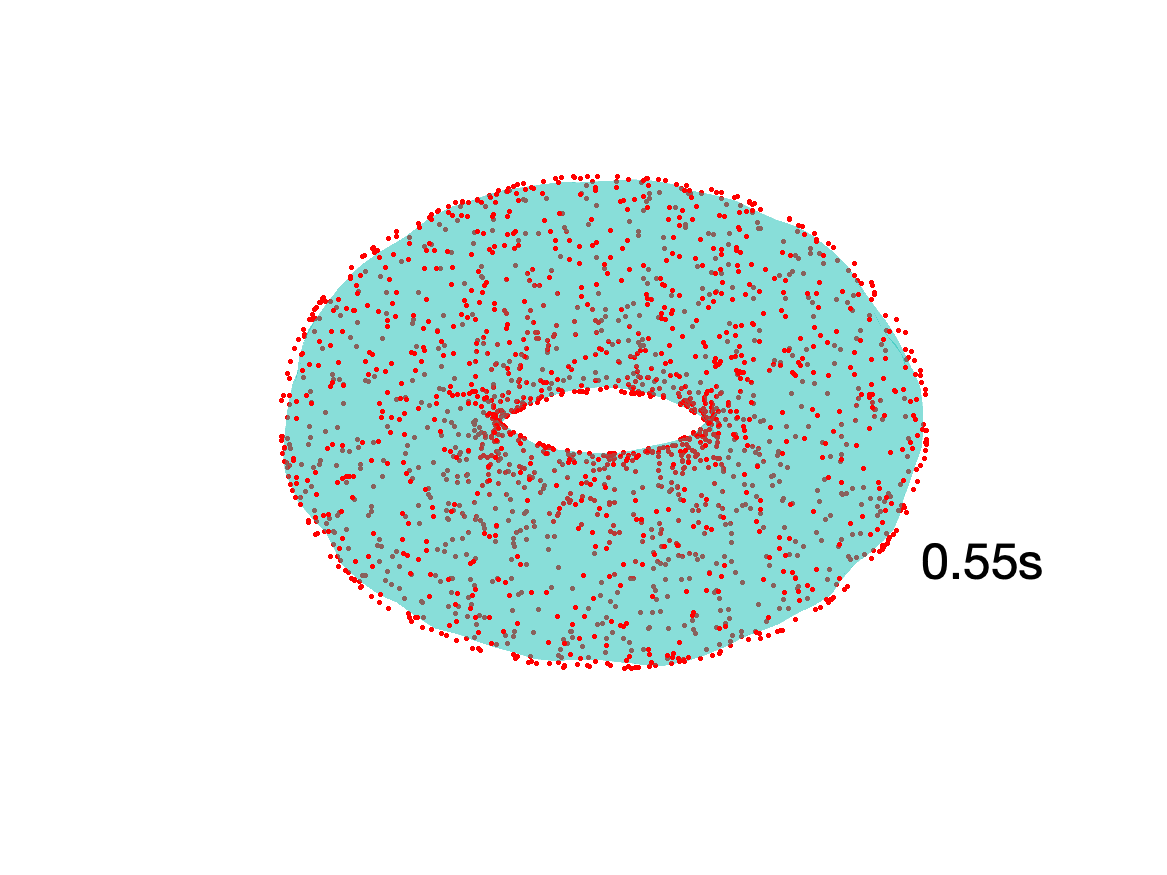}
  \caption{{\bf Left:} The result computed with Algorithm~\ref{a:MBO3}\_\ref{a:MBO} . {\bf Right:} The result computed with  Algorithm~\ref{a:MBO3}\_\ref{a:MBO2}. The $\tau_i$ ($i = 1,2,3,4$) used in both experiments are $0.02, 0.01, 0.005$, and $0.0025$; respectively. See Section~\ref{sec:propertiescheck}.} \label{fig:ex1_adaptive2}
\end{figure}

\subsubsection{The energy decaying property.}\label{sec:energy}
In this section, we show the energy decaying property (especially on Algorithm~\ref{a:MBO3}) via a 2-dimensional three-fold example. The point cloud is generated using $N = 100$ uniform points $\theta_i$ in $[0,2\pi]$:
\[\begin{cases} 
x_i = r_i\cos(\theta_i),\\
y_i = r_i\sin(\theta_i)
\end{cases}\]
where $r_i = 1+0.5\cos(3(\theta-\pi/2))$.

Figure~\ref{fig:energy_decaying_1} displays the energy decaying curves of Algorithms~\ref{a:MBO} and \ref{a:MBO2}.  Even we only have the theoretical proof of the energy decaying property for Algorithm~\ref{a:MBO2}, numerical experiments indicate that Algorithm~\ref{a:MBO} also has the energy decaying property. In both figures, for different choices of $\tau$, the curves have similar profiles with only sketching (or compressing) along the $x$-axis (number of iterations). This is also consistent with the fact that $\tau$ plays the role of time step in the algorithm for the dynamics of the contour, as we discussed in Section~\ref{sec:connection}.

\begin{figure}[ht!]
\centering
 \includegraphics[width = 0.4\textwidth,clip, trim = 0cm 0cm 1cm 1cm ]{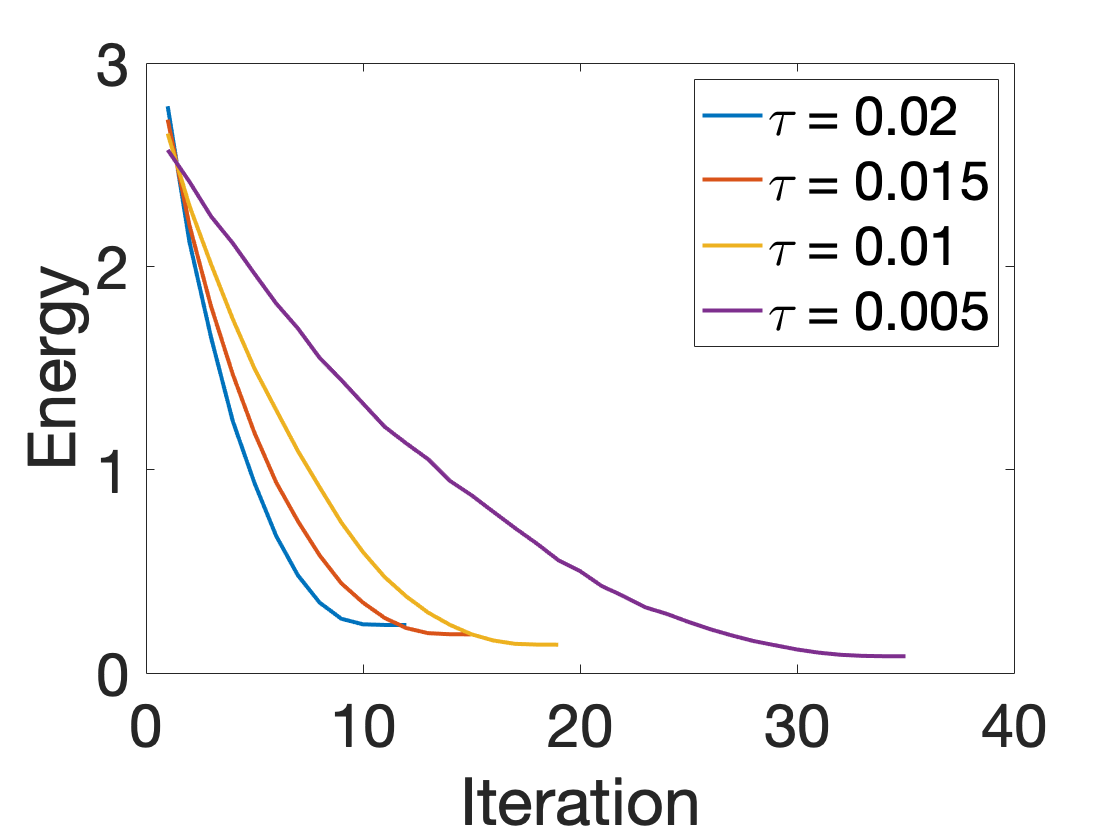} \qquad 
 \includegraphics[width = 0.4\textwidth,clip, trim = 0cm 0cm 1cm 1cm ]{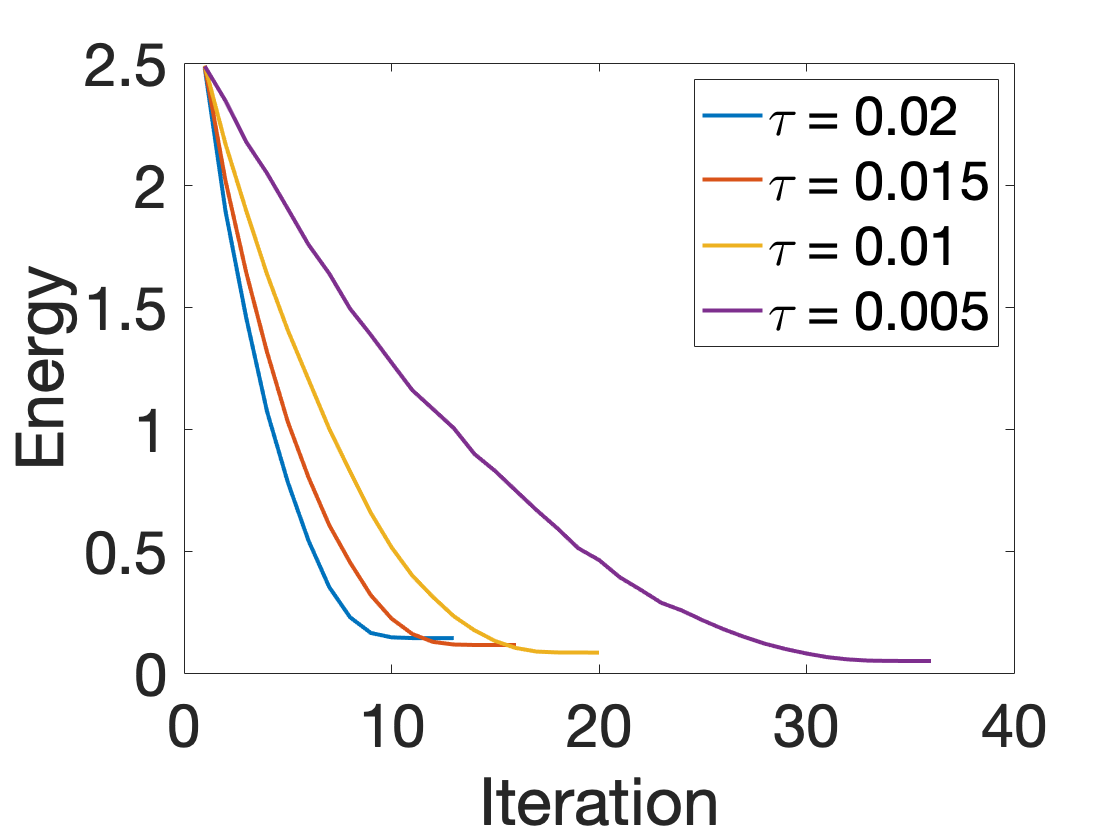}
 \caption{Energy decaying curves for Algorithms~\ref{a:MBO} and \ref{a:MBO2} ({\bf left:} Algorithm~\ref{a:MBO}, {\bf right:} Algorithm~\ref{a:MBO2}). See Section~\ref{sec:energy}.}\label{fig:energy_decaying_1}
\end{figure}

Figure~\ref{fig:energy_decaying_2} displays energy decaying curves during iterations for Algorithm~\ref{a:MBO3}\_\ref{a:MBO} and \ref{a:MBO3}\_\ref{a:MBO2}. In both figures, relatively big jumps occur at the iteration when $\tau$ is changed. Several snapshots at the jumps are shown in both figures. From Figure~\ref{fig:energy_decaying_2}, we observe that the initially choice of $\tau$ (relatively large) can quickly give a solution in the regime of the local minimizer and smaller $\tau$ refines the solution.

\begin{figure}[ht!]
\centering
 \includegraphics[width = 0.8\textwidth,clip, trim = 0cm 0cm 0cm 0cm ]{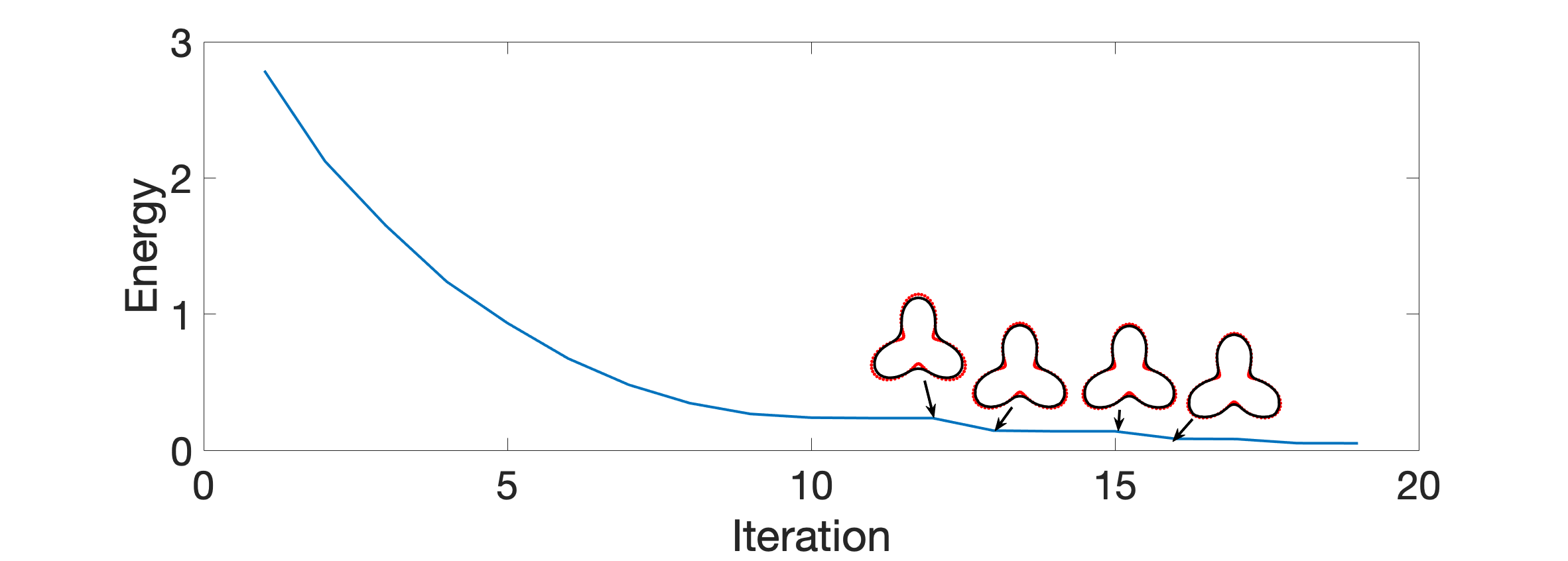} 
 
 \smallskip
 
 \includegraphics[width = 0.8\textwidth,clip, trim = 0cm 0cm 0cm 0cm ]{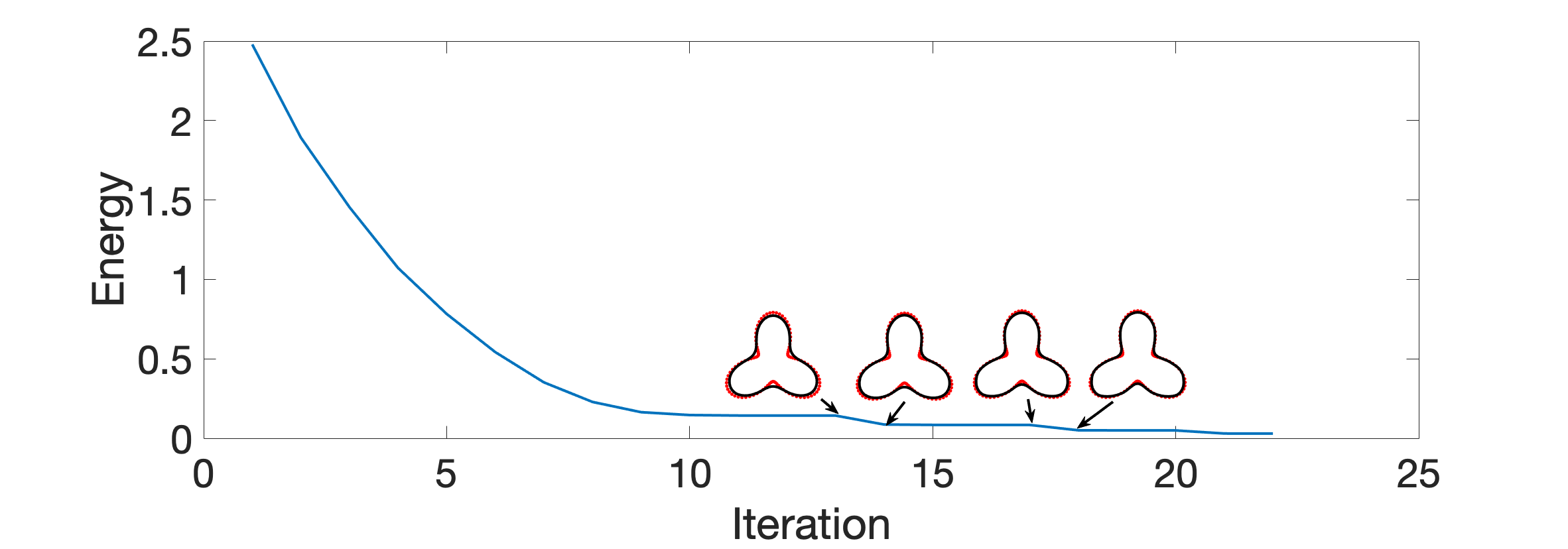}
 \caption{Energy decaying curves for Algorithms~\ref{a:MBO3}\_\ref{a:MBO} and \ref{a:MBO3}\_\ref{a:MBO2} ({\bf top:} Algorithm~\ref{a:MBO3}\_\ref{a:MBO}, {\bf bottom:} Algorithm~\ref{a:MBO3}\_\ref{a:MBO2}). See Section~\ref{sec:energy}.}\label{fig:energy_decaying_2}
\end{figure}

\subsubsection{Sensitivity to $p$.} \label{sec:p}
In this experiment, we check the dependency of the results on the value of $p$. In Figure~\ref{fig:p}, we list the results obtained from Algorithm~\ref{a:MBO3}\_\ref{a:MBO2} for different choices of $p$.  The values of $p$ used in Figure~\ref{fig:p} are $1, 2, 3, 4$, and $5$ from the left to the right, respectively. We observe that for $p\geq2$, the results show little difference. However, the result obtained from Algorithm~\ref{a:MBO3}\_\ref{a:MBO2} using $p=1$ deviates a little from the correct curve. Actually, we did a lot similar experiments for different types of point clouds with different values of $p$, the behavior of the solutions are similar as those in Figure~\ref{fig:p}.  In other words, $p=1$ gives worse results from our numerical observation. This is reasonable because $|d|^{p/2}$ is not Lipschitz on the front when $p=1$.

\begin{figure}[ht!]
\centering
 \includegraphics[width = 0.18\textwidth,clip, trim = 6cm 4.5cm 4cm 3.5cm ]{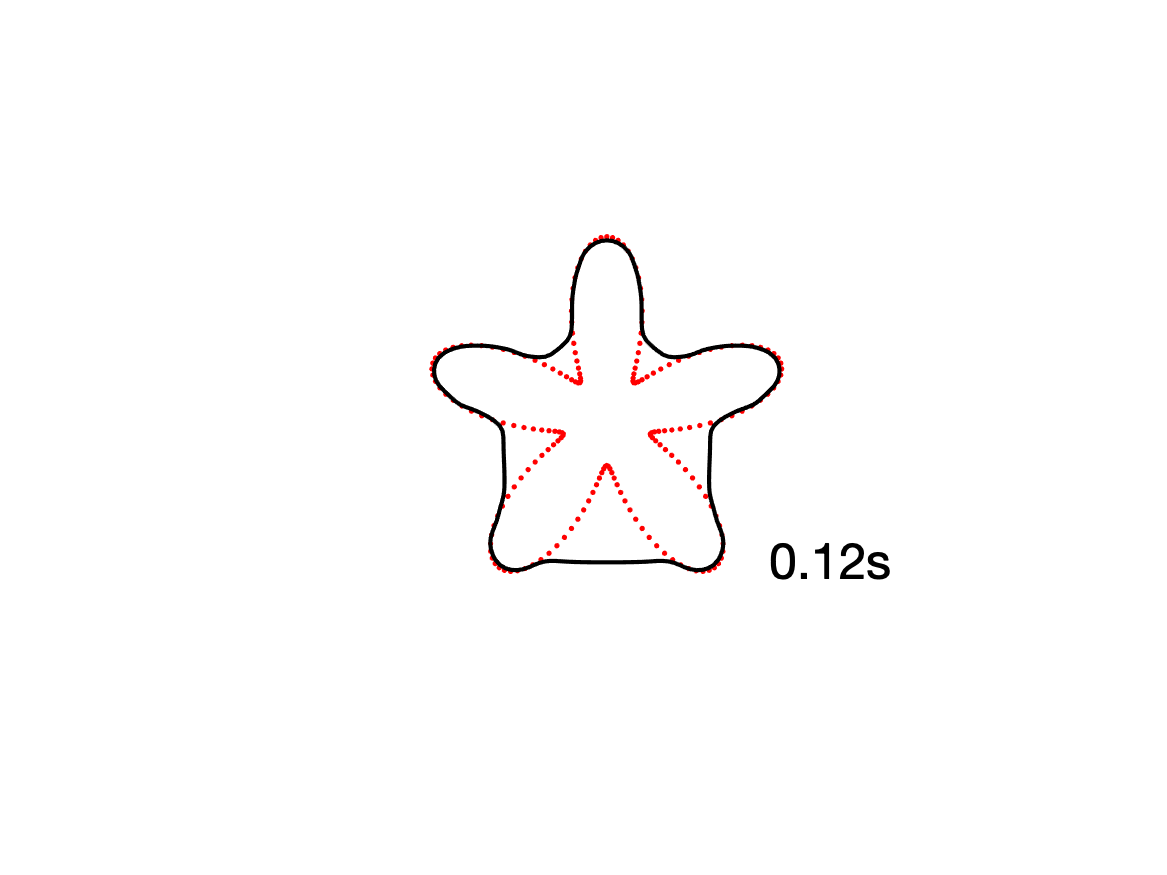}
 \includegraphics[width = 0.18\textwidth,clip, trim =  6cm 4.5cm 4cm 3.5cm  ]{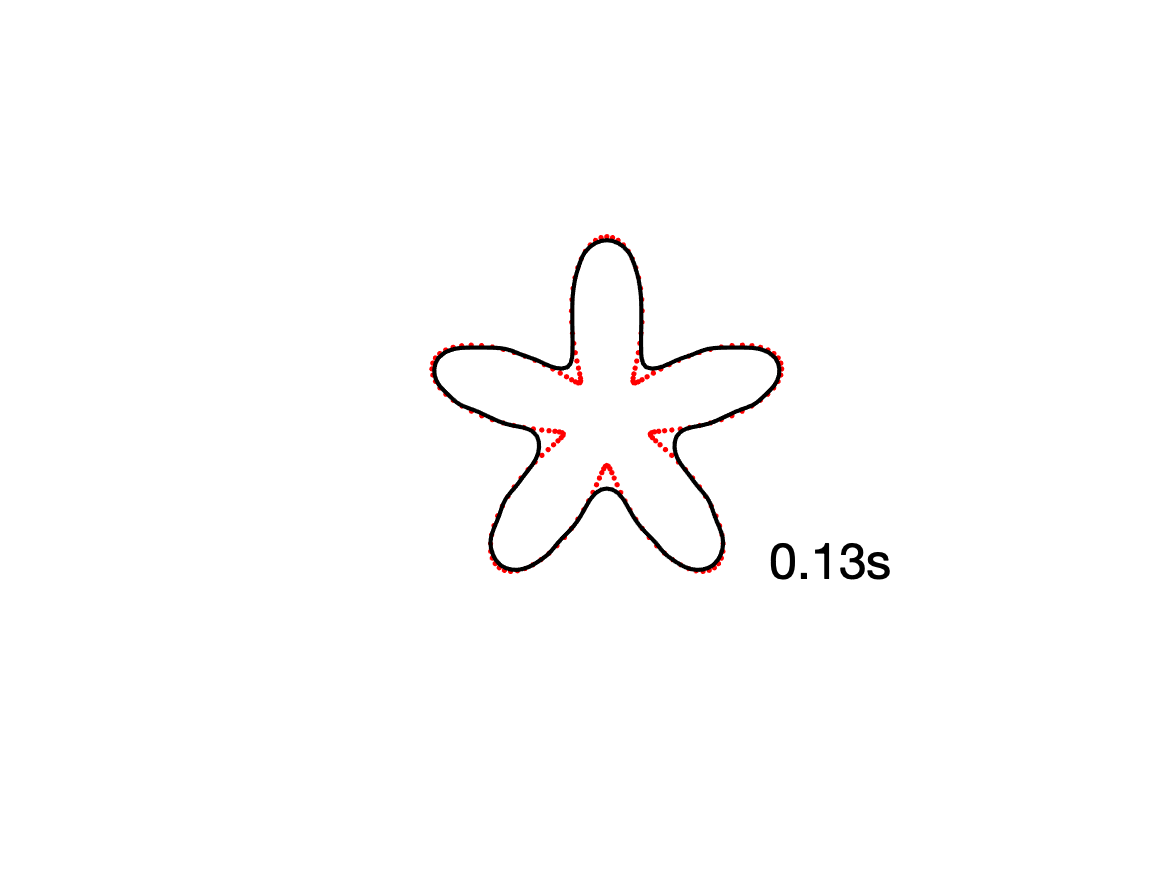}
 \includegraphics[width = 0.18\textwidth,clip, trim =  6cm 4.5cm 4cm 3.5cm  ]{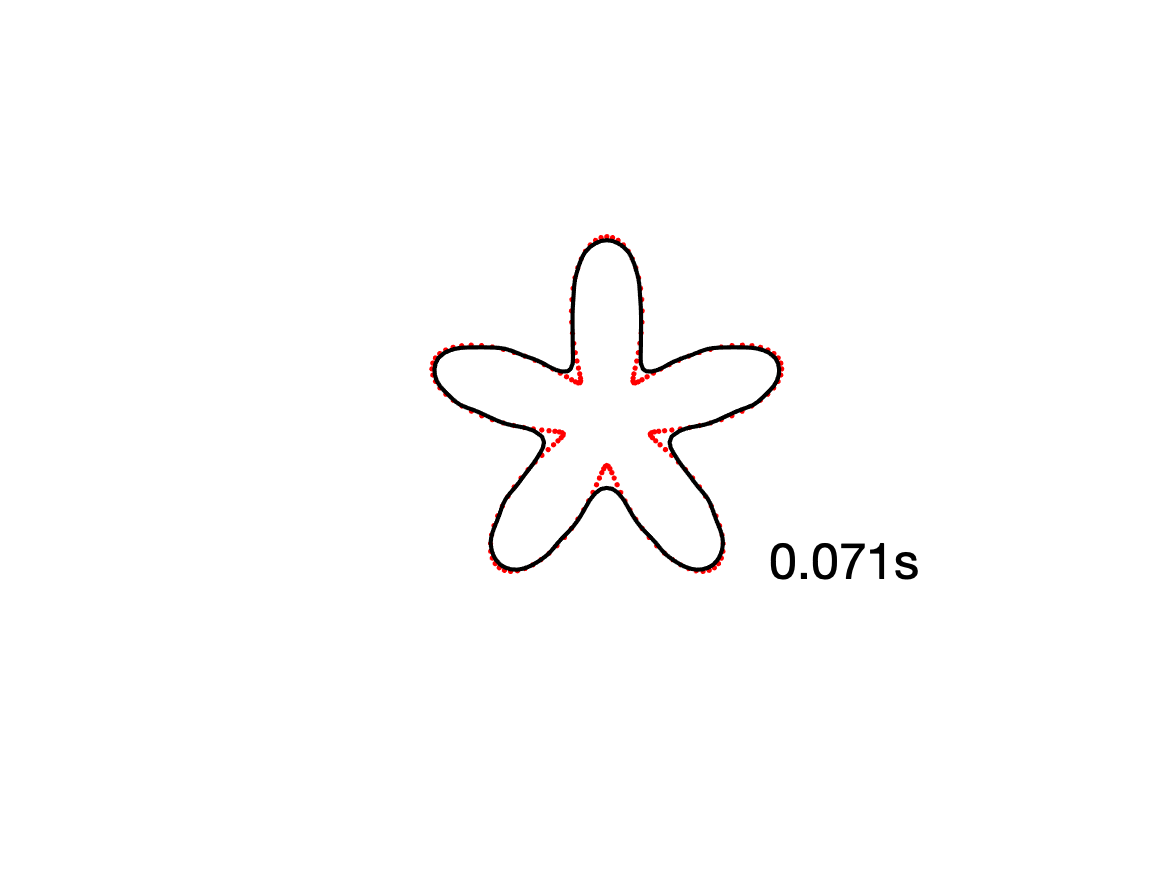}
 \includegraphics[width = 0.18\textwidth,clip, trim =  6cm 4.5cm 4cm 3.5cm  ]{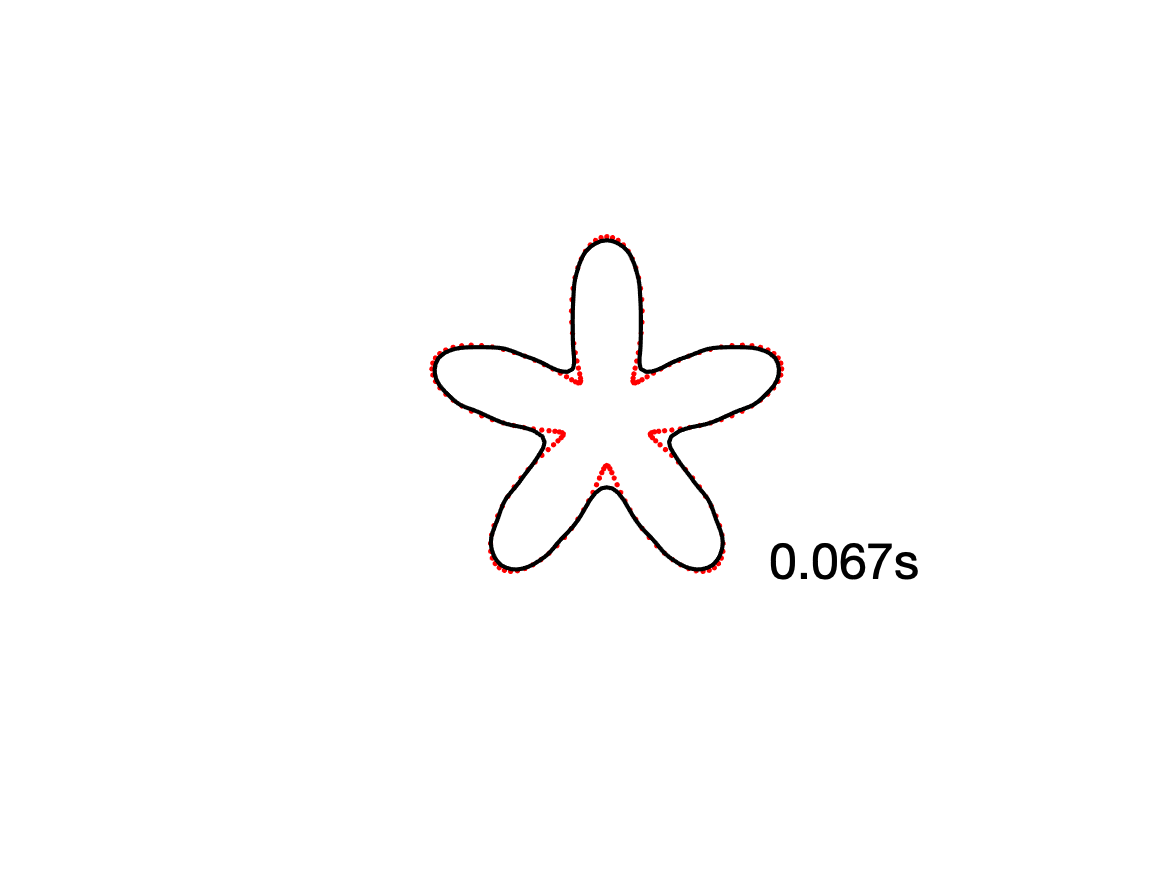}
 \includegraphics[width = 0.18\textwidth,clip, trim =  6cm 4.5cm 4cm 3.5cm ]{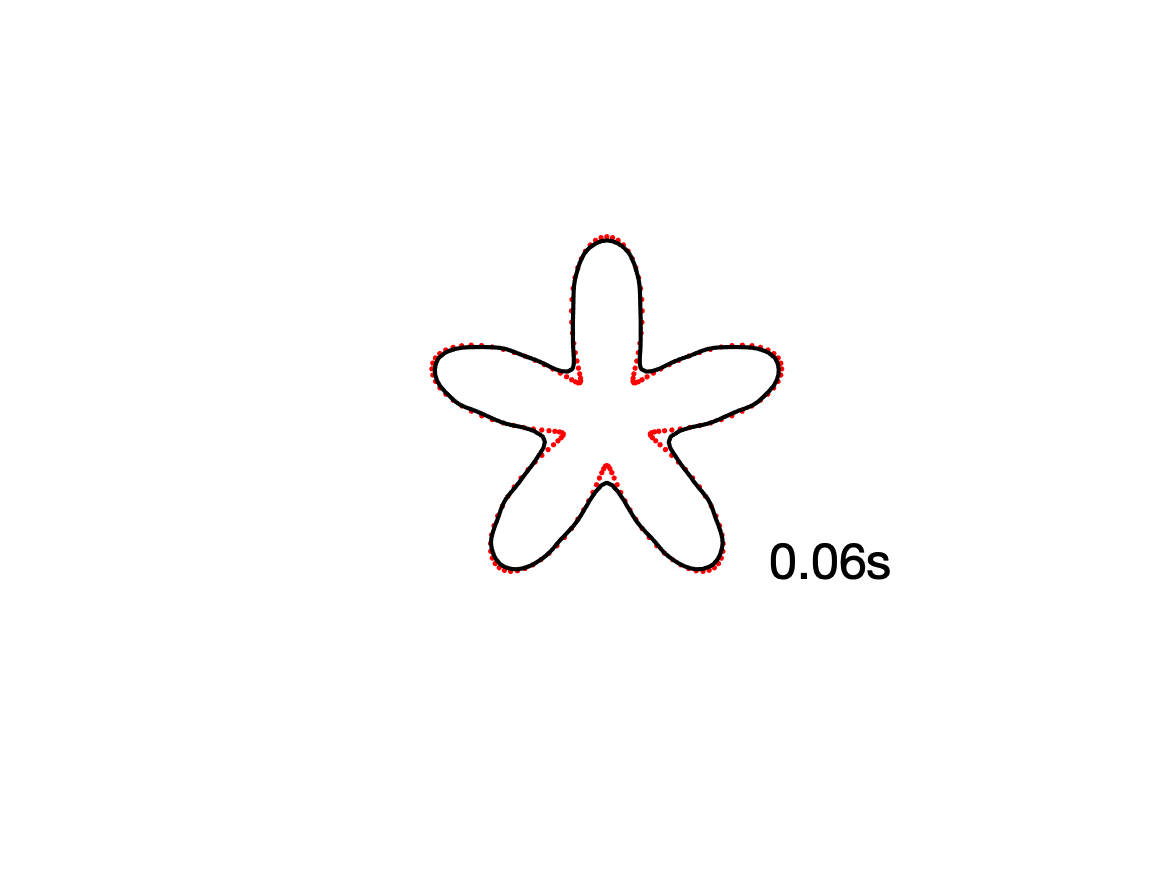} \\
 p = 1  \qquad \qquad \qquad  p = 2  \qquad \qquad \qquad  p = 3  \qquad \qquad \qquad  p = 4  \qquad \qquad \qquad p = 5
 \caption{Results obtained from Algorithm~\ref{a:MBO3}\_\ref{a:MBO2} using different values of $p$. {\bf Left to right:} $p = 1, 2, 3, 4$, and $5$. See Section~\ref{sec:p}.}\label{fig:p}
\end{figure}

\subsubsection{Sensitivity to noisy data.} \label{sec:noise}
In this experiment, we perform several experiments to show the capability of proposed algorithms in noisy data. Because Algorithm~\ref{a:MBO3} gives better results than Algorithms~\ref{a:MBO} and \ref{a:MBO2}, in the following and subsequent experiments, we only show results from Algorithm~\ref{a:MBO3}\_\ref{a:MBO2}. 

We consider the noisy data $\tilde \bx$ generated from the pure data $\bx$: $\tilde \bx= \bx+\mu \nu$, where $\nu$ is a vector whose entries are independently and identically distributed random variables from a normal distribution and $\mu$ is a parameter to control the intensity of noise. In Figure~\ref{fig:noise_1}, we observe that the method still works for noisy data, especially when the noise intensity is not very high. In addition, the two-circle case indicates that the algorithm works well for topological changing cases. In these experiments, the computational domain $[-\pi,\pi]^2$ is discretized by $128^2$ grids.

\begin{figure}[ht!]
\centering
 \includegraphics[width = 0.2\textwidth,clip, trim = 6cm 3cm 5cm 3cm]{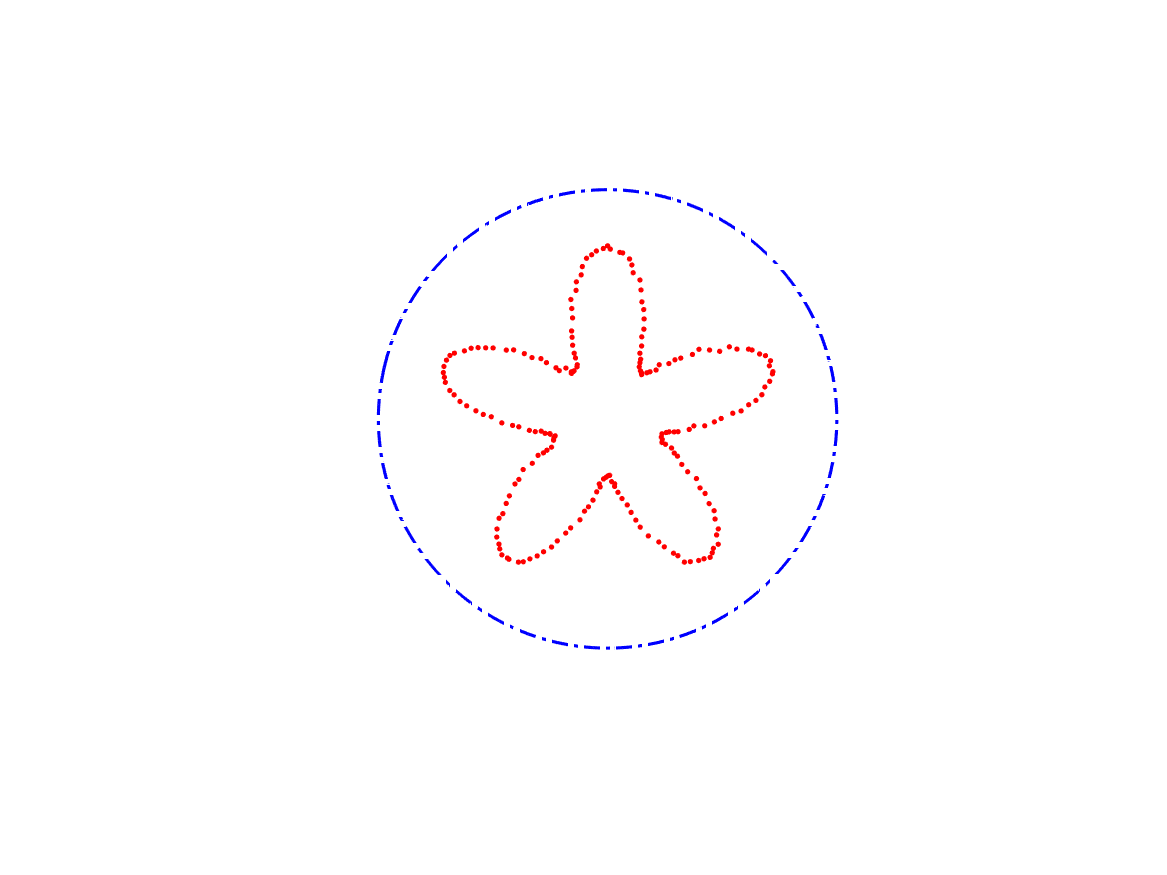} 
 \includegraphics[width = 0.2\textwidth,clip, trim = 6cm 3cm 5cm 3cm]{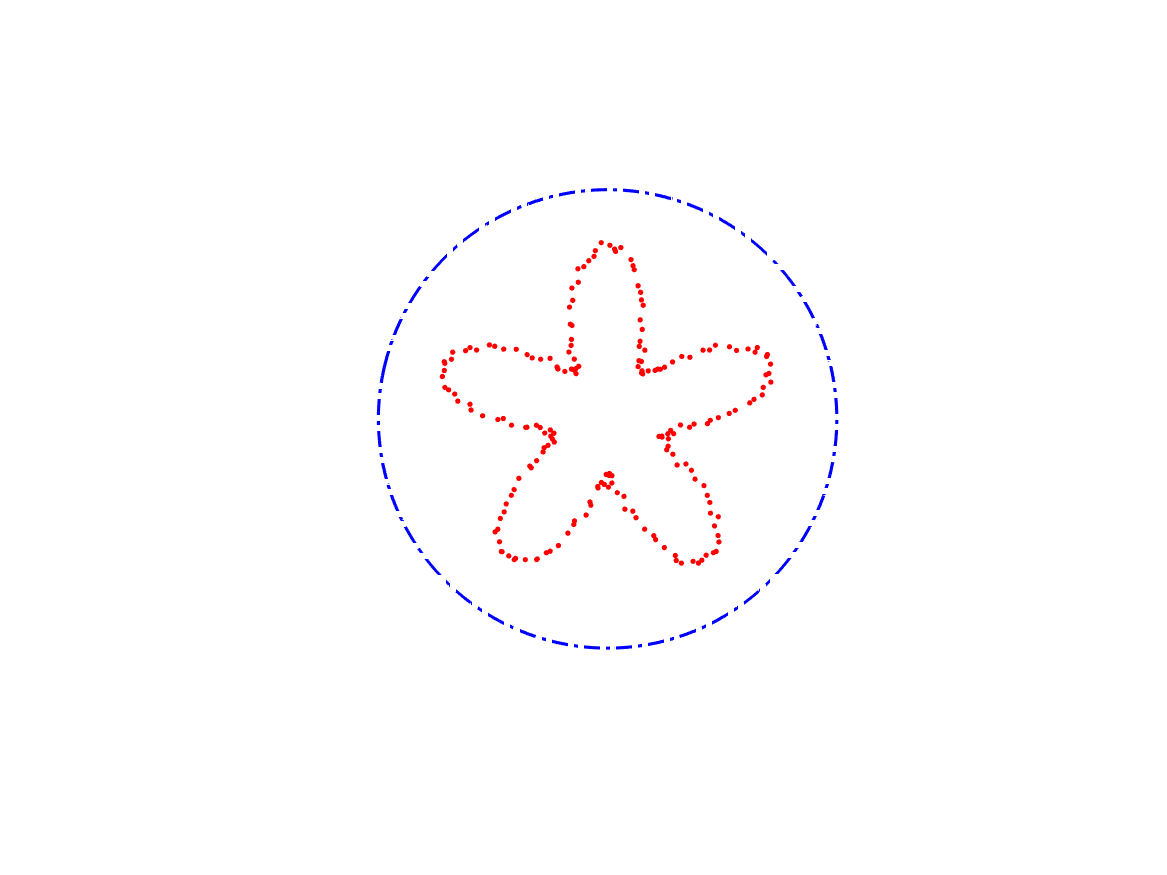} 
 \includegraphics[width = 0.2\textwidth,clip, trim = 6cm 3cm 5cm 3cm]{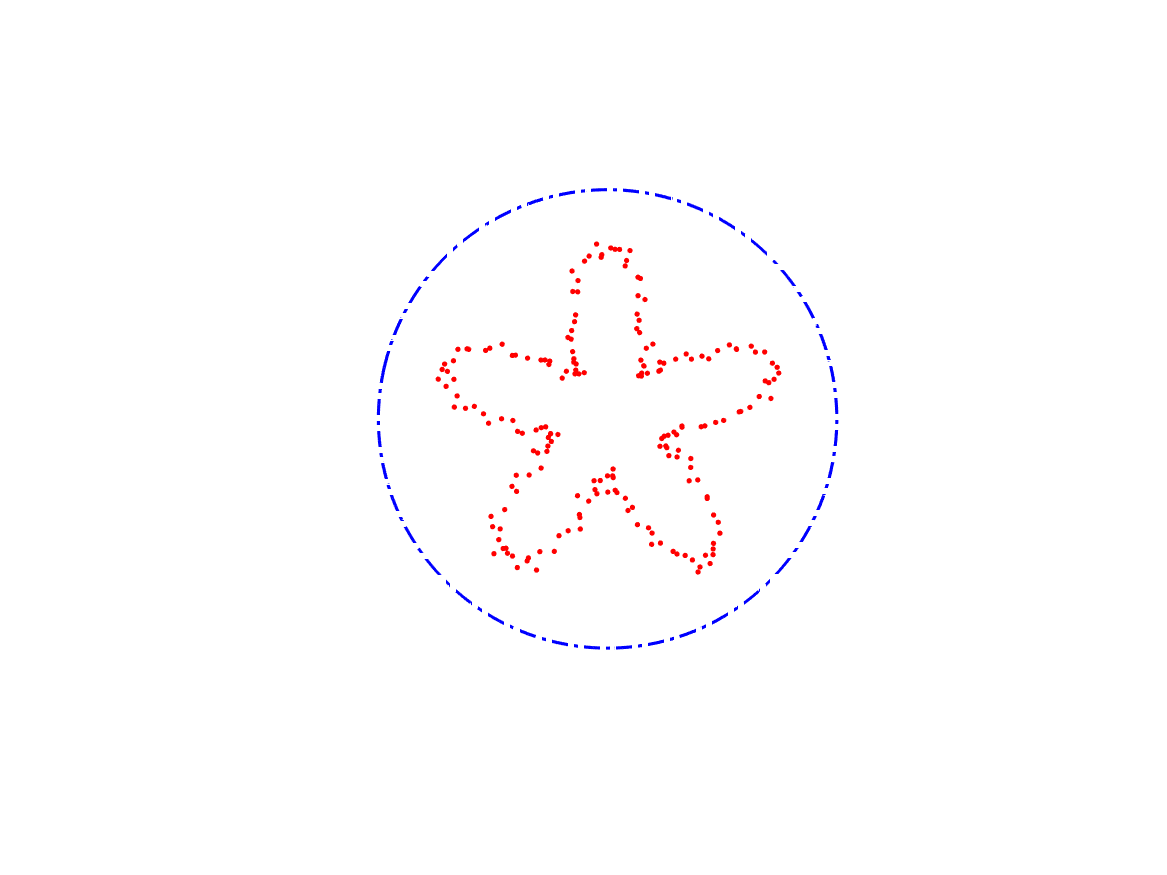} 
 \includegraphics[width = 0.2\textwidth,clip, trim = 6cm 3cm 5cm 3cm]{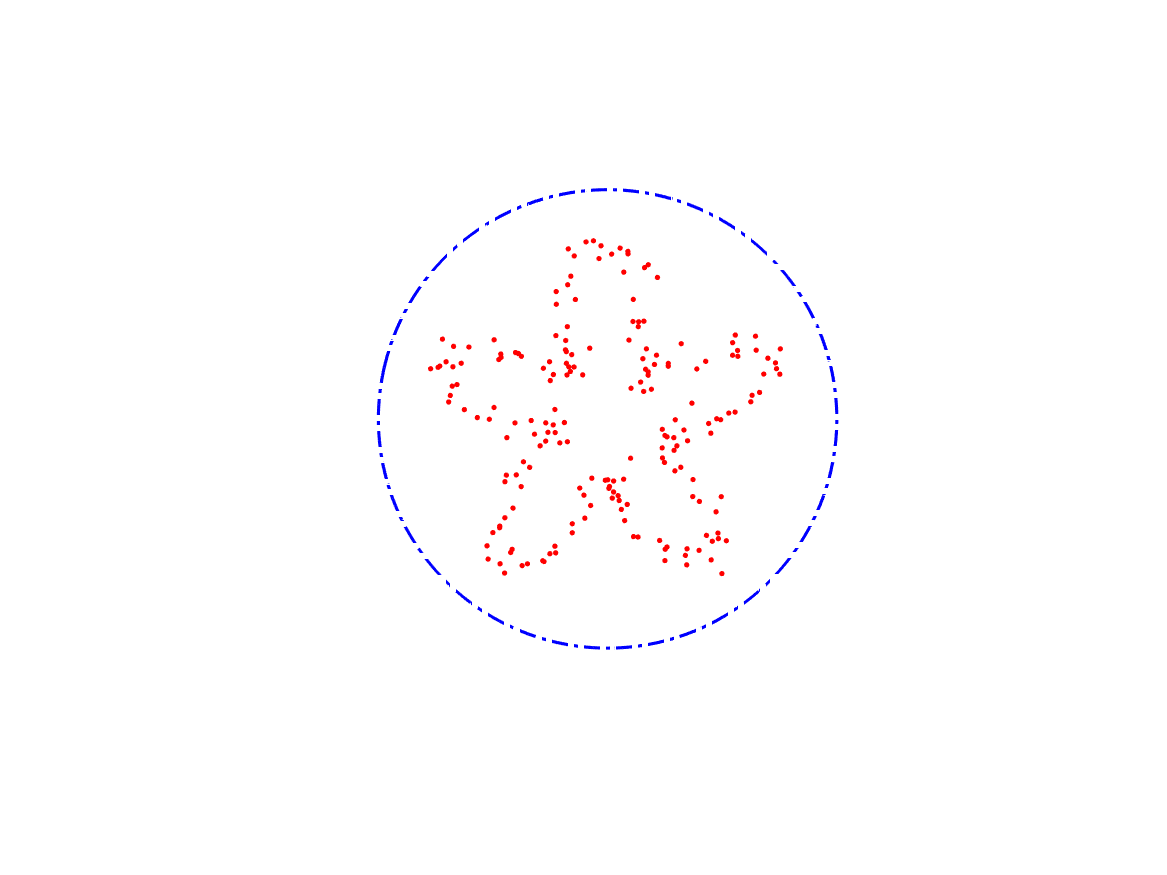} 
 \includegraphics[width = 0.2\textwidth,clip, trim = 6cm 3cm 4cm 3cm]{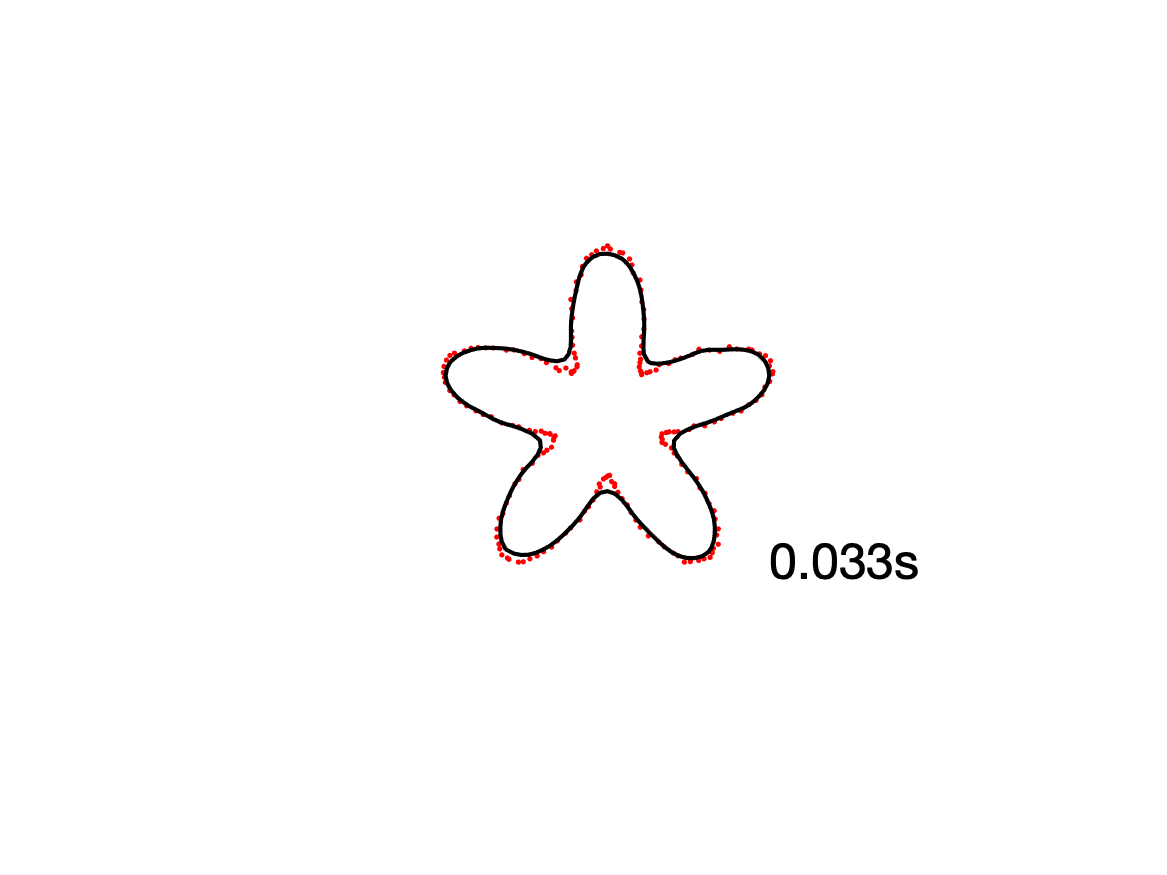} 
 \includegraphics[width = 0.2\textwidth,clip, trim = 6cm 3cm 4cm 3cm]{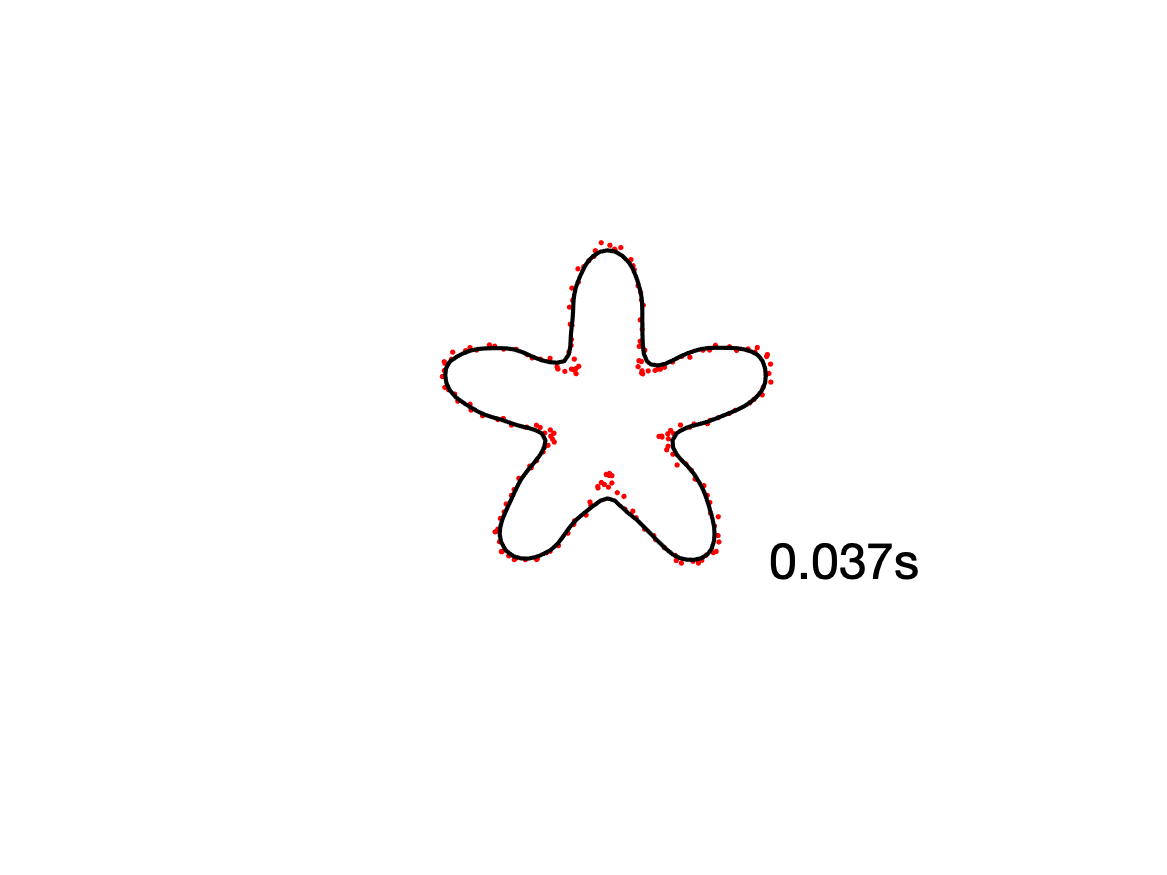} 
 \includegraphics[width = 0.2\textwidth,clip, trim = 6cm 3cm 4cm 3cm]{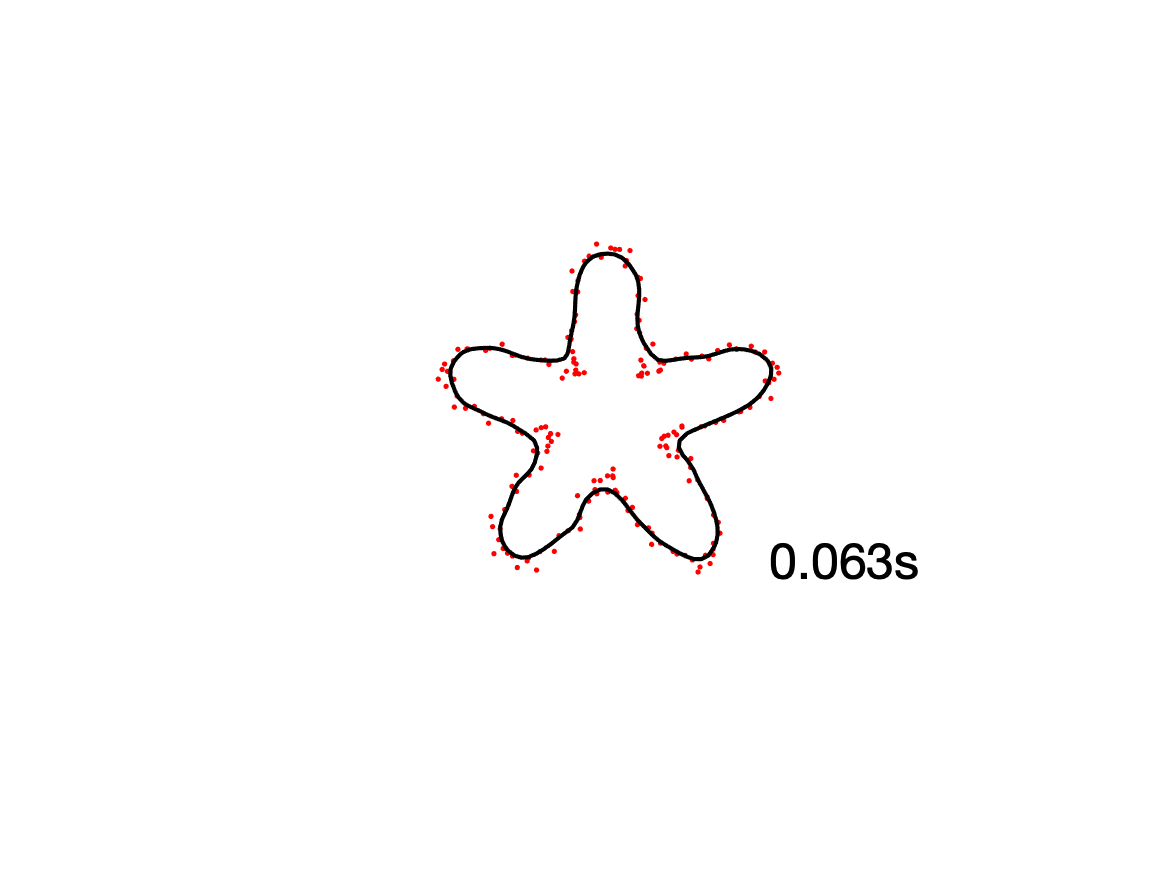} 
 \includegraphics[width = 0.2\textwidth,clip, trim = 6cm 3cm 4cm 3cm]{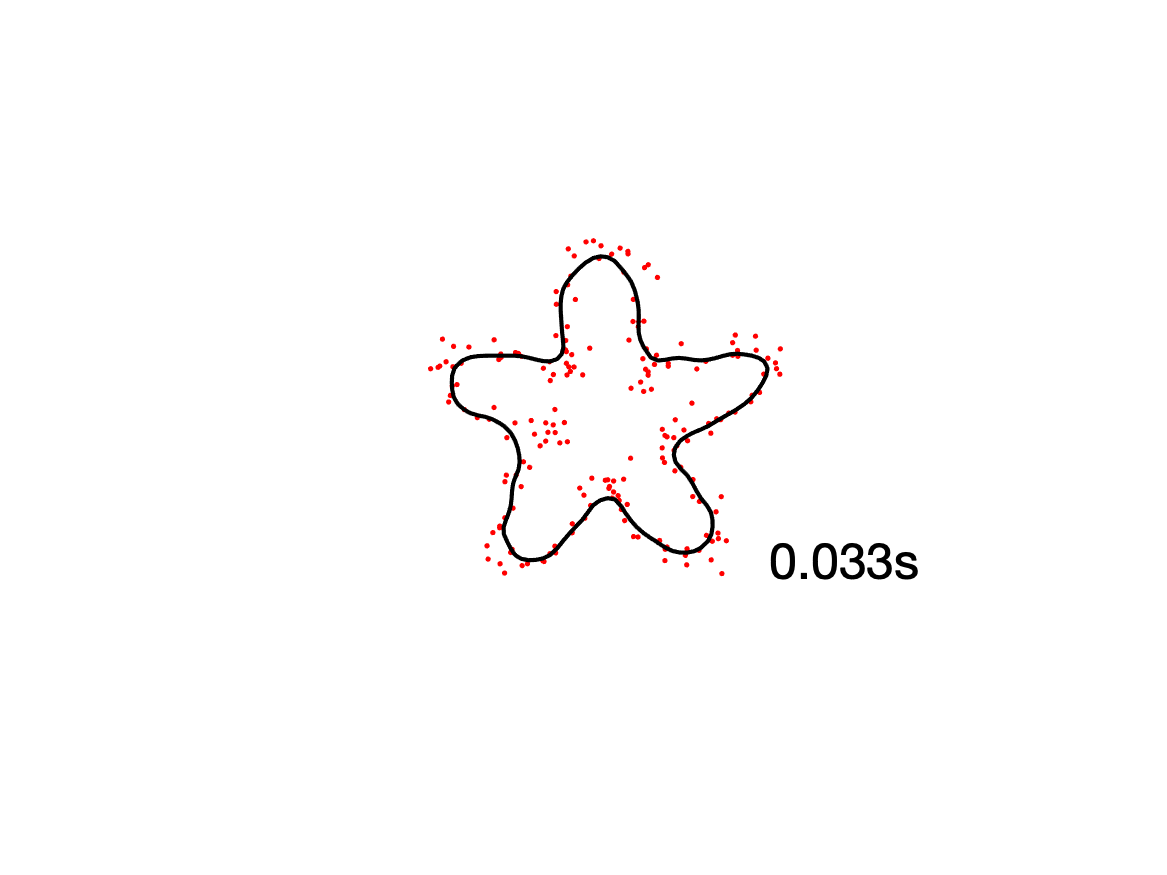} 
  \includegraphics[width = 0.2\textwidth,clip, trim = 6cm 3cm 5cm 3cm]{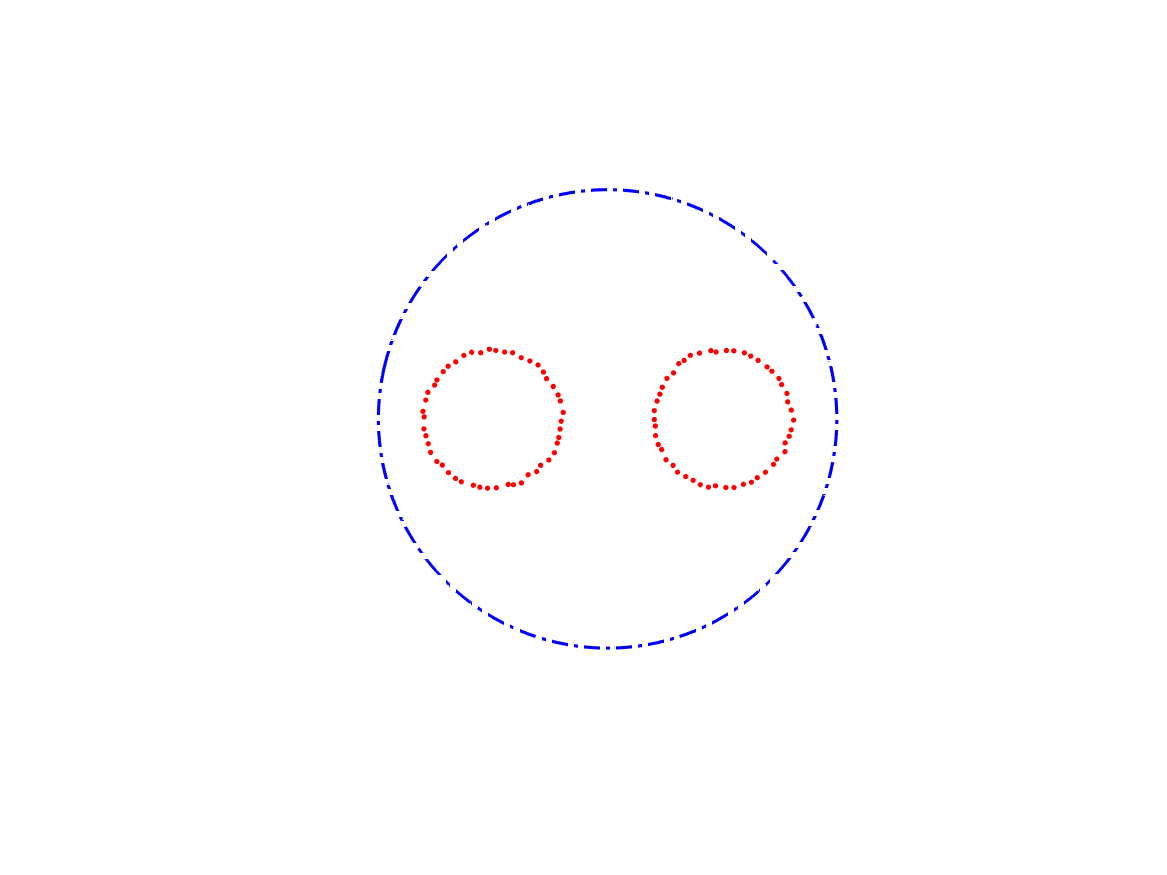} 
 \includegraphics[width = 0.2\textwidth,clip, trim = 6cm 3cm 5cm 3cm]{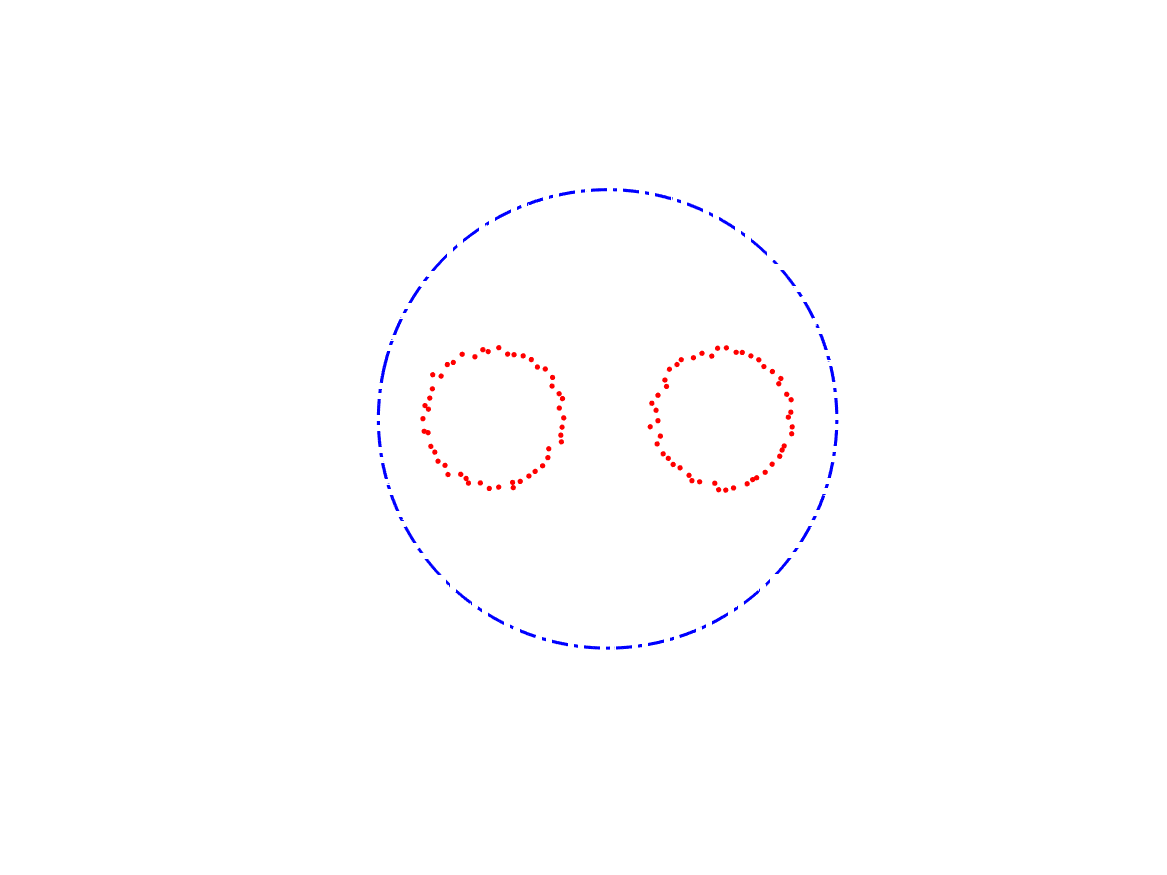} 
 \includegraphics[width = 0.2\textwidth,clip, trim = 6cm 3cm 5cm 3cm]{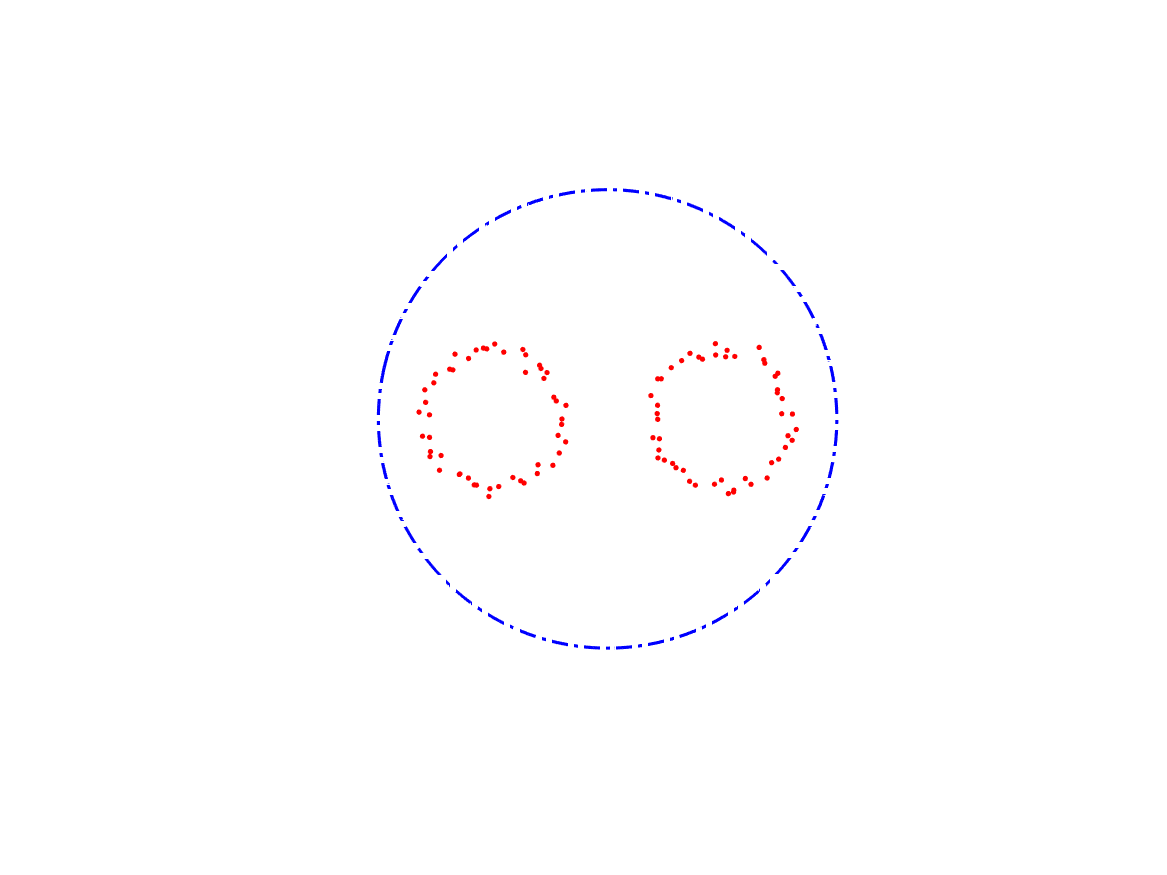} 
 \includegraphics[width = 0.2\textwidth,clip, trim = 6cm 3cm 5cm 3cm]{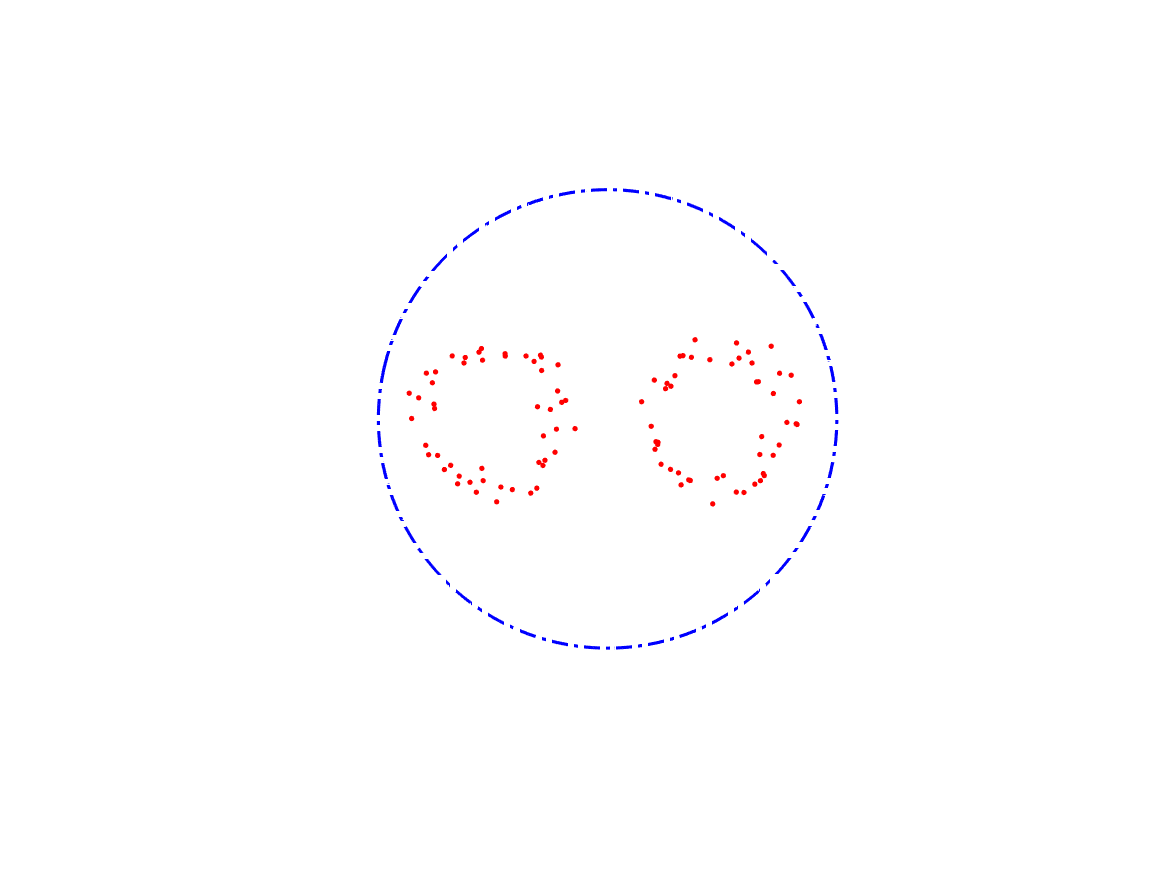} 
 \includegraphics[width = 0.2\textwidth,clip, trim = 6cm 3cm 4cm 3cm]{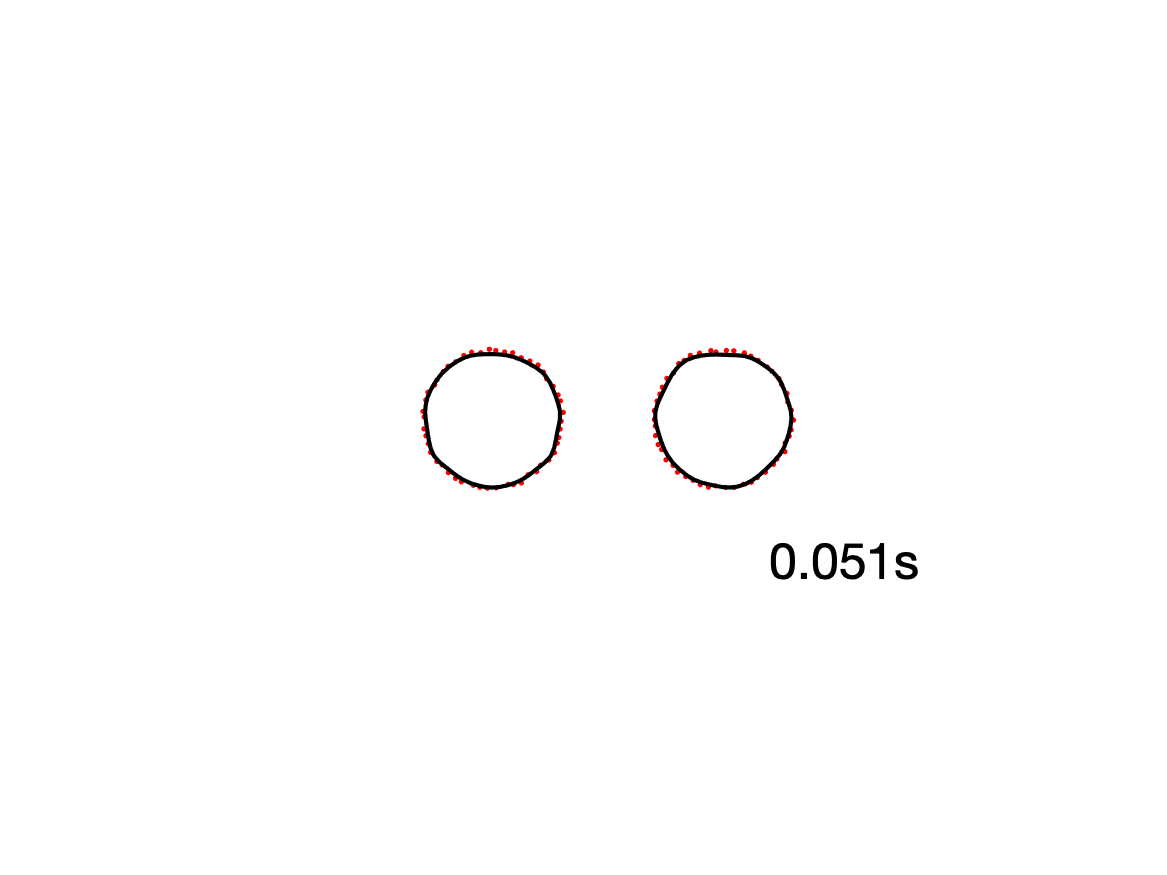} 
 \includegraphics[width = 0.2\textwidth,clip, trim = 6cm 3cm 4cm 3cm]{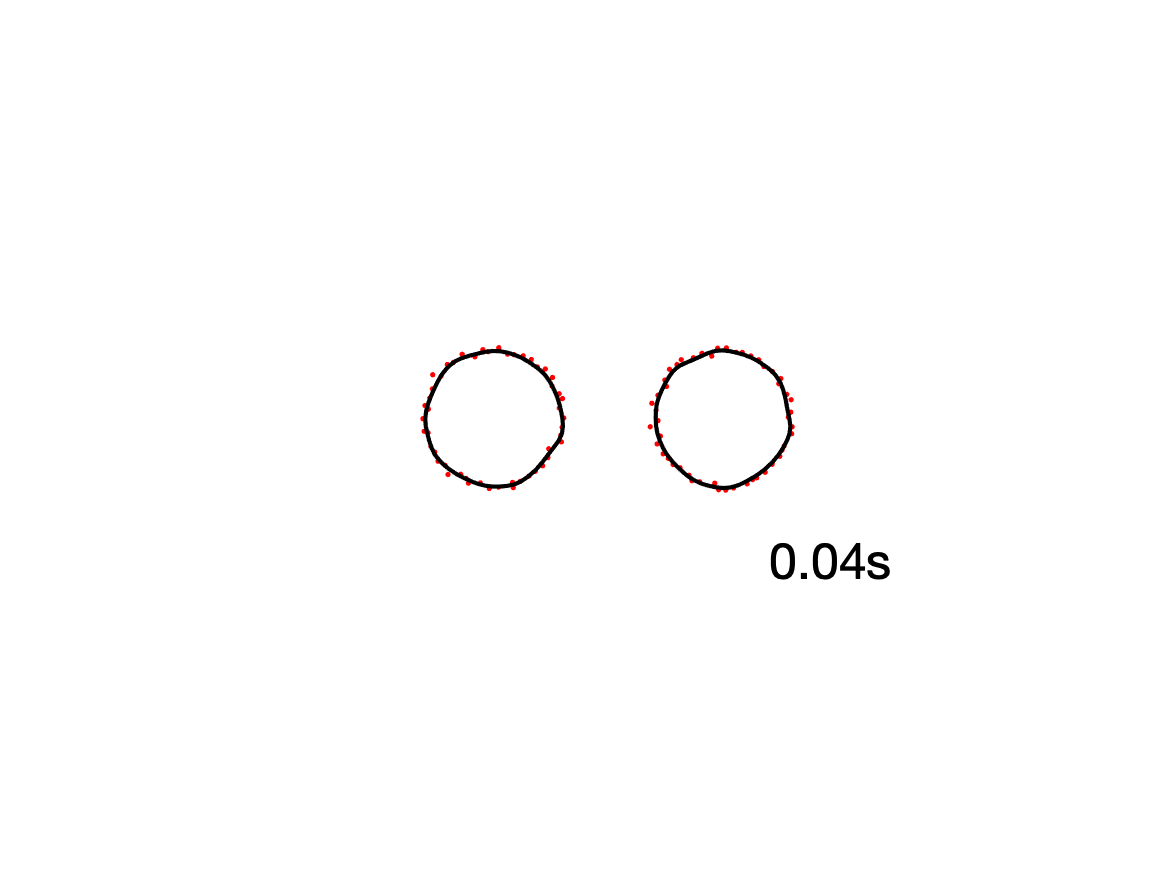} 
 \includegraphics[width = 0.2\textwidth,clip, trim = 6cm 3cm 4cm 3cm]{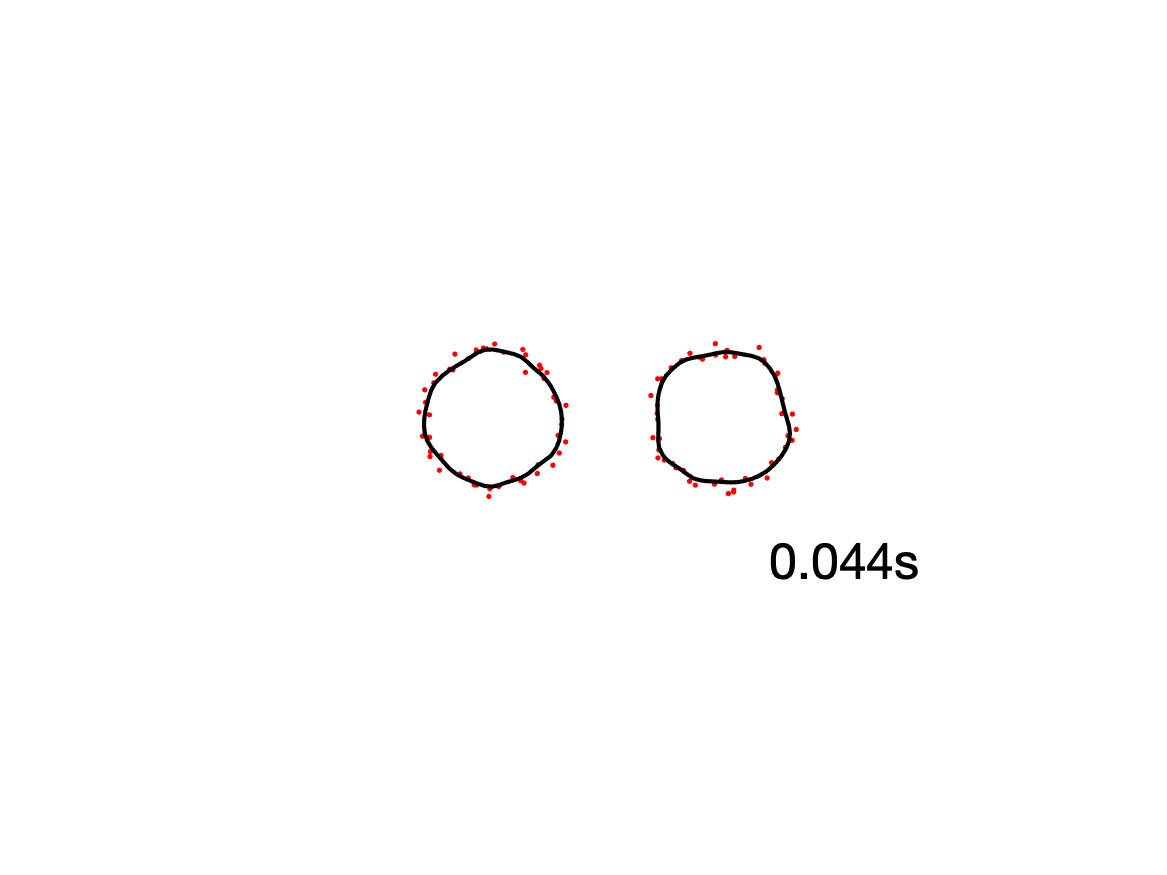} 
 \includegraphics[width = 0.2\textwidth,clip, trim = 6cm 3cm 4cm 3cm]{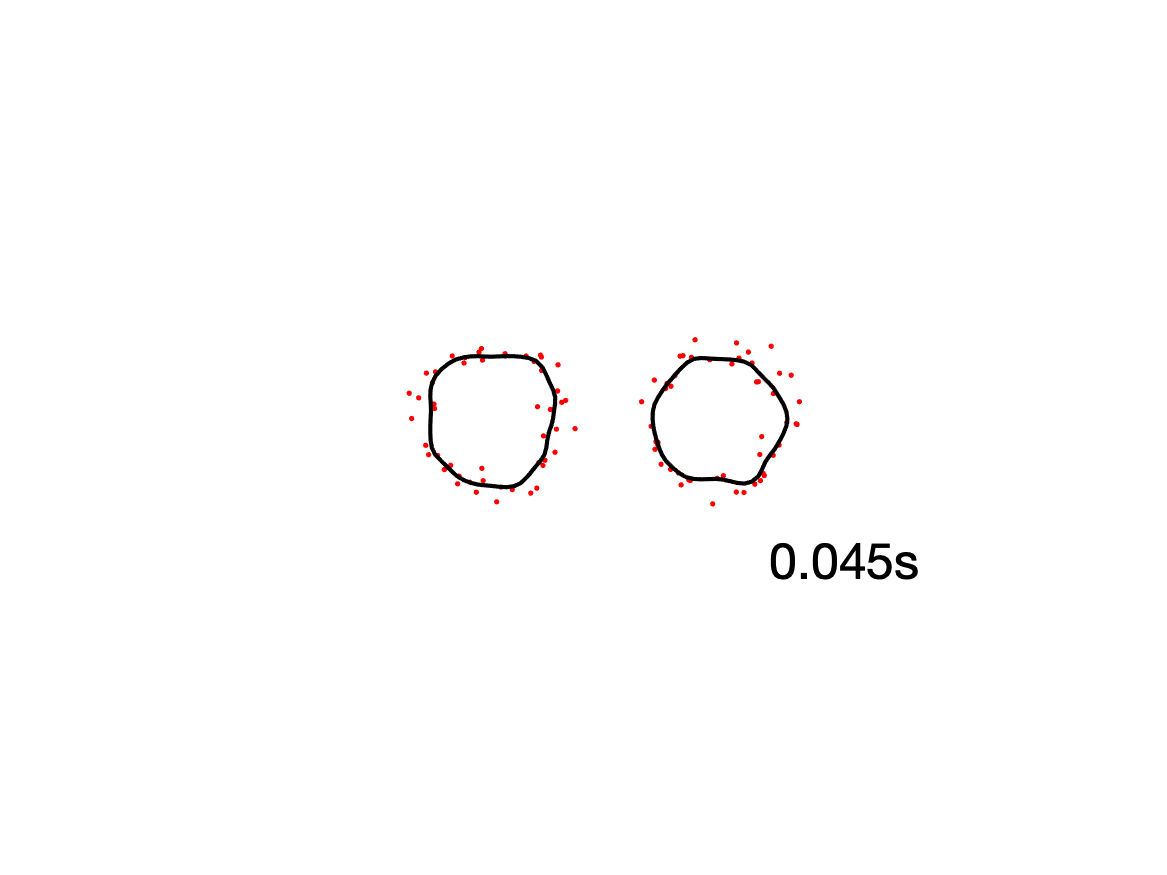} 
 \caption{Results obtained for the five-fold data with different noise. {\bf Left to right:} $\mu = 0.01, 0.02, 0.04$, $0.08$. {\bf 1st and 3rd row:} Noisy data with the initial guess. {\bf 2nd and 4th row:} Results obtained from Algorithm~\ref{a:MBO3}\_\ref{a:MBO2}.  See Section~\ref{sec:noise}. }\label{fig:noise_1}
\end{figure}

\subsubsection{Sensitivity to resolutions.}\label{sec:sensitivity}

In this experiment, we check the sensitivity of the proposed algorithms to the resolution of the computational domain and the number of the point cloud. In Figure~\ref{fig:sensitivity}, we list the results obtained from Algorithm~\ref{a:MBO3}\_\ref{a:MBO2}, using different discretization of the computational domain and different size of the point cloud. In Figure~\ref{a:MBO3}, the size of point clouds from the left to the right are $100$, $200$, $300$, and $400$, respectively. The computational domains from the top to the bottom are discretized by $64^2$, $128^2$, $256^2$, and $512^2$ grids. We observe that as the mesh is refined, the result is closer to the desired curve. In addition, point clouds with larger size can give better results. These agree with our expectations.

\begin{figure}[ht!]
\centering
\includegraphics[width = 0.2\textwidth,clip, trim = 6cm 3cm 4cm 3cm]{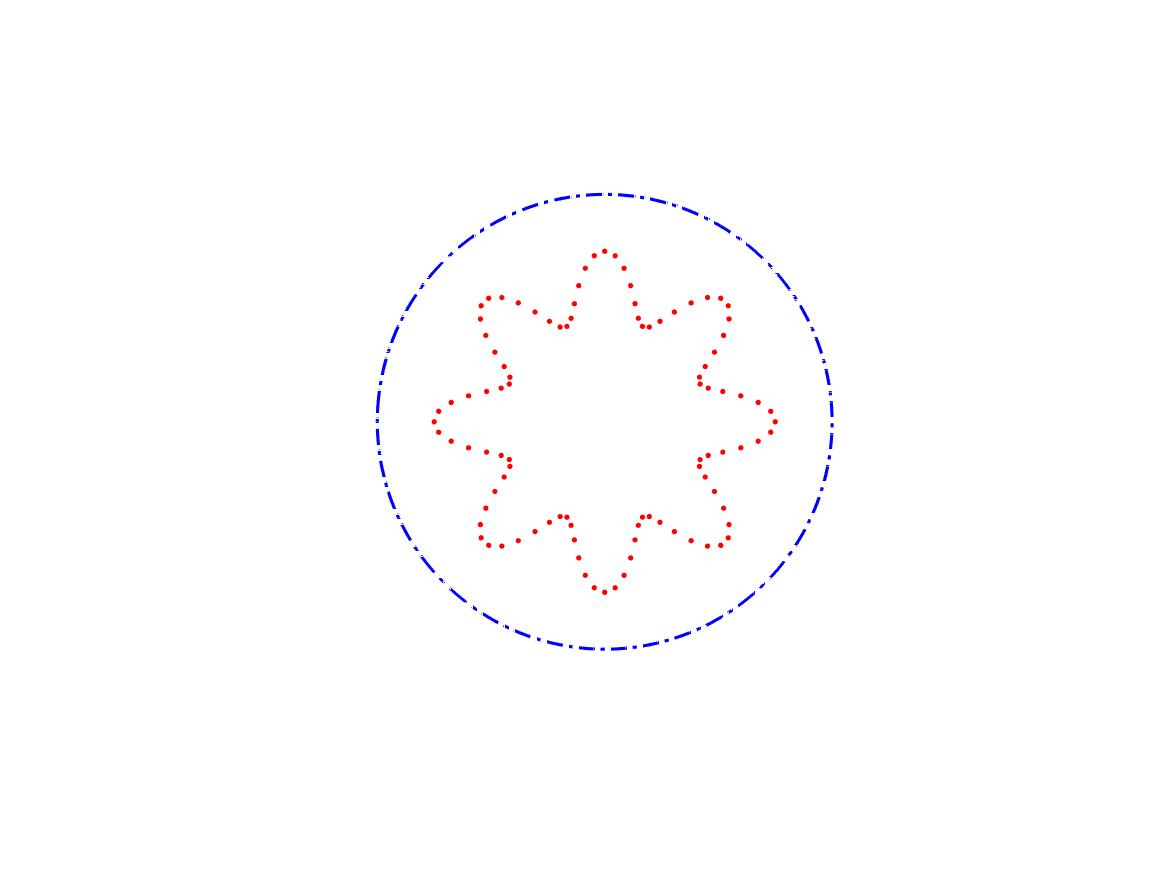}
\includegraphics[width = 0.2\textwidth,clip, trim = 6cm 3cm 4cm 3cm]{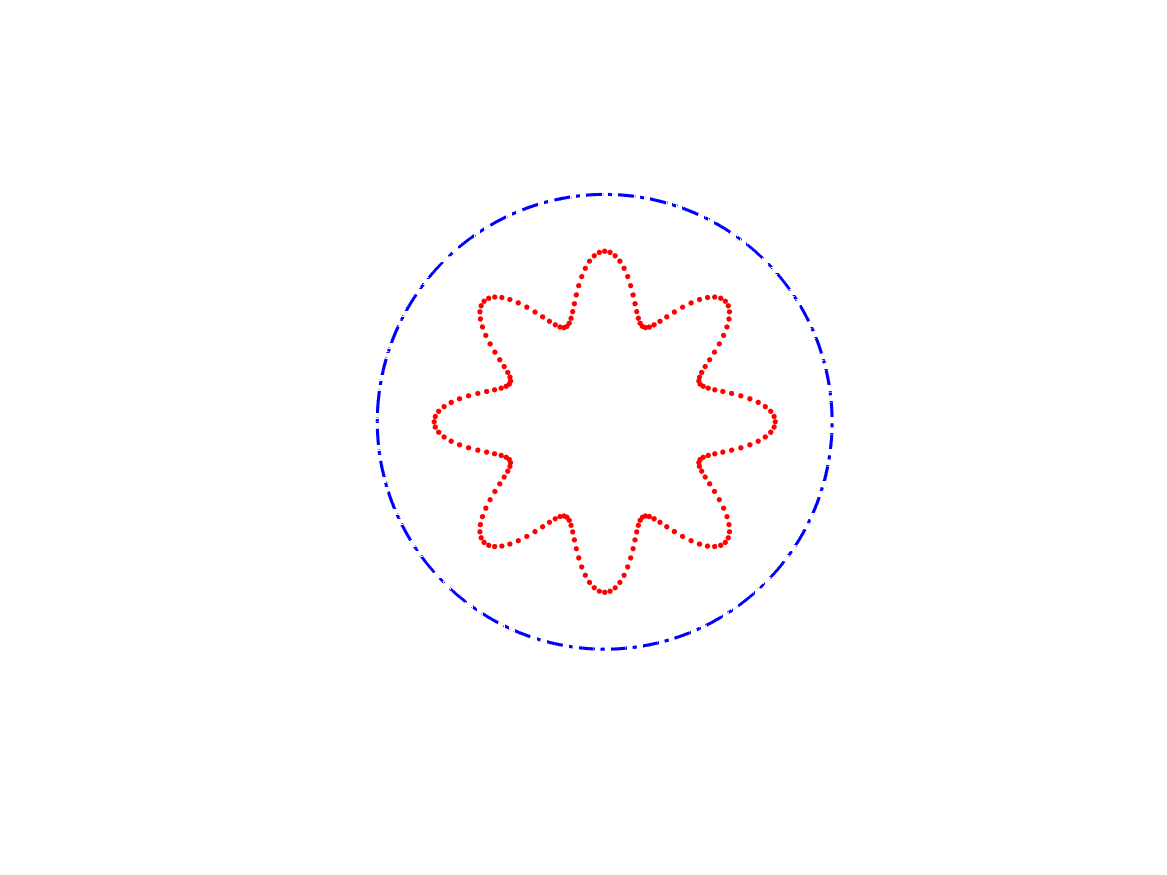}
   \includegraphics[width = 0.2\textwidth,clip, trim = 6cm 3cm 4cm 3cm]{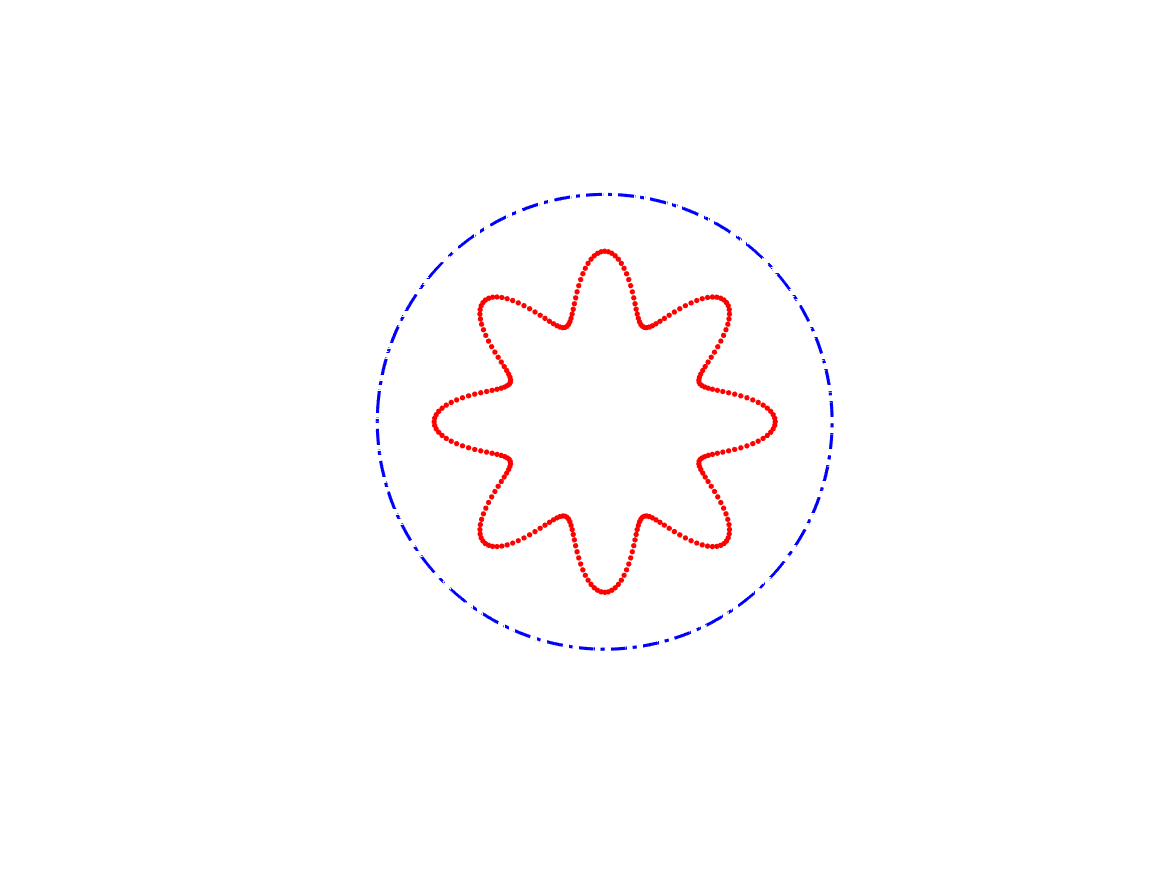}
   \includegraphics[width = 0.2\textwidth,clip, trim = 6cm 3cm 4cm 3cm]{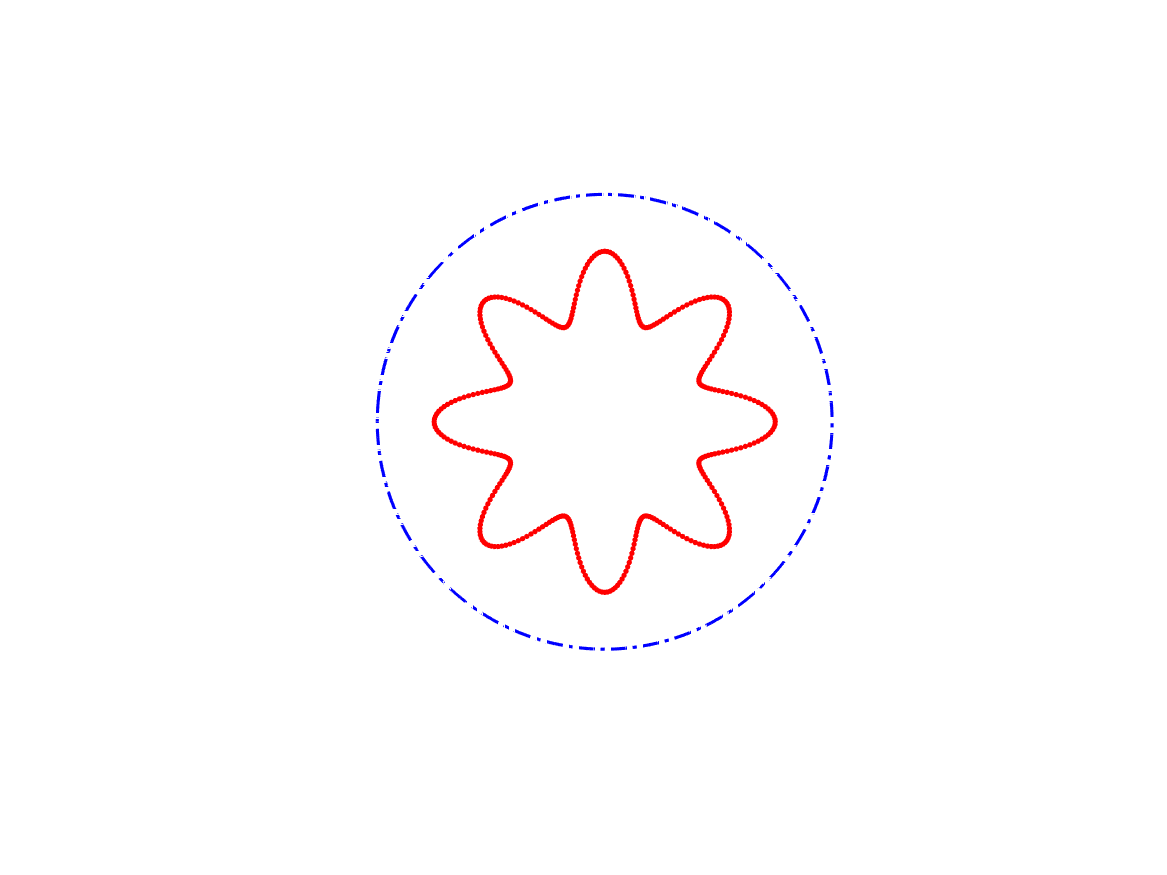} \\
\includegraphics[width = 0.2\textwidth,clip, trim = 6cm 4cm 4cm 4cm]{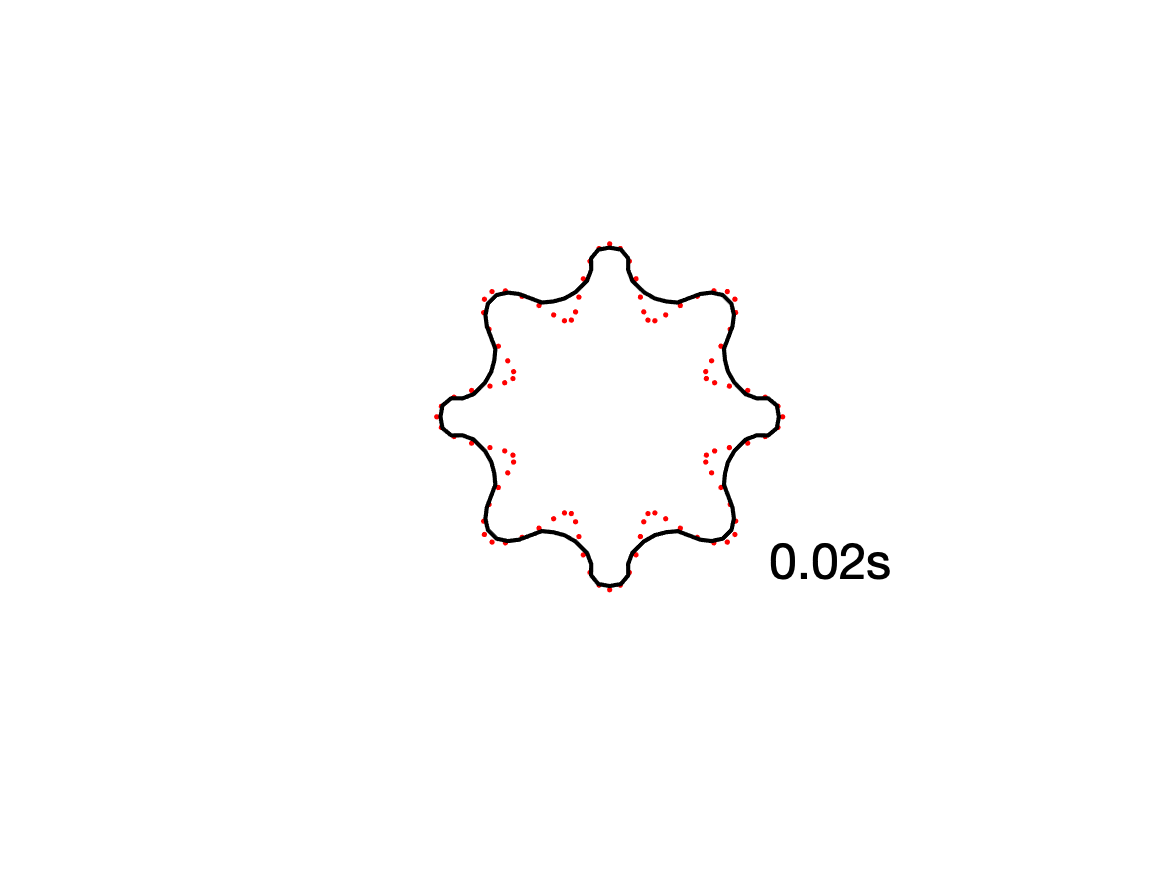} 
  \includegraphics[width = 0.2\textwidth,clip, trim = 6cm 4cm 4cm 4cm]{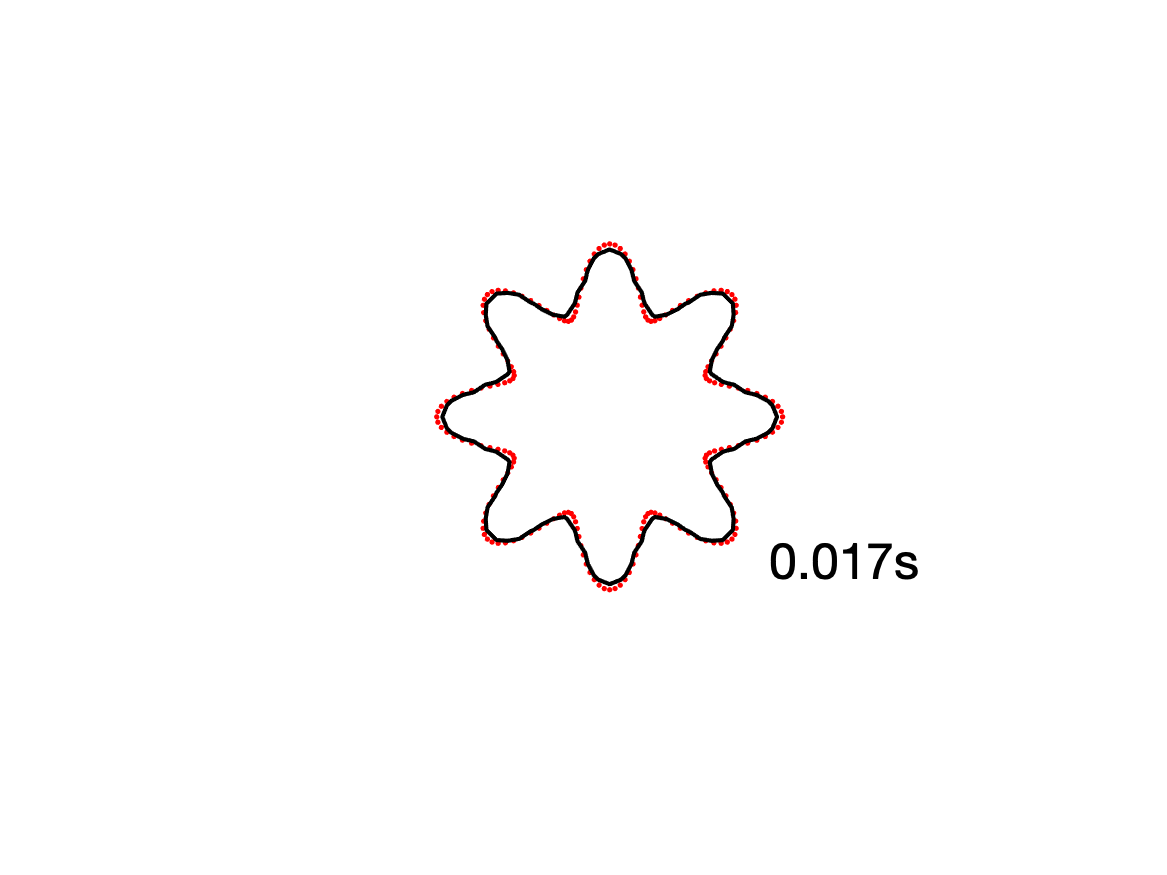} 
   \includegraphics[width = 0.2\textwidth,clip, trim = 6cm 4cm 4cm 4cm]{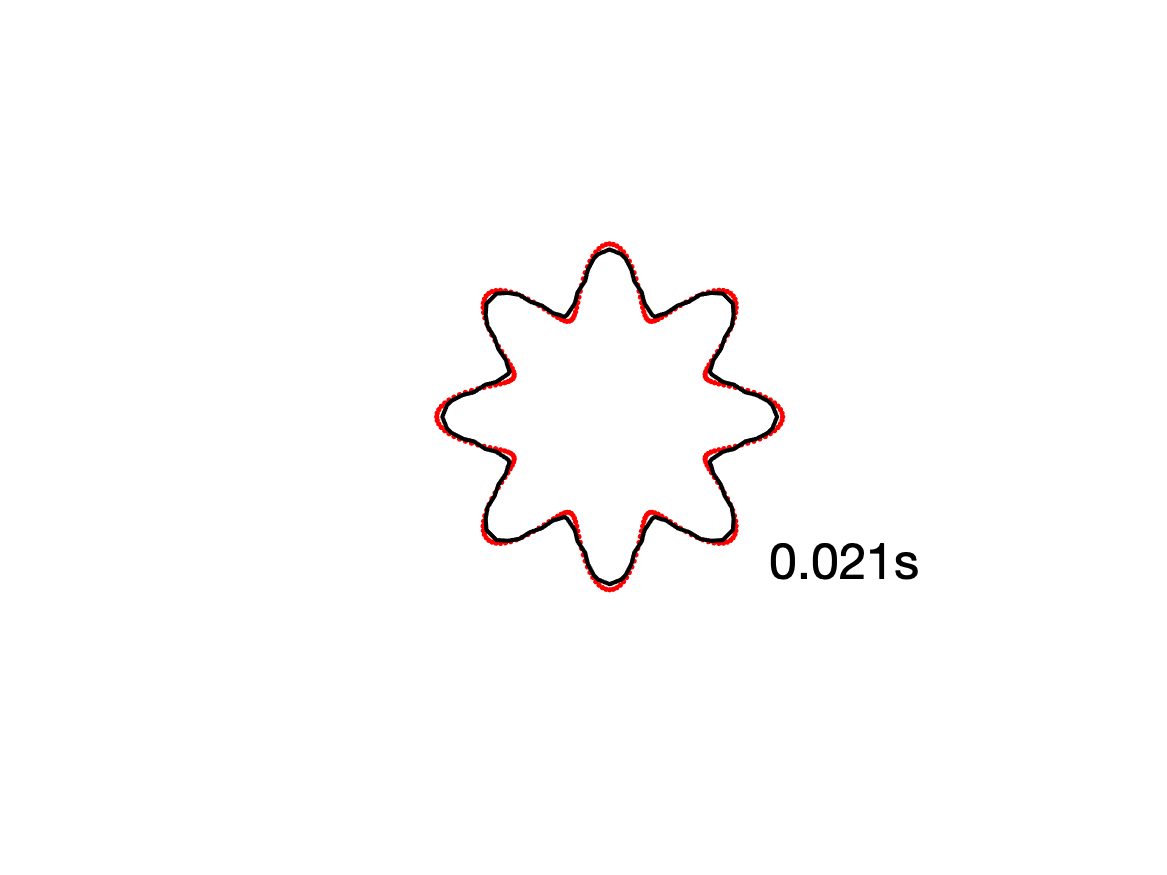} 
    \includegraphics[width = 0.2\textwidth,clip, trim = 6cm 4cm 4cm 4cm]{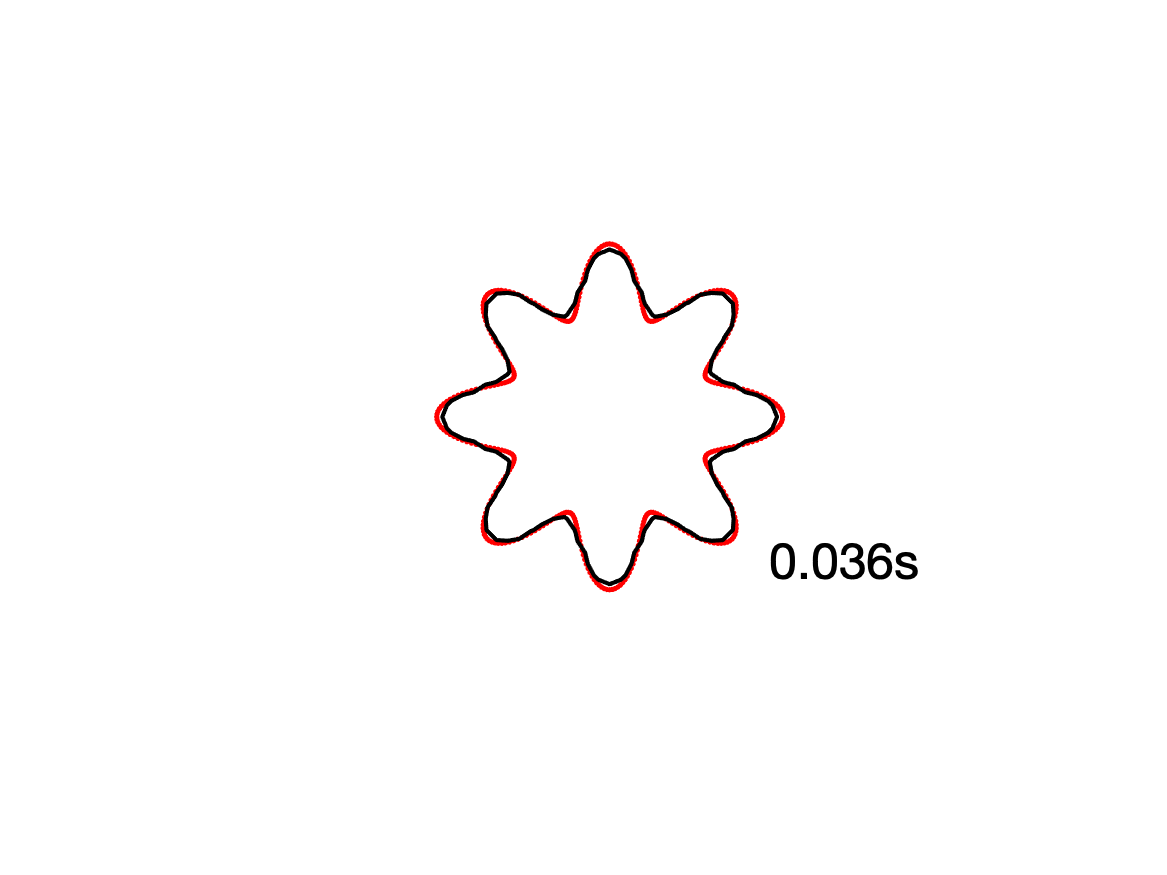}\\
 \includegraphics[width = 0.2\textwidth,clip, trim = 6cm 4cm 4cm 4cm]{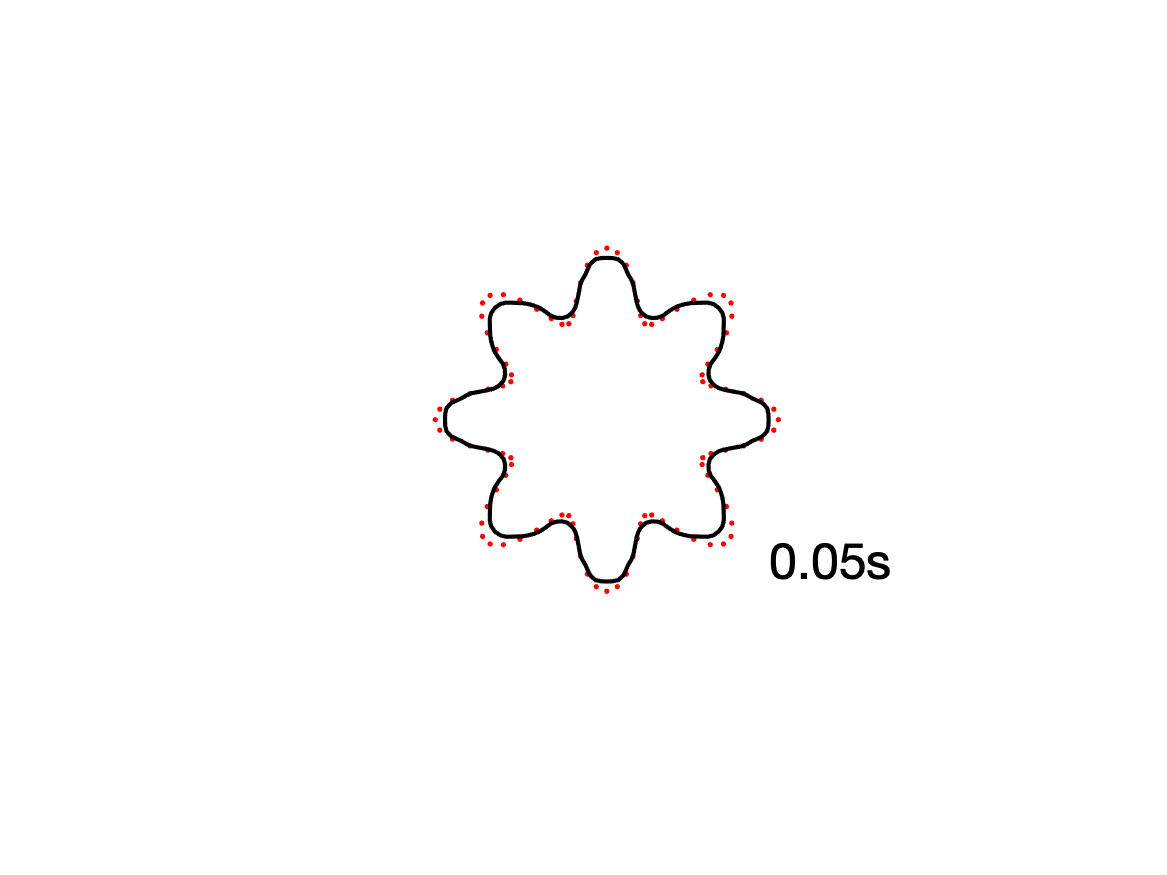} 
  \includegraphics[width = 0.2\textwidth,clip, trim = 6cm 4cm 4cm 4cm]{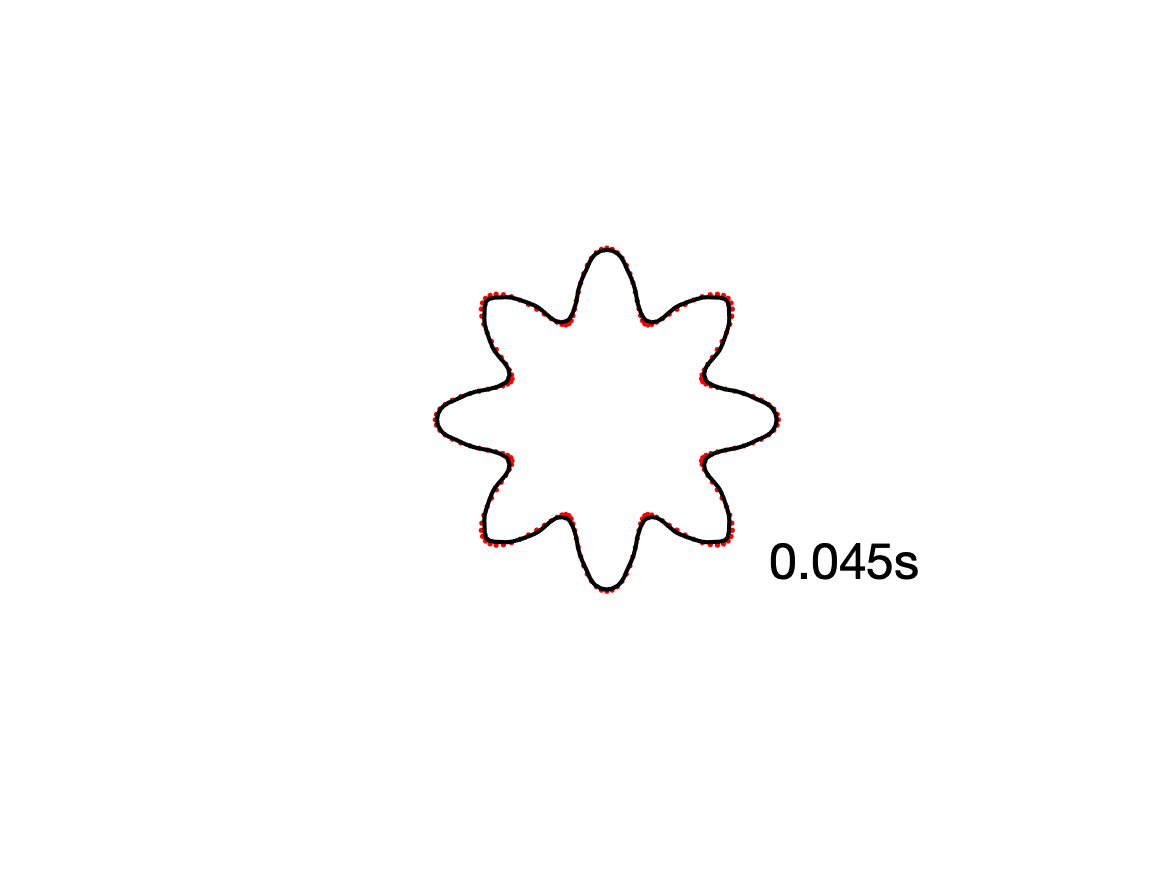} 
   \includegraphics[width = 0.2\textwidth,clip, trim = 6cm 4cm 4cm 4cm]{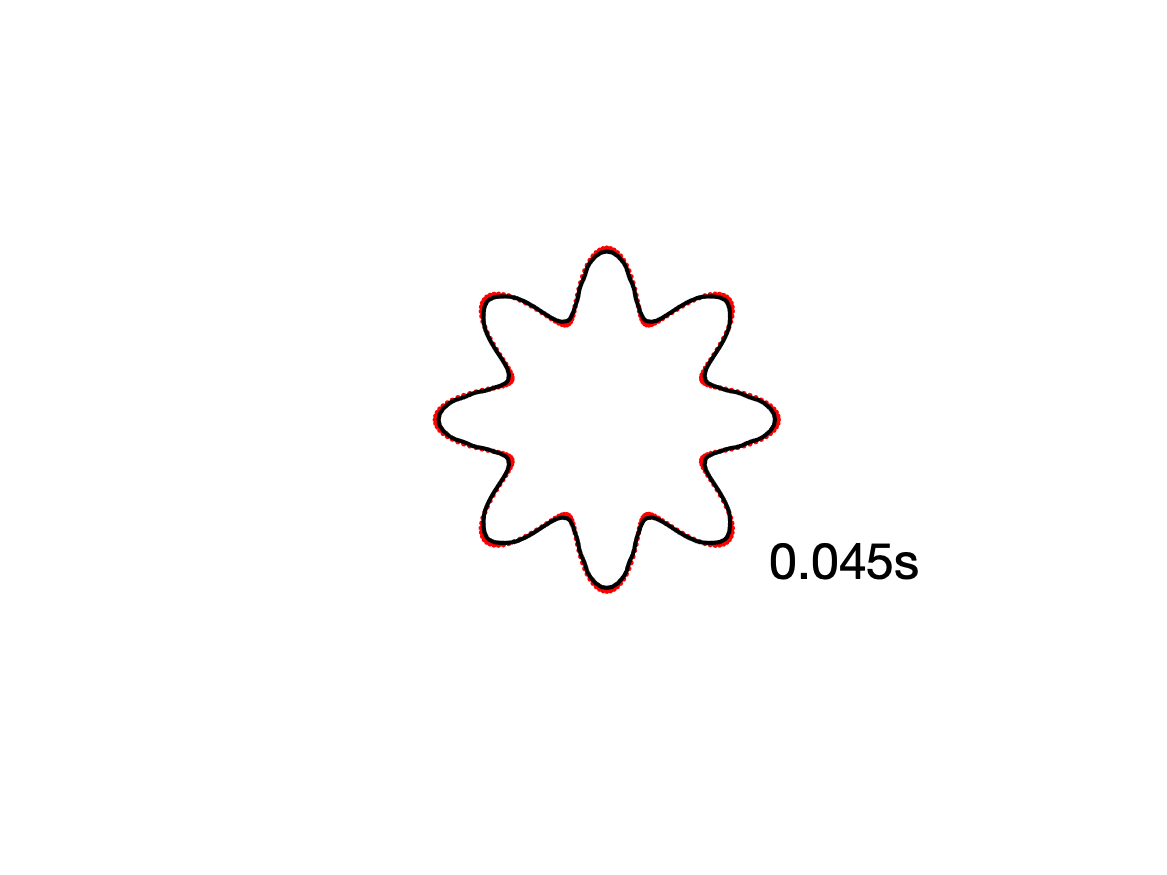} 
    \includegraphics[width = 0.2\textwidth,clip, trim = 6cm 4cm 4cm 4cm]{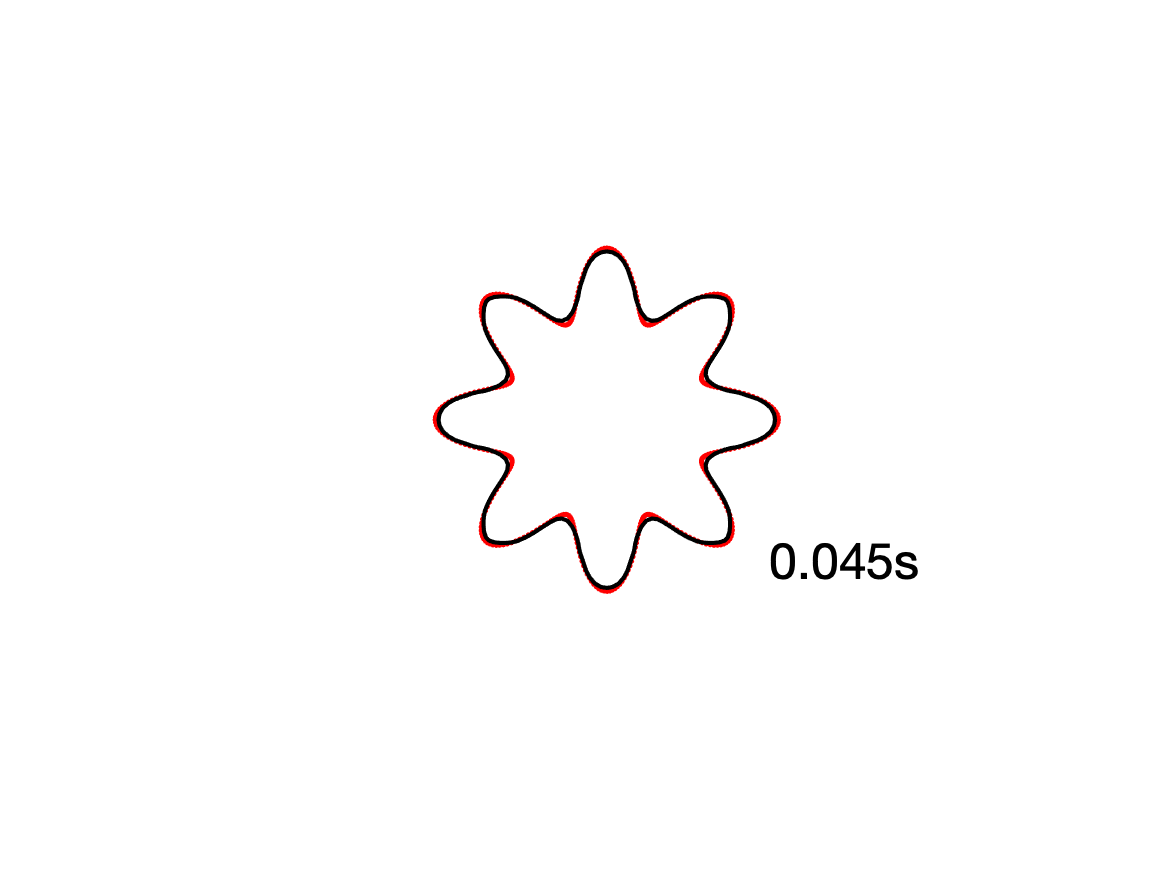}\\
 \includegraphics[width = 0.2\textwidth,clip, trim = 6cm 4cm 4cm 4cm]{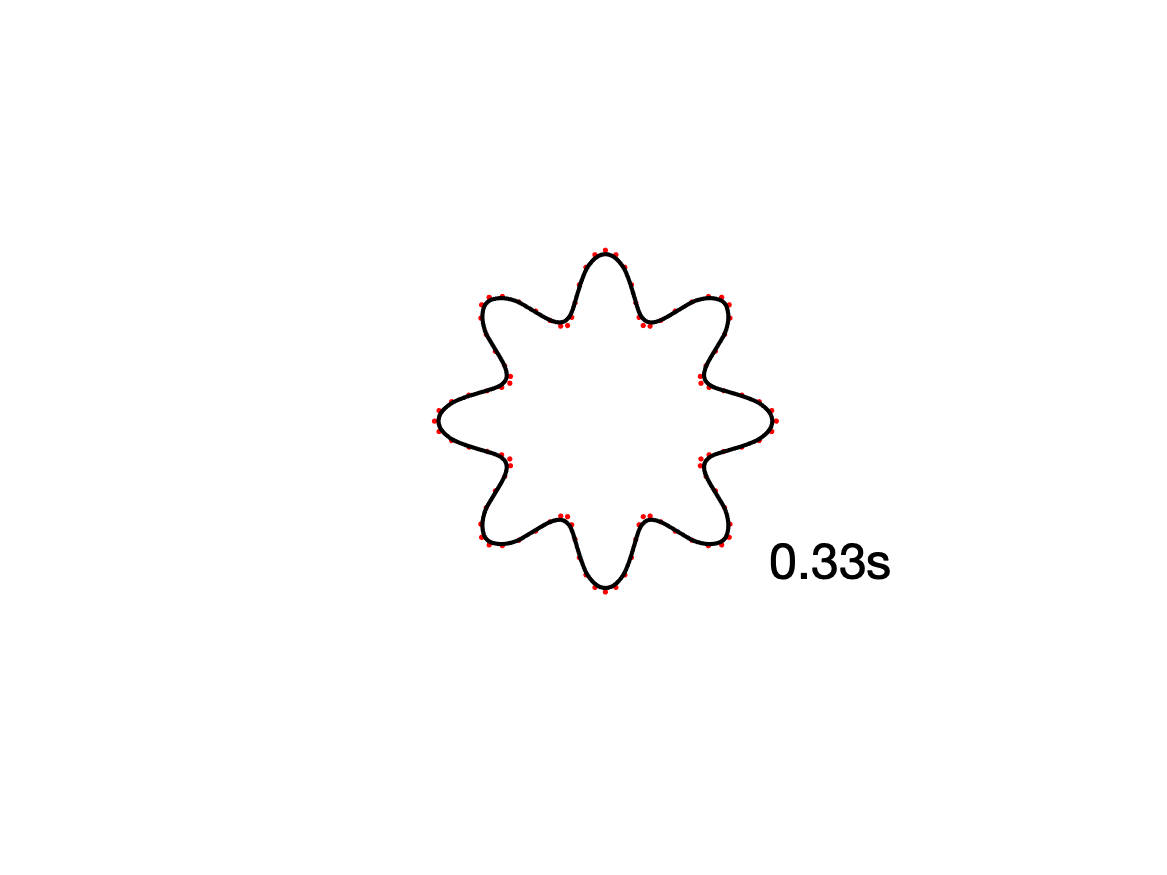} 
  \includegraphics[width = 0.2\textwidth,clip, trim = 6cm 4cm 4cm 4cm]{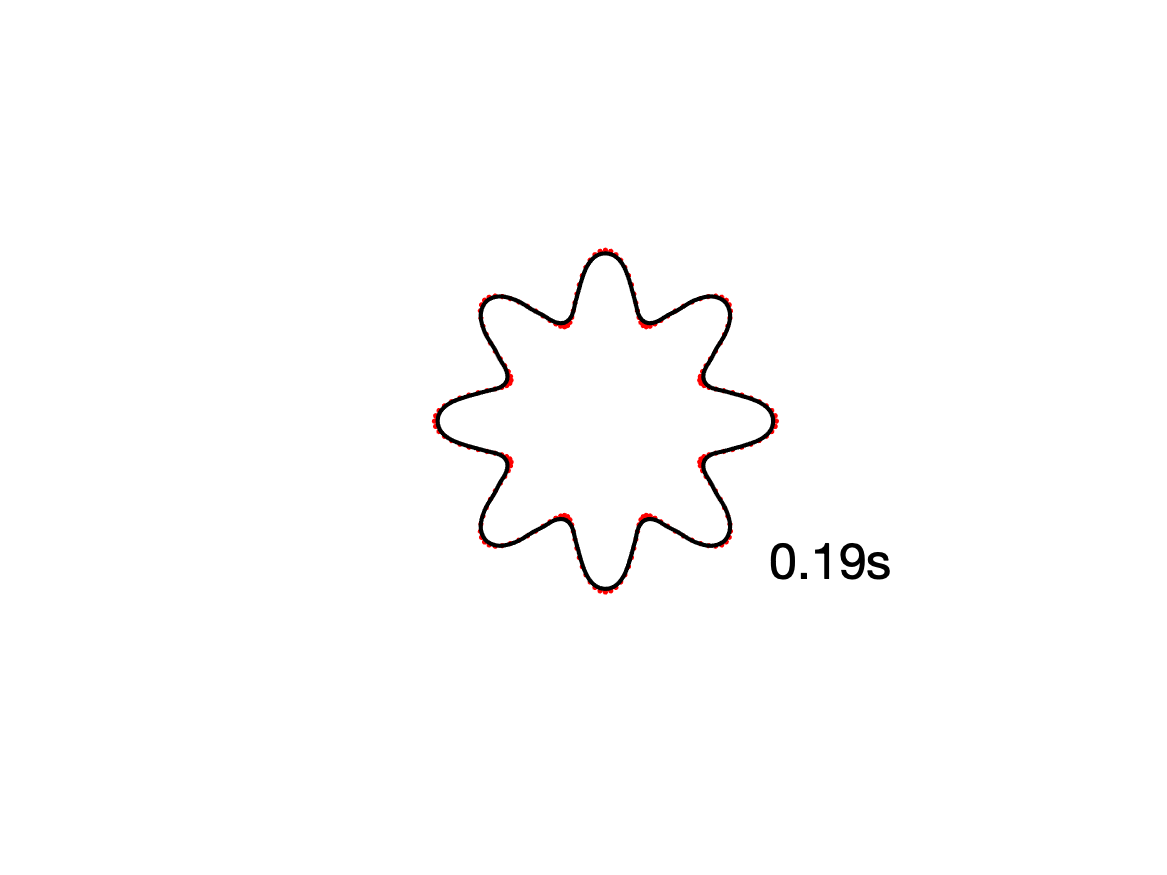} 
   \includegraphics[width = 0.2\textwidth,clip, trim = 6cm 4cm 4cm 4cm]{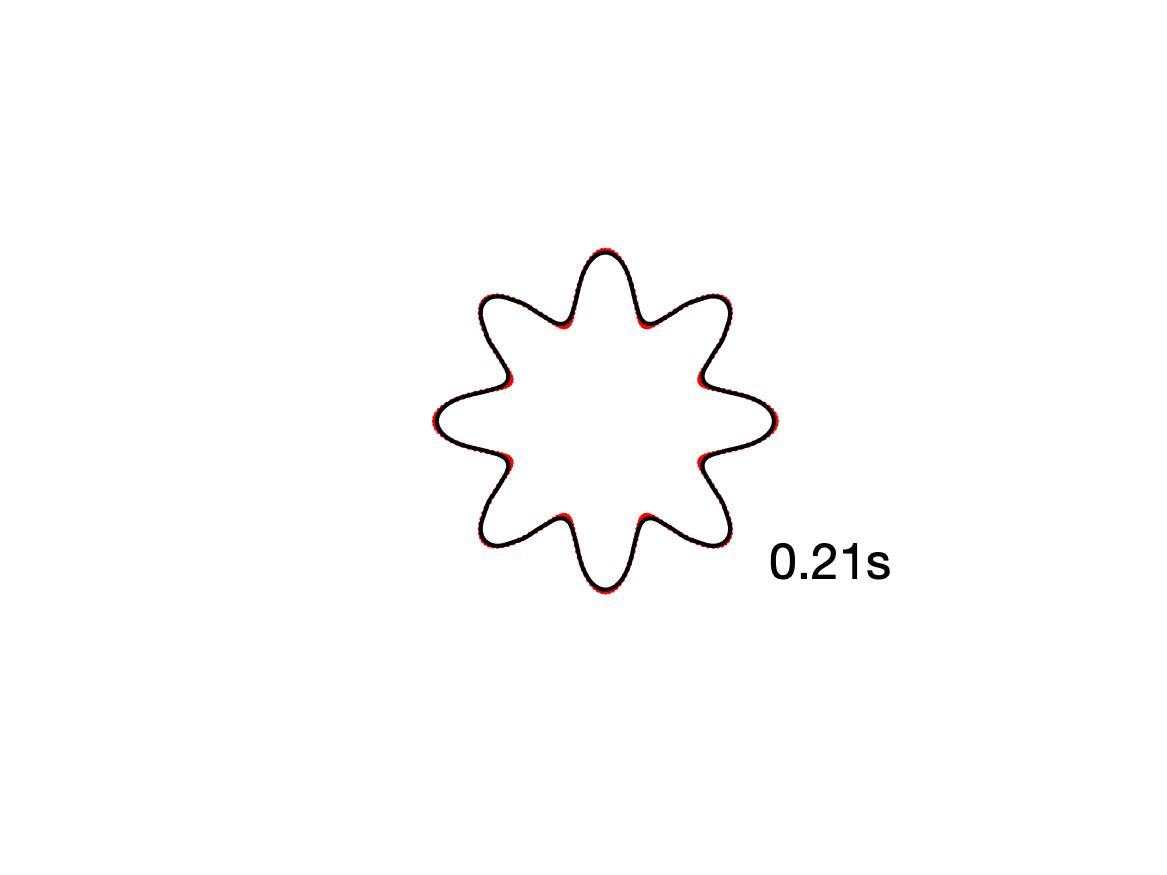} 
    \includegraphics[width = 0.2\textwidth,clip, trim = 6cm 4cm 4cm 4cm]{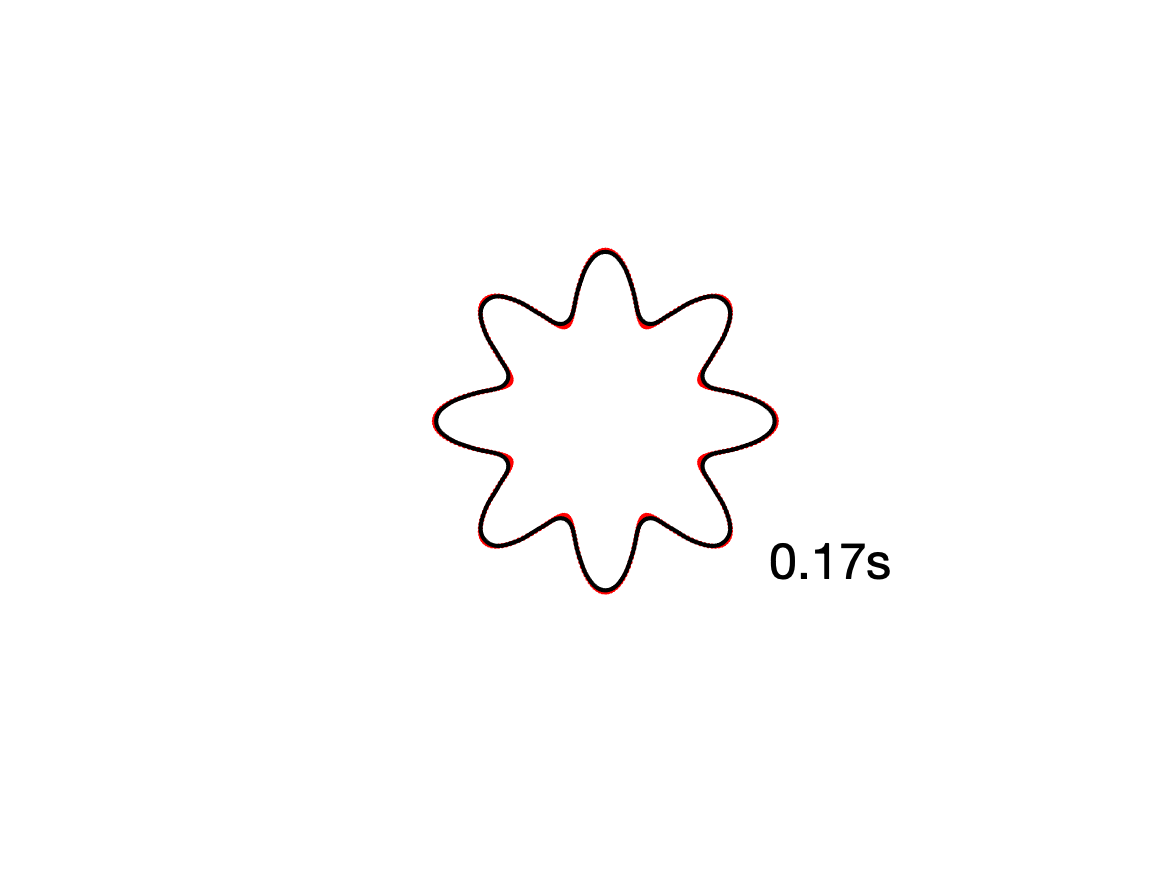}\\
 \includegraphics[width = 0.2\textwidth,clip, trim = 6cm 4cm 4cm 4cm]{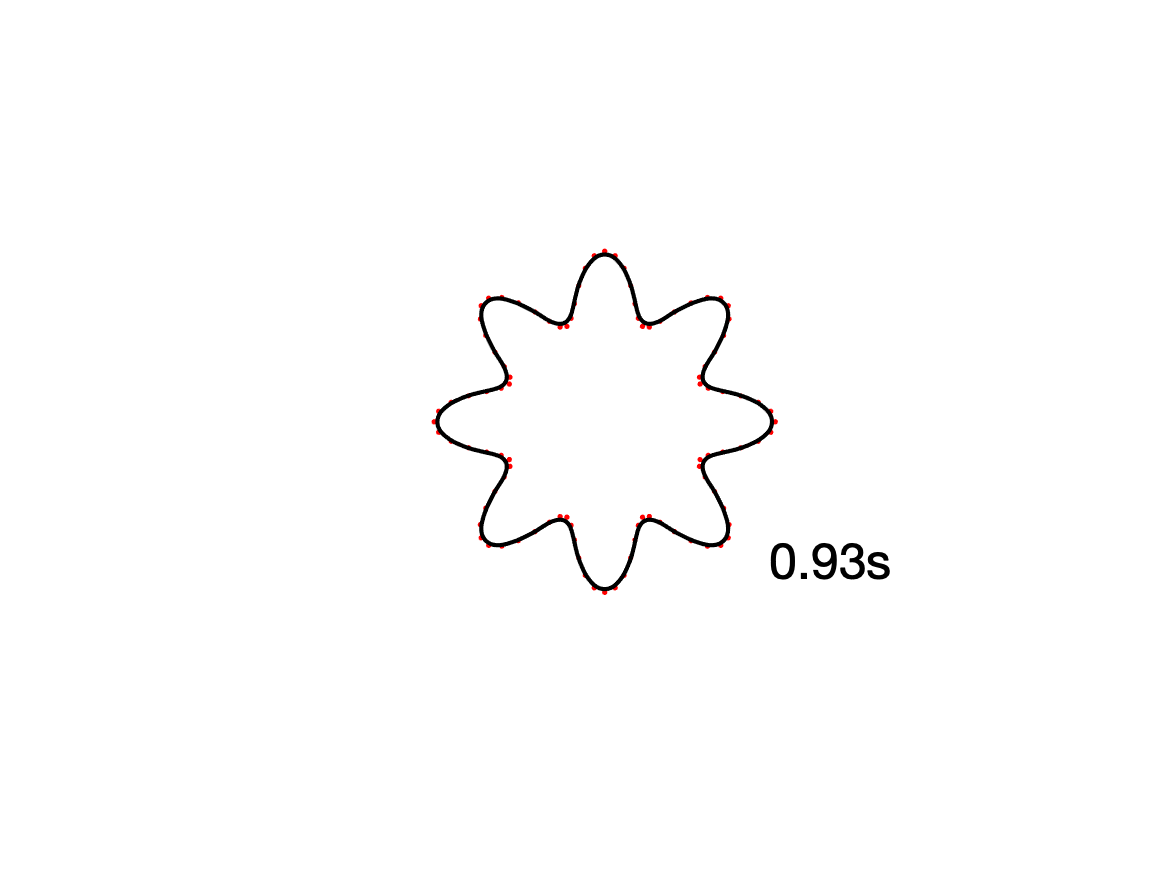} 
\includegraphics[width = 0.2\textwidth,clip, trim = 6cm 4cm 4cm 4cm]{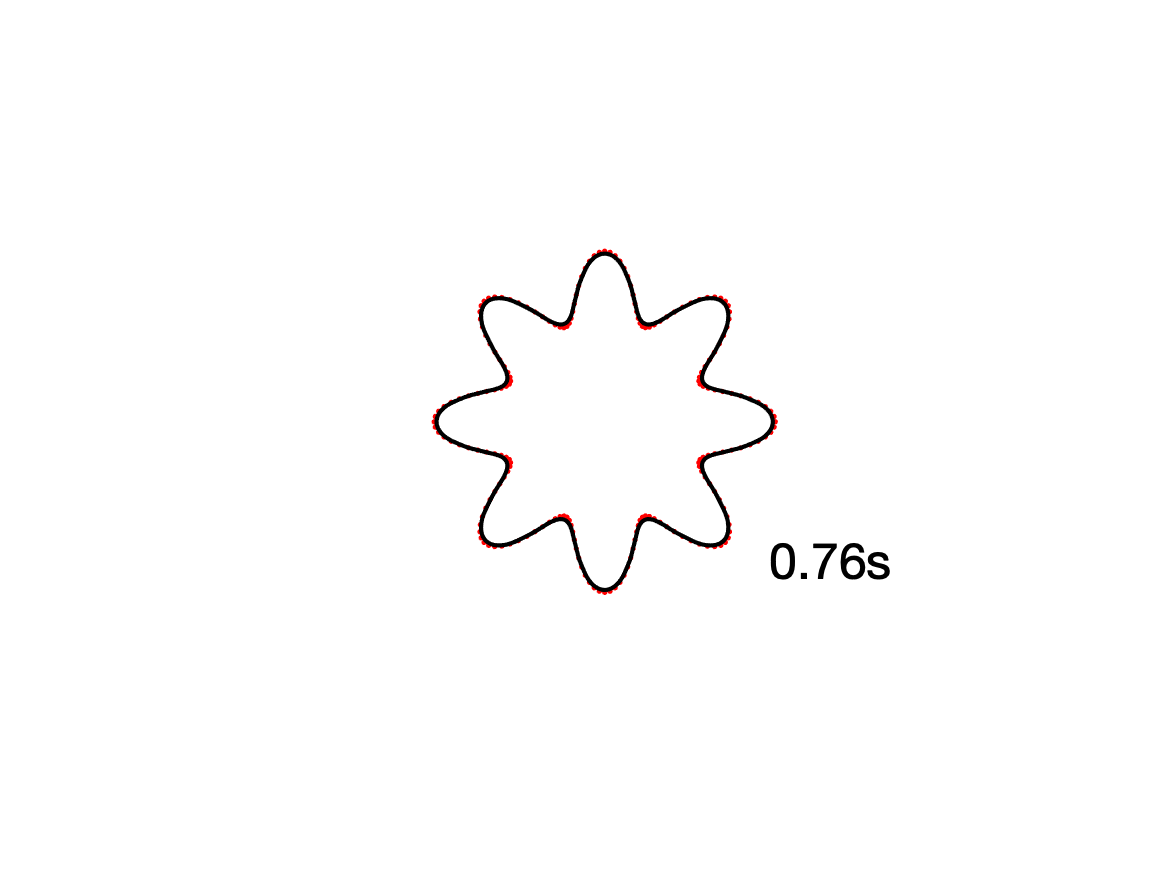} 
   \includegraphics[width = 0.2\textwidth,clip, trim = 6cm 4cm 4cm 4cm]{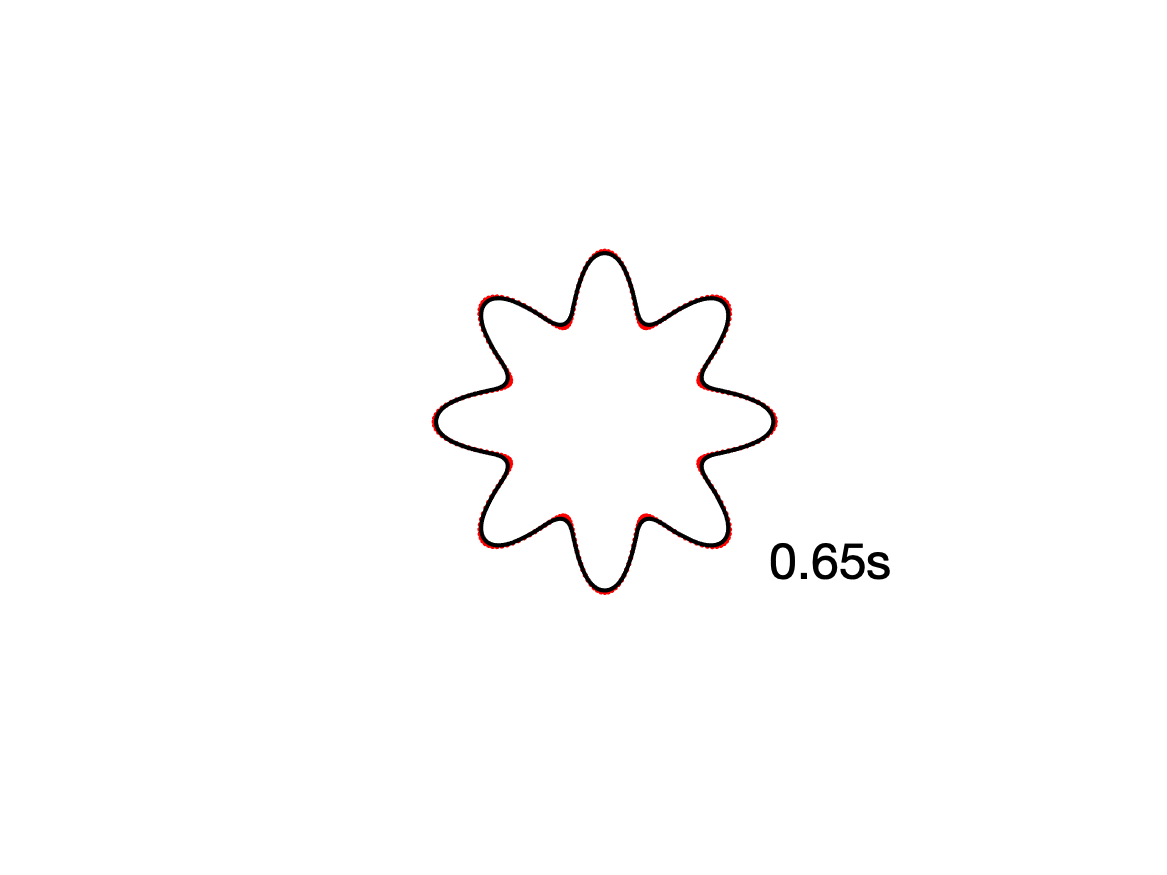} 
   \includegraphics[width = 0.2\textwidth,clip, trim = 6cm 4cm 4cm 4cm]{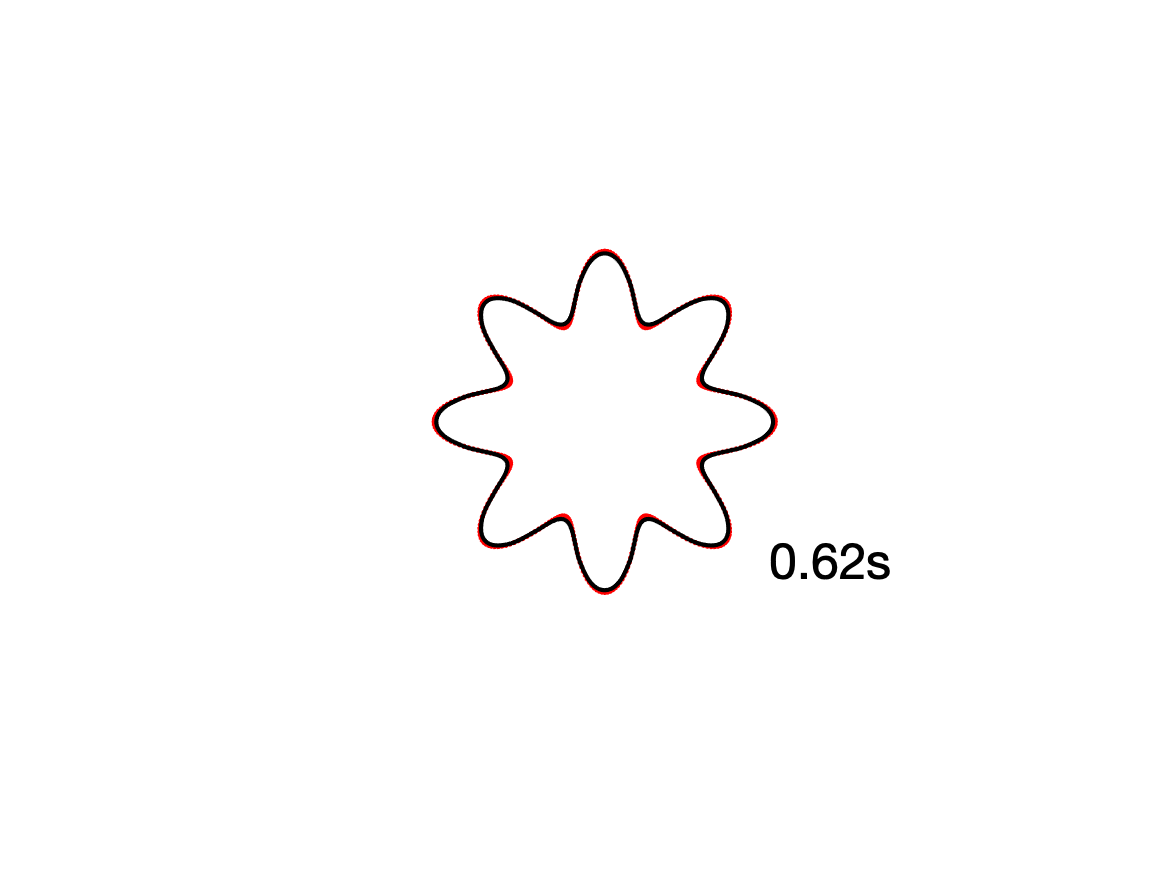} 
   \caption{Results obtained from Algorithm~\ref{a:MBO3} using different discretization of the computational domain and different sizes of point clouds. {\bf Top to bottom:} Computational domain discretized by $64^2$, $128^2$, $256^2$, and $512^2$ grids. {\bf Left to right:} Point clouds with $100$, $200$, $300$, and $400$ points. See Section~\ref{sec:sensitivity}.}\label{fig:sensitivity}
\end{figure}

\subsubsection{Efficiency comparison with level set approaches.}\label{sec:com}
In this example, we compare the efficiency between Algorithm~\ref{a:MBO3}\_\ref{a:MBO2} and level set approaches. A recent work in \cite{he2019fast} has carefully studied the efficiency comparison between the semi-implicit method (SIM) and the classic level set approach in \cite{Zhao_2000},  showing great improvement in the efficiency. Therefore, in this section, we simply compare the efficiency between Algorithm~\ref{a:MBO3}\_\ref{a:MBO2} and SIM in \cite{he2019fast}. We consider different point clouds (different $m$) generated using $N=200$ uniform points $\theta_i$ in $[0,2\pi]$:
\[\begin{cases}
x_i = r_i \cos(\theta_i), \\
y_i = r_i \sin(\theta_i),
\end{cases}
\]
where $r_i = 1+0.4\sin(m\theta_i)$; see Figure~\ref{fig:com} for the cases when $m = 3, 4, 5, 6, 7$, and $8$. To make the comparison fair, we use the same initial guess as shown in Figure~\ref{fig:com} and same discretization ($128^2$ grids) of the computational domain. The table shows the dramatical acceleration in the computational CPU time and figures indicate the improvement on the accuracy.  

\begin{figure}[ht!]
\centering
\includegraphics[width = 0.3\textwidth,clip, trim = 4cm 2cm 4cm 1.5cm]{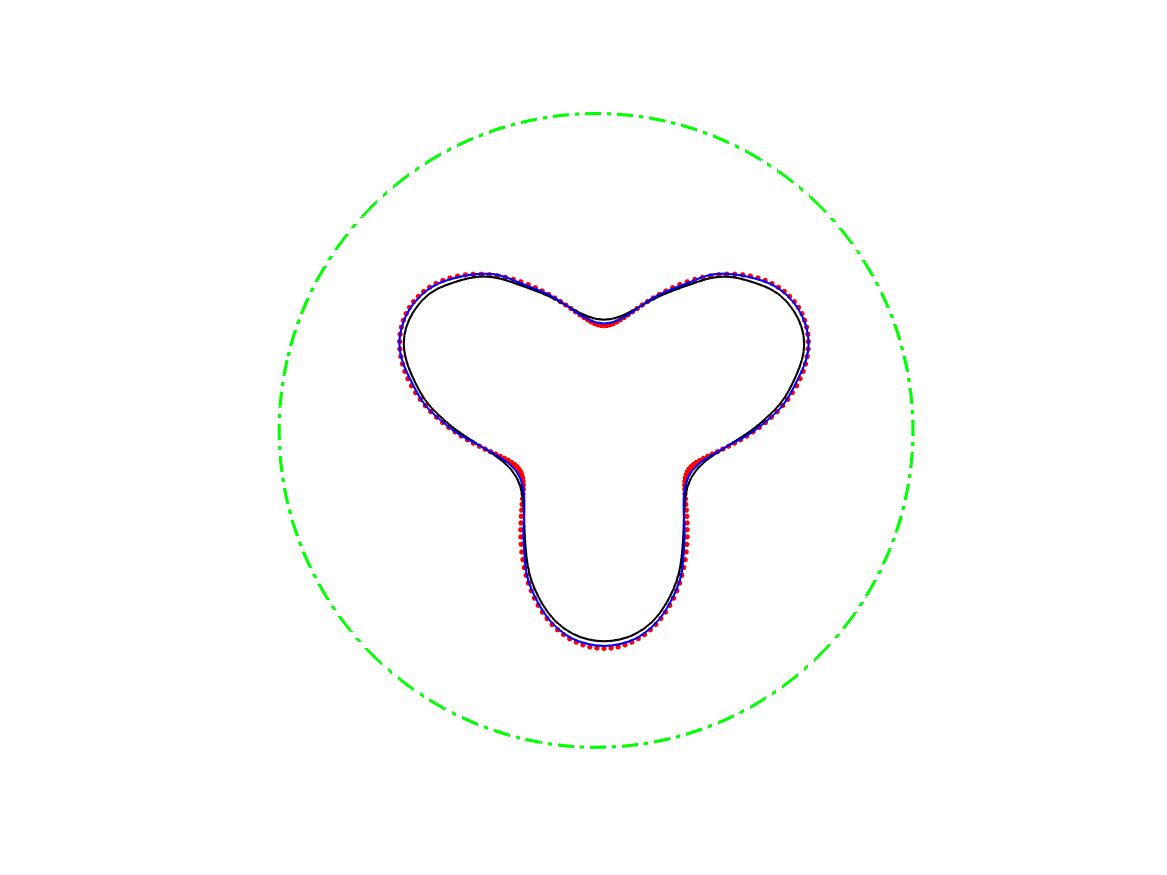}
\includegraphics[width = 0.3\textwidth,clip, trim = 4cm 2cm 4cm 1.5cm]{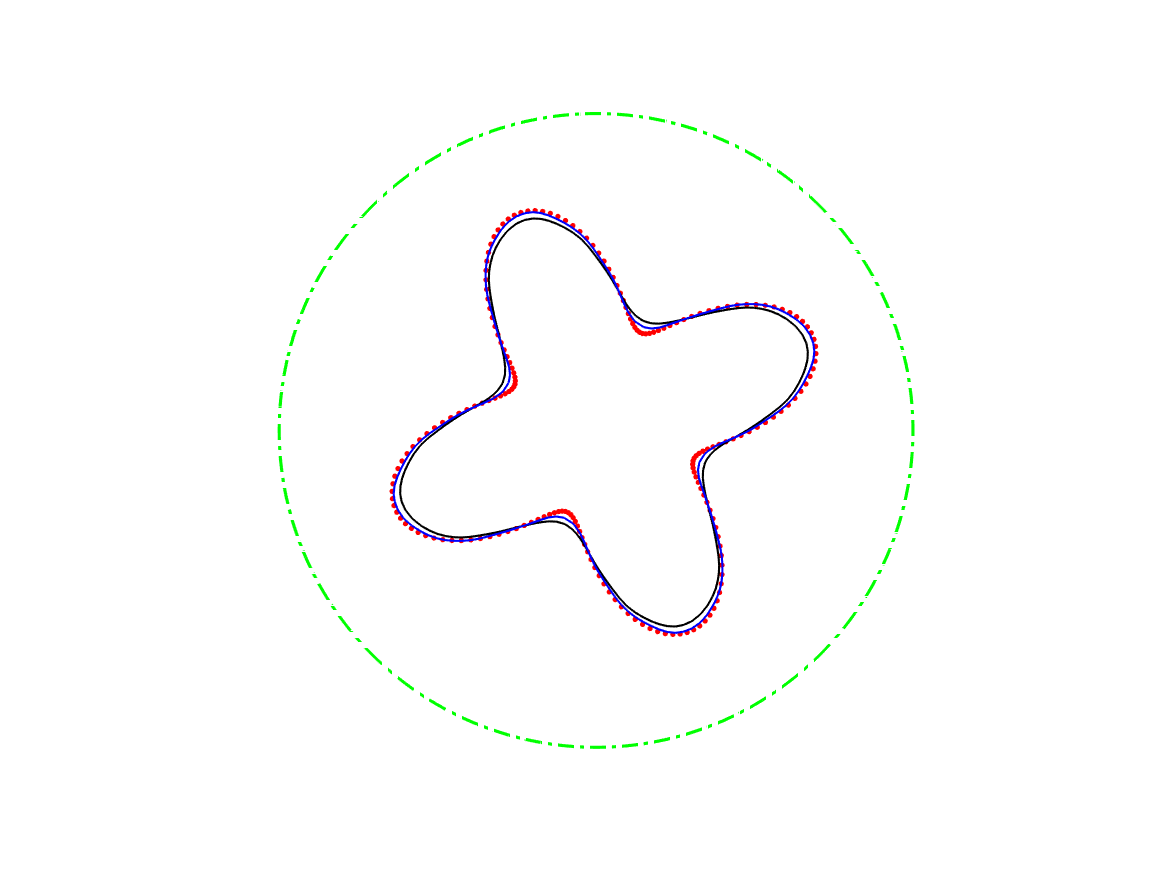}
\includegraphics[width = 0.3\textwidth,clip, trim = 4cm 2cm 4cm 1.5cm]{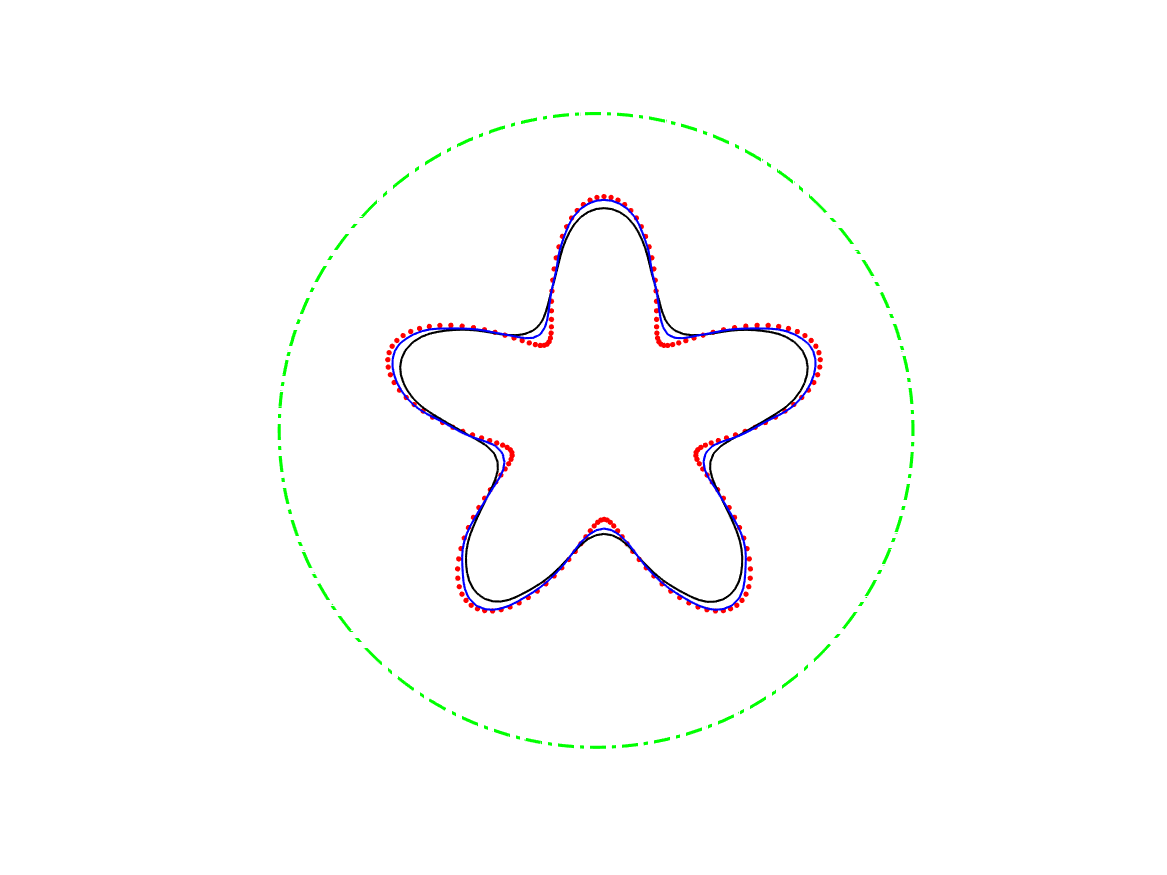}
\includegraphics[width = 0.3\textwidth,clip, trim = 4cm 2cm 4cm 1.5cm]{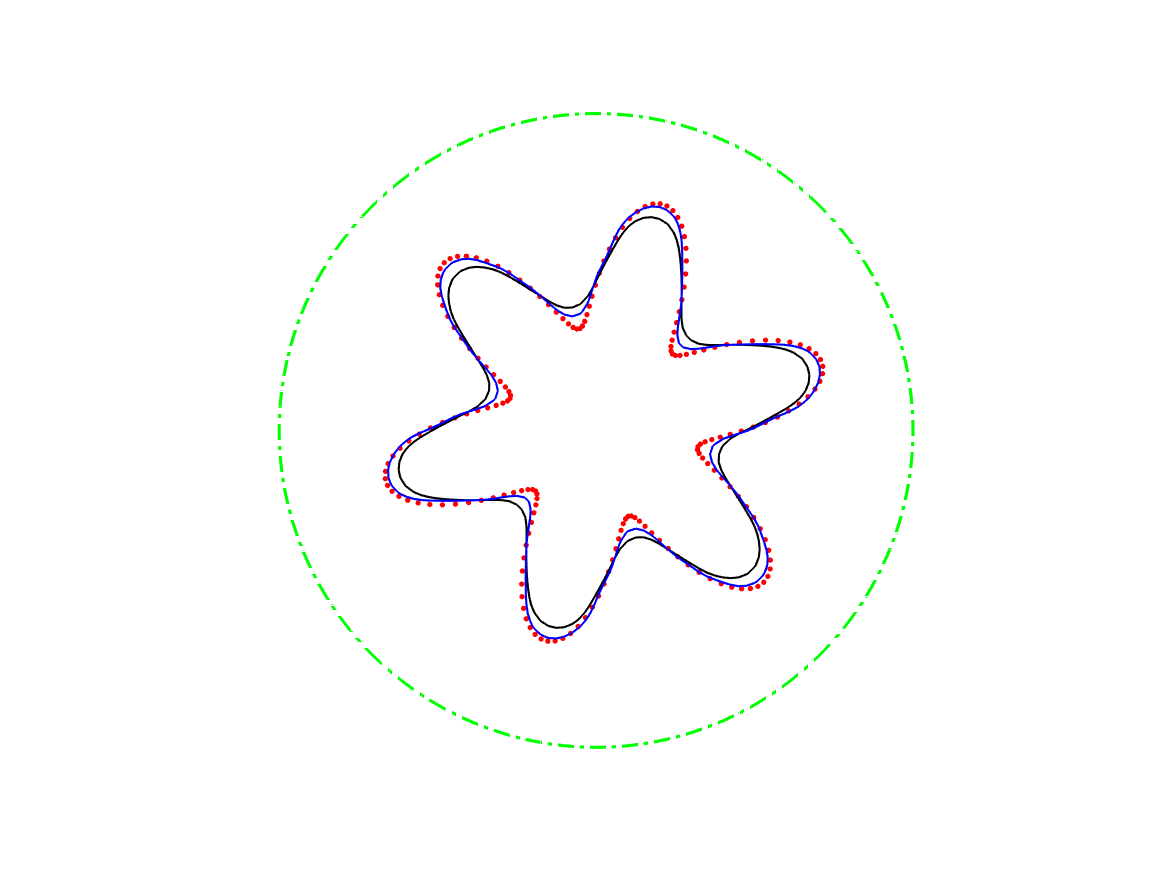}
\includegraphics[width = 0.3\textwidth,clip, trim = 4cm 2cm 4cm 1.5cm]{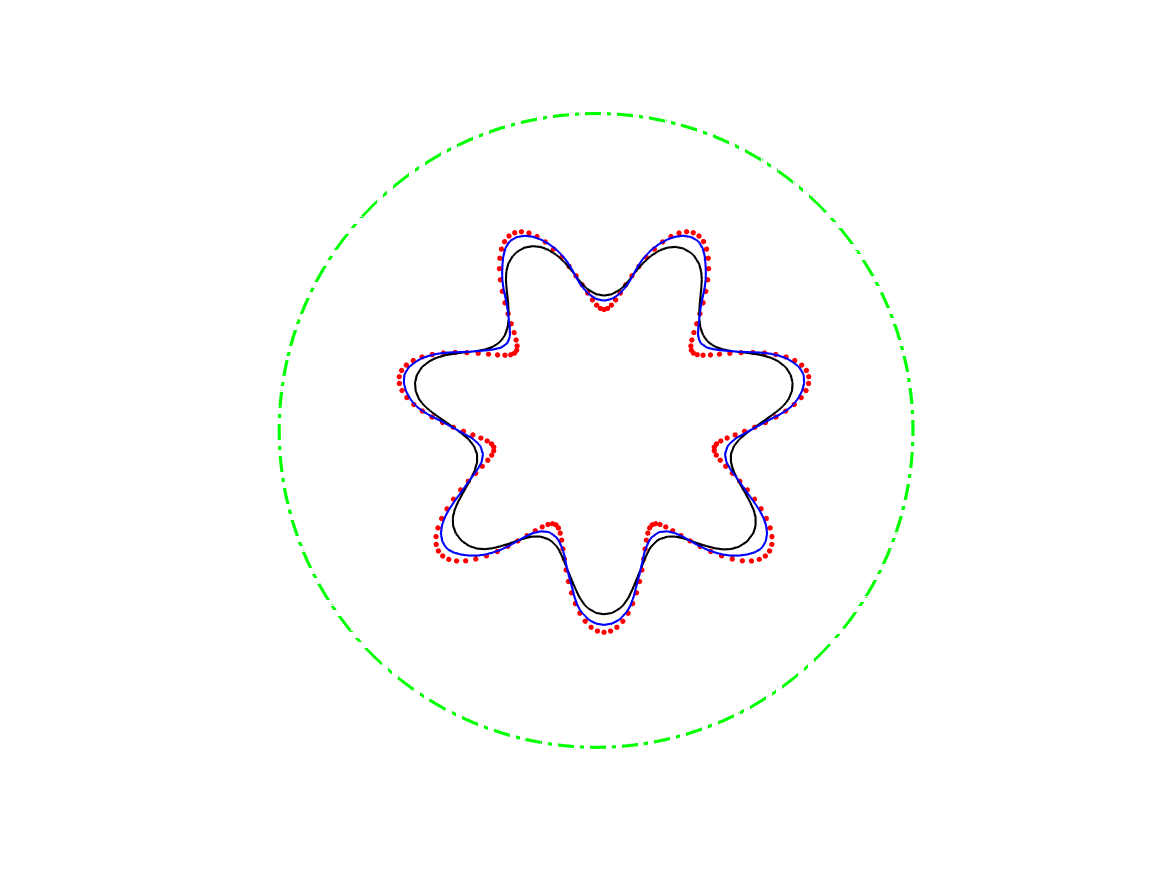}
\includegraphics[width = 0.3\textwidth,clip, trim = 4cm 2cm 4cm 1.5cm]{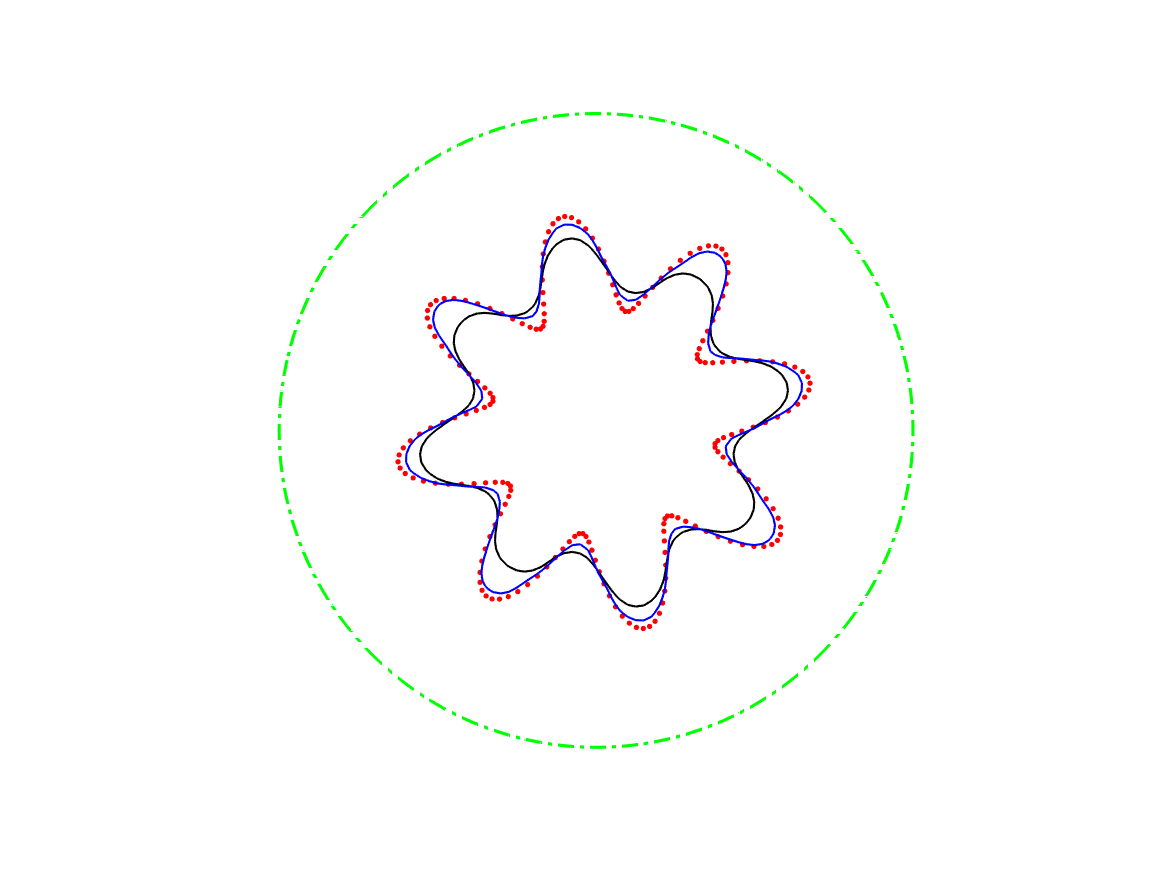}

\medskip

\begin{tabular}{c|c|c|c|c|c|c}
SIM in \cite{he2019fast} & $3.2$ s &  $4.47$ s & $5.26$ s & $6.40$ s & $7.18$ s & $4.56$ s  \\
\hline
\hline
Algorithm~\ref{a:MBO3}\_\ref{a:MBO2}  & $0.029$ s & $0.036$ s & $0.029$ s & $0.026$ s & $0.031$ s & $0.030$ s
\end{tabular}
\caption{Comparisons between Algorithm~\ref{a:MBO3}\_\ref{a:MBO2} and SIM in \cite{he2019fast}. {\bf Blue curve:} results obtained from Algorithm~\ref{a:MBO3}. {\bf Black curve:} results obtained from SIM in \cite{he2019fast}. {\bf Green curve:} initial guess. {\bf Red points}: point cloud with $200$ points. {\bf Table:} CPU times for the computations. See Section~\ref{sec:com}.} \label{fig:com}
\end{figure}

\subsubsection{3-dimensional examples.}\label{sec:3dex}
In the last example, we show the performance of the proposed algorithms in reconstructing 2-dimensional surface from the 3-dimensional point cloud. In the follows, we choose the $\sigma$-level set from $d(\bx)$ as the initial guess of the Algorithm with a relatively large $\sigma$. We note that $\sigma$ can not be very small because the $\sigma$-level set are then small balls around each point, similar to those in Figure~\ref{fig:distance_diag}.

Figure~\ref{fig:3d} displays the reconstructed results from different point clouds: noisy torus, two tori, bumpy torus, bunny, and pig. The computational domain $[-\pi,\pi]^3$ is discretized by $128^3$ grids. We observe that for most parts in these point clouds, they are well reconstructed especially when the curvature is not that big.  However, when the curvature is very big, it is very difficult to touch (\eg, the ears of the bunny and pig) because of the diffusion effect. This is consistent with the dynamical motion law we derived in Section~\ref{sec:connection}. We expect that this could be improved by considering some more terms (for example, a curvature term as that in \cite{he2020curvature}) in the objective energy. The CPU time of each experiment is printed on each figure respectively. All of them are done in seconds on $128^3$ grids, indicating the efficiency of the method. 

\begin{figure}[ht!]
\centering
\includegraphics[width = 0.2\textwidth, clip, trim = 4.5cm 2.5cm 1.5cm 2cm]{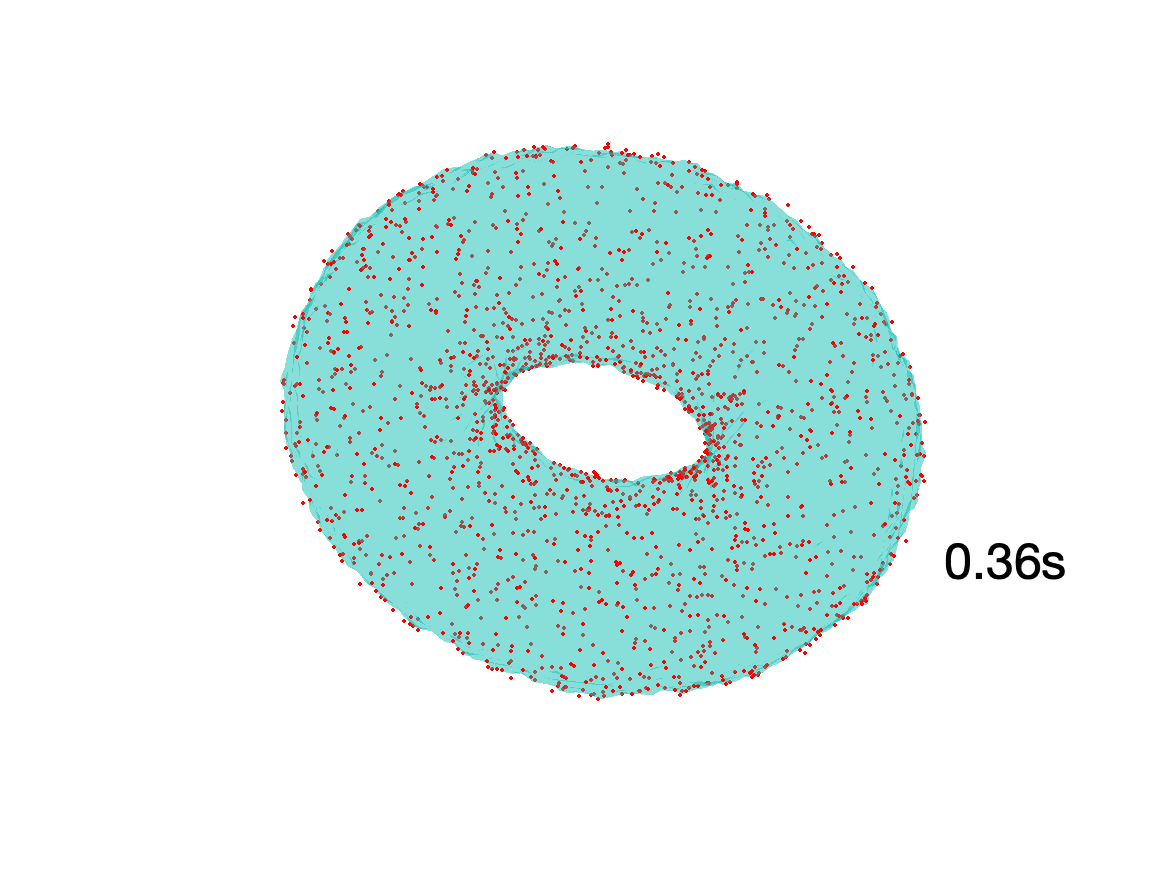} \  \ 
\includegraphics[width = 0.2\textwidth, clip, trim = 3.5cm 1.5cm 3cm 1cm]{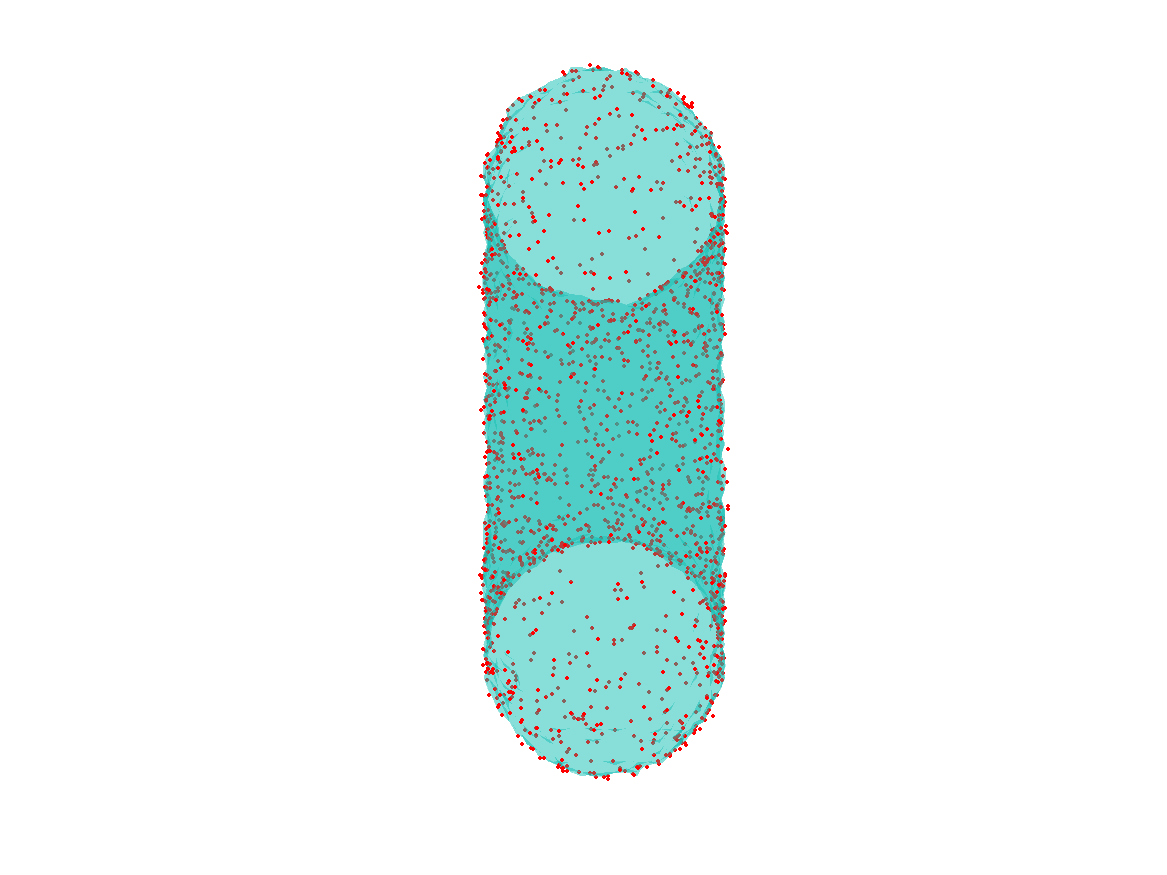} \  \ 
\includegraphics[width = 0.2\textwidth, clip, trim = 4cm 1.5cm 1.5cm 1cm]{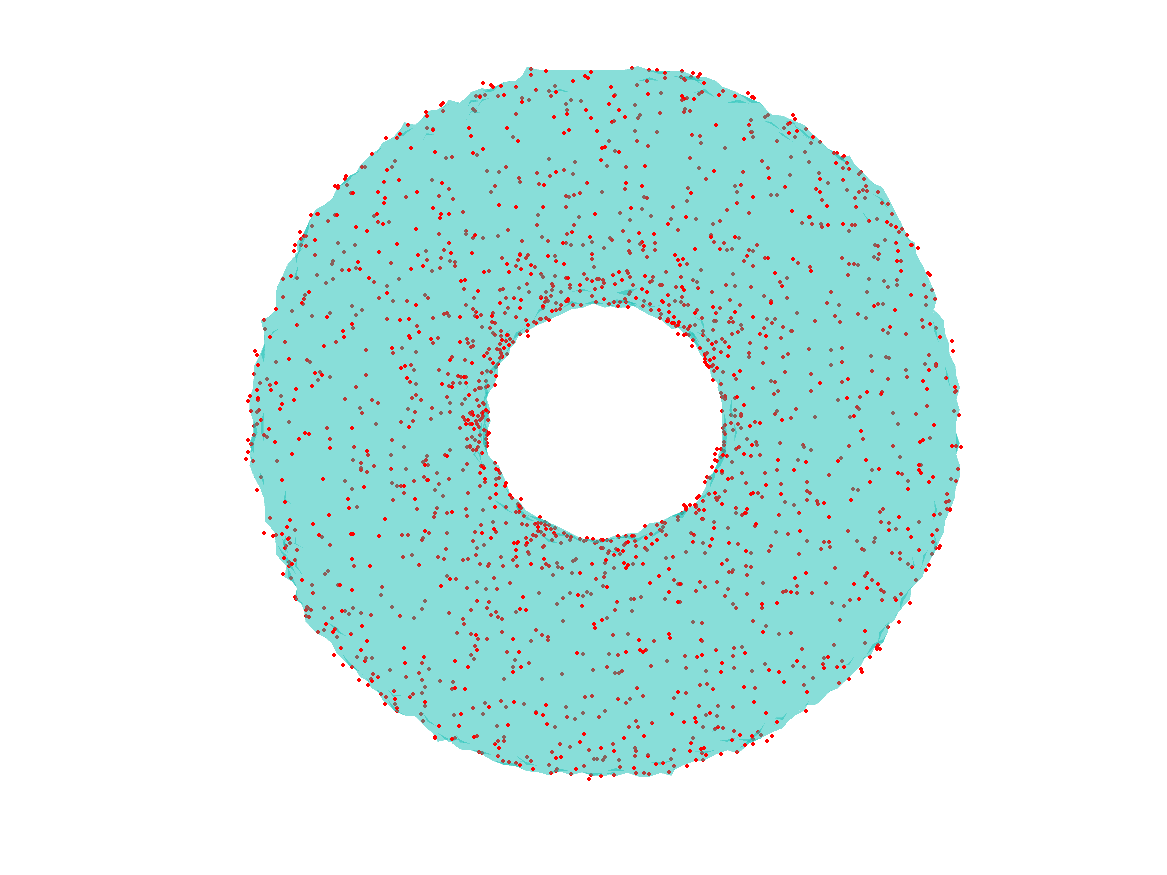} \  \
\includegraphics[width = 0.2\textwidth, clip, trim = 3.5cm 1.5cm 3cm 1cm]{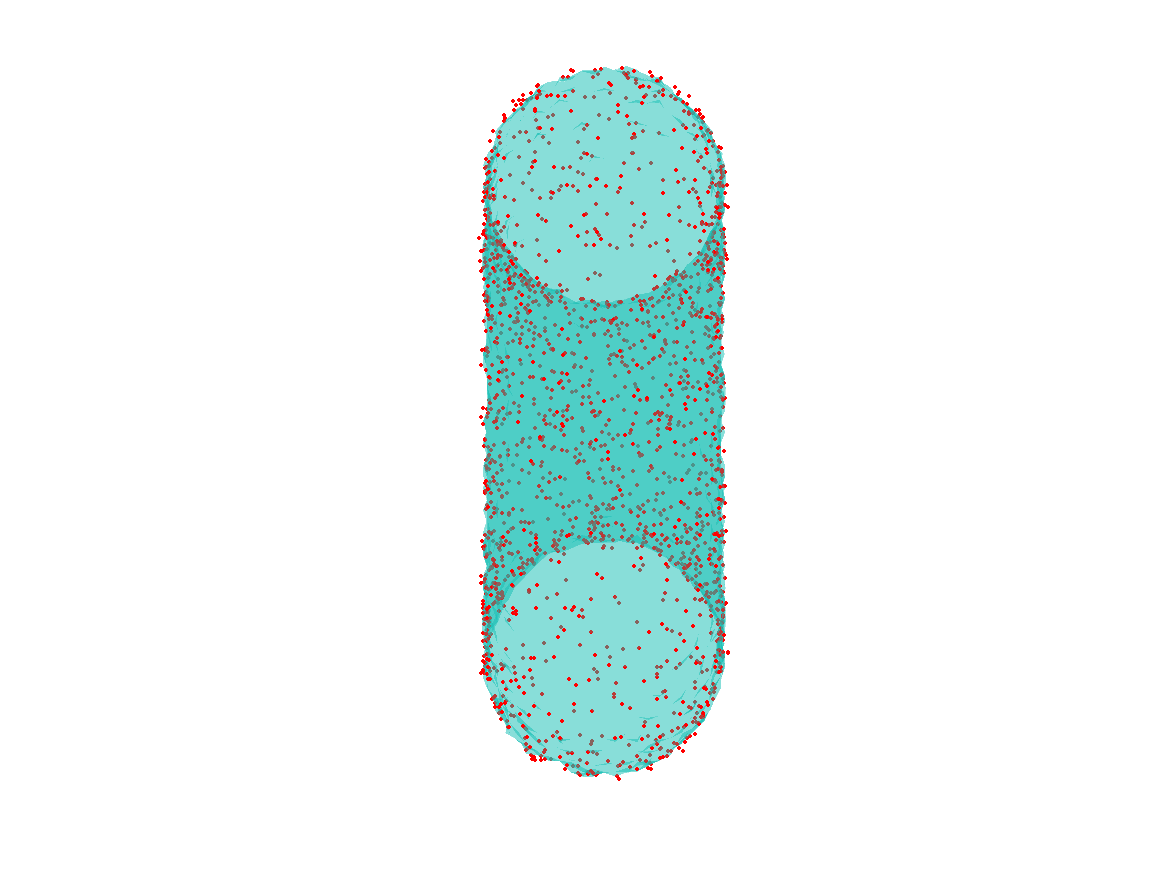} \\
\includegraphics[width = 0.2\textwidth, clip, trim = 3.5cm 2.5cm 1.5cm 2cm]{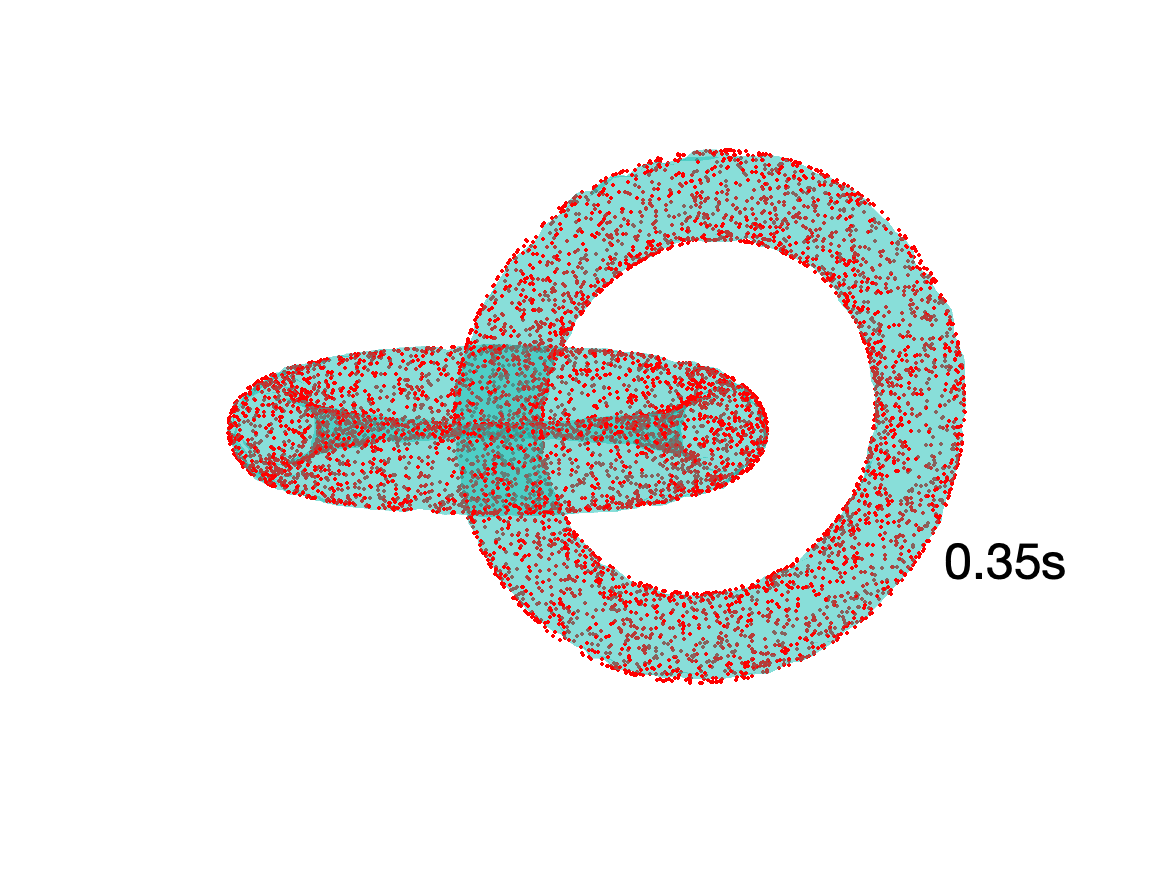} \  \
\includegraphics[width = 0.2\textwidth, clip, trim = 2.5cm 1.5cm 1.5cm 1cm]{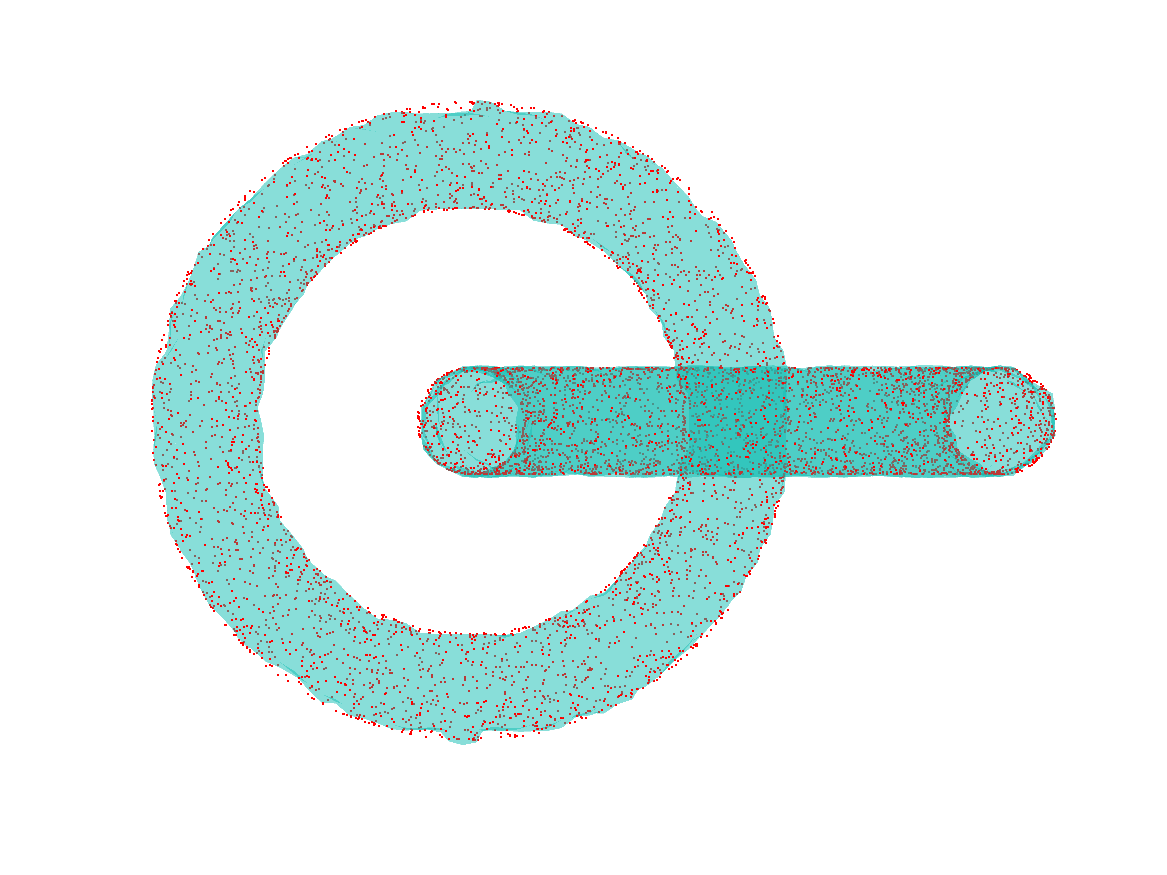}  \ \ 
\includegraphics[width = 0.2\textwidth, clip, trim = 4cm 1.5cm 1.5cm 1cm]{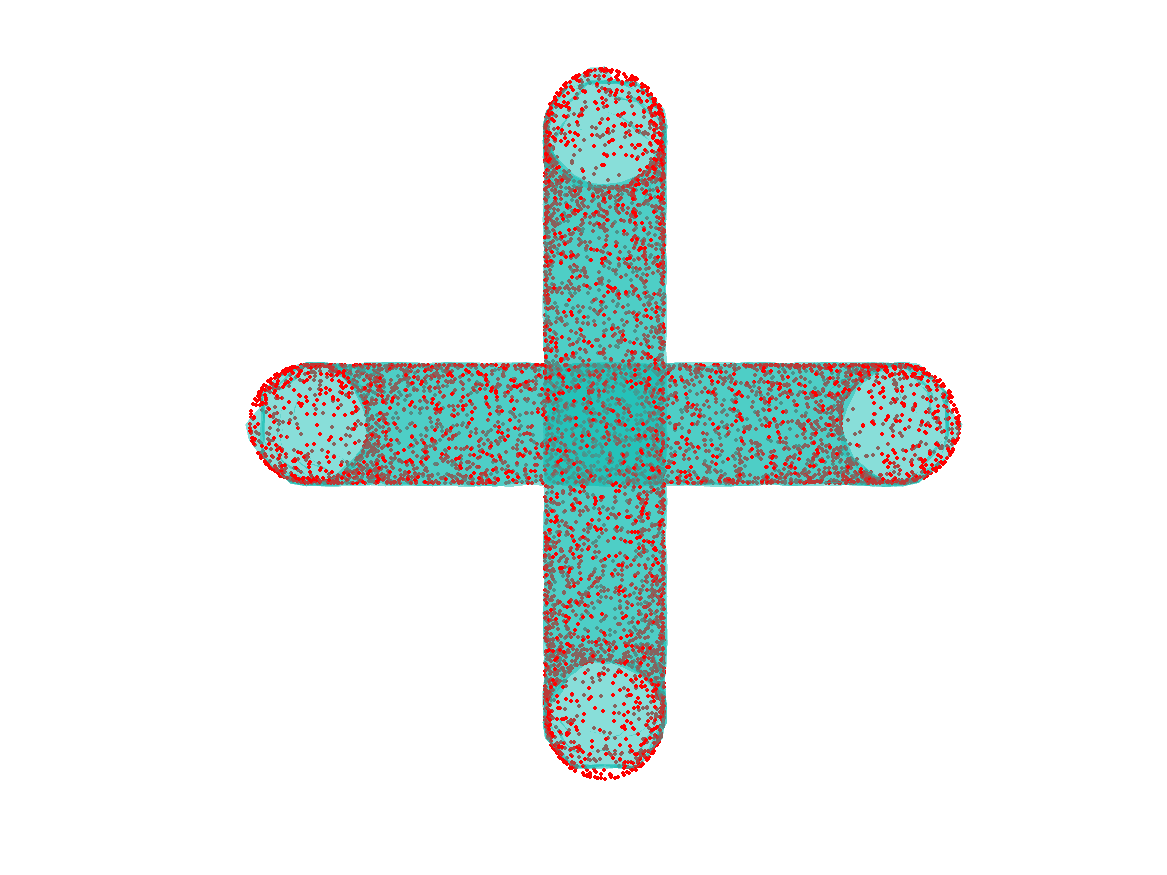} \  \
\includegraphics[width = 0.2\textwidth, clip, trim = 2.5cm 1.5cm 1.5cm 1cm]{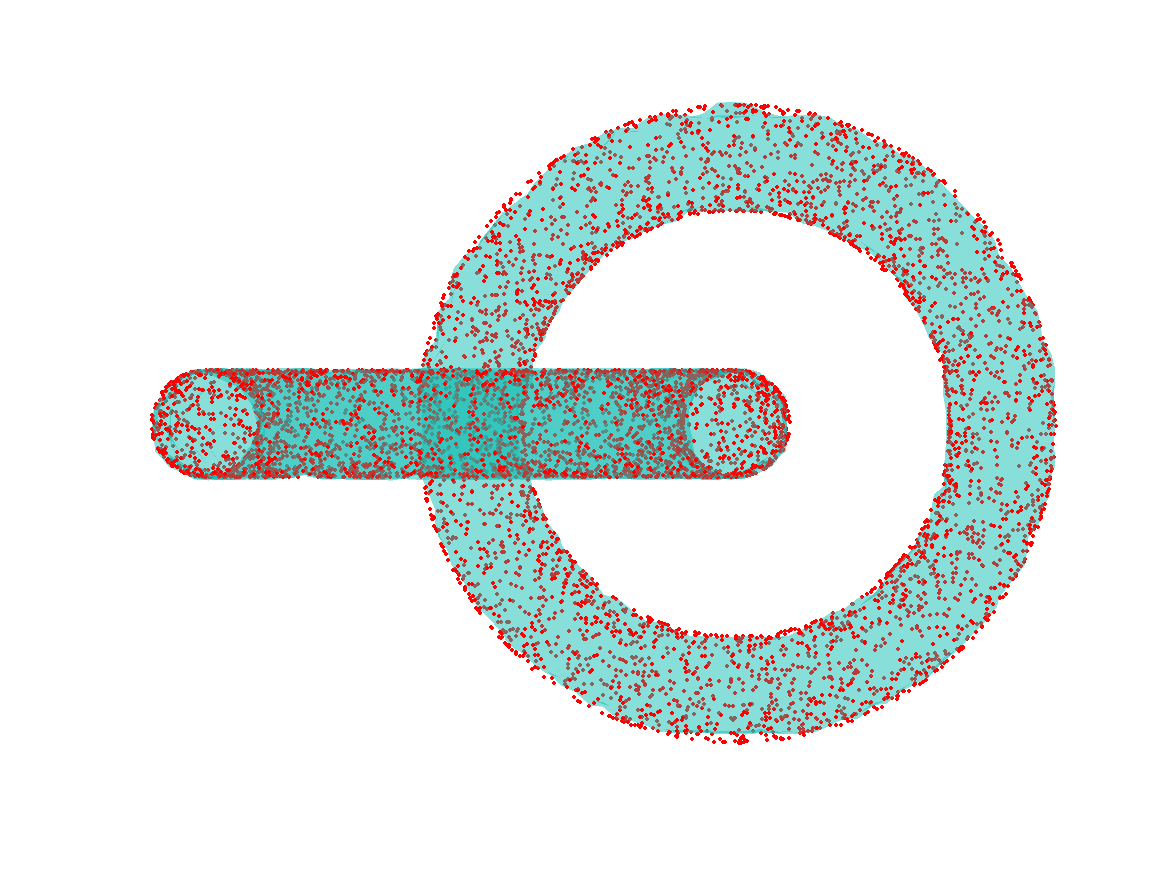} \\
\includegraphics[width = 0.2\textwidth, clip, trim = 4.5cm 2.5cm 2.5cm 2cm]{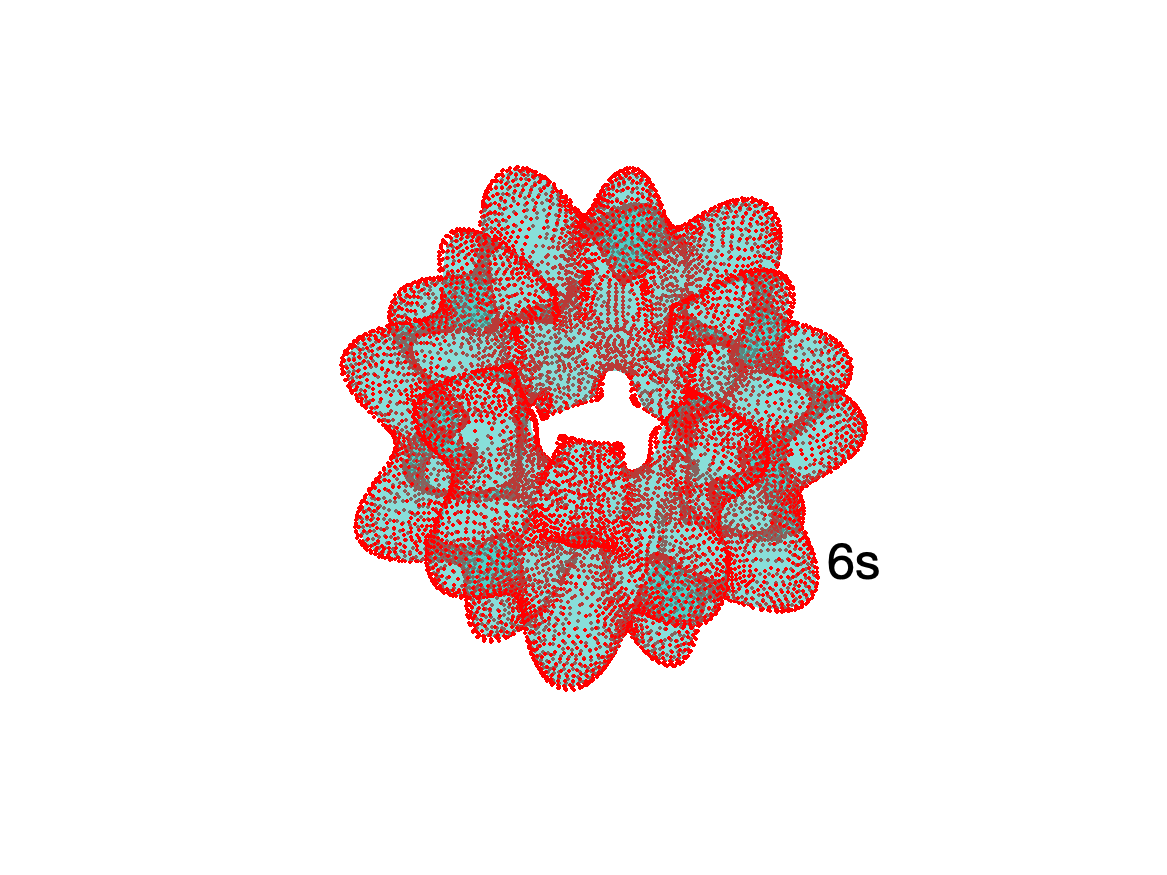} \  \
\includegraphics[width = 0.2\textwidth, clip, trim = 3cm 1.5cm 3cm 1cm]{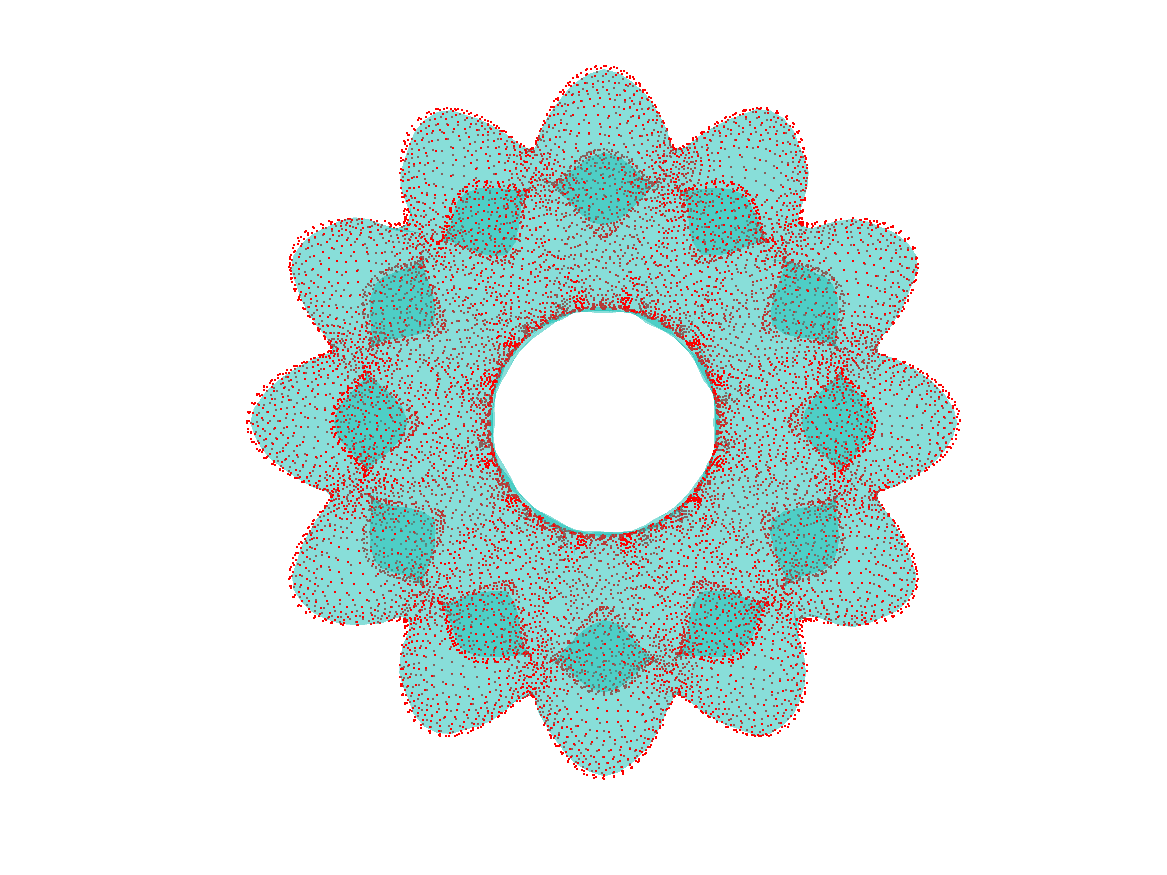}  \ \ 
\includegraphics[width = 0.2\textwidth, clip, trim = 2.5cm 1.5cm 1.5cm 1cm]{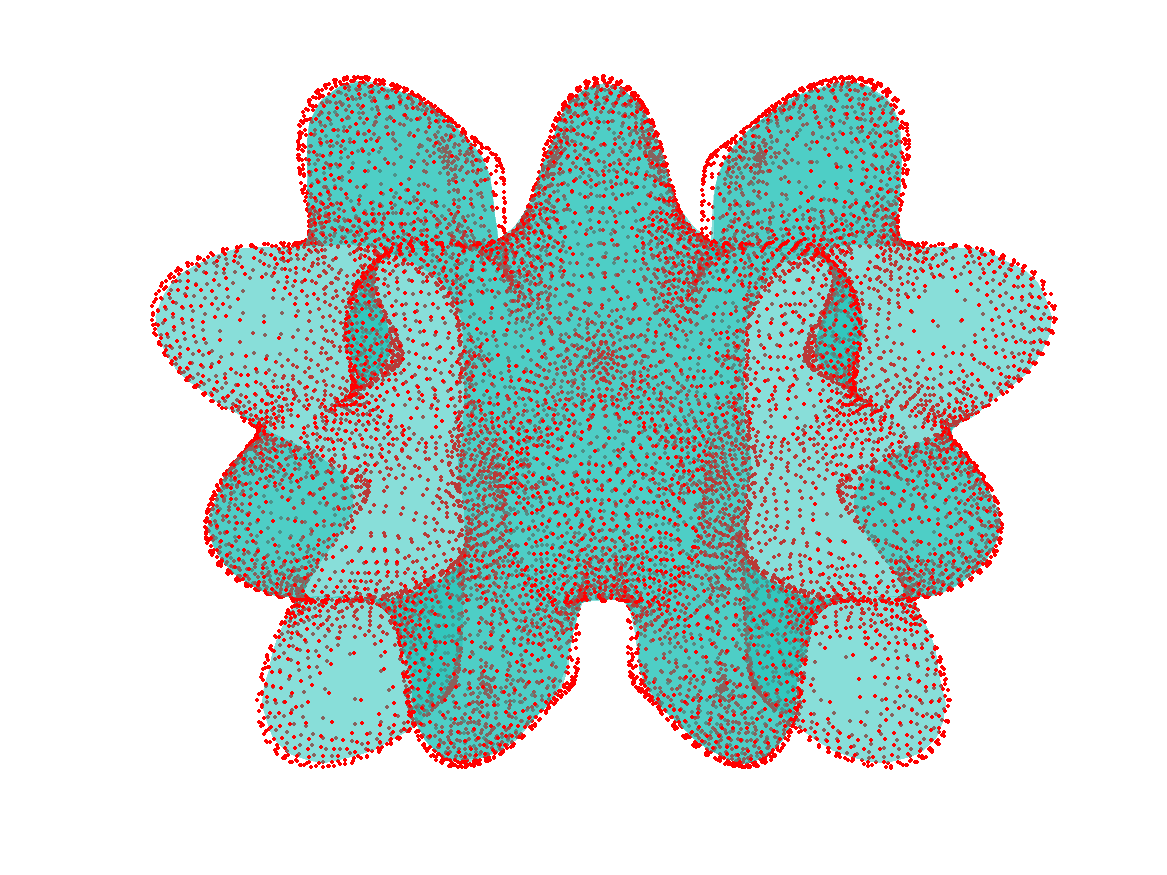} \  \
\includegraphics[width = 0.2\textwidth, clip, trim = 2.5cm 1.5cm 1.5cm 1cm]{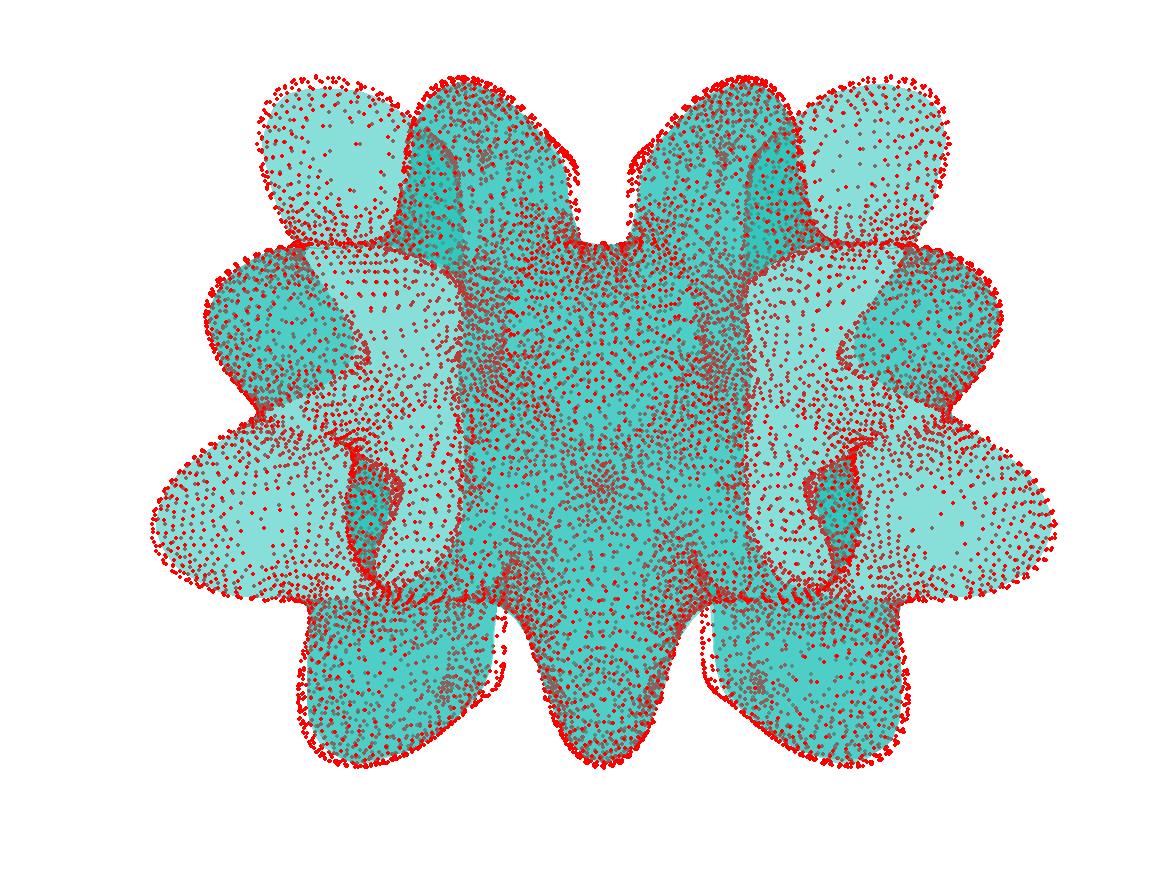} \\
\includegraphics[width = 0.2\textwidth, clip, trim = 6cm 2.5cm 4cm 3.5cm]{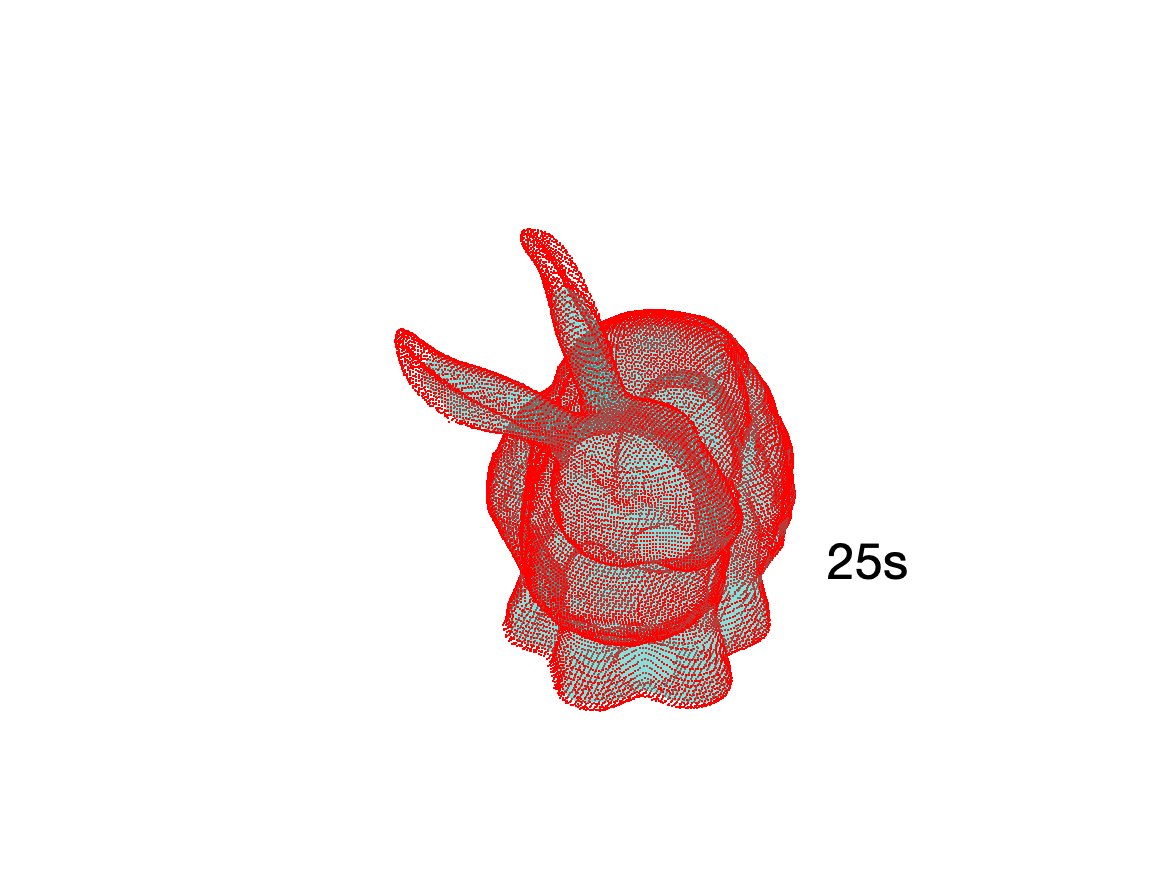} \ \
\includegraphics[width = 0.2\textwidth, clip, trim = 2cm 1.5cm 1.5cm 1cm]{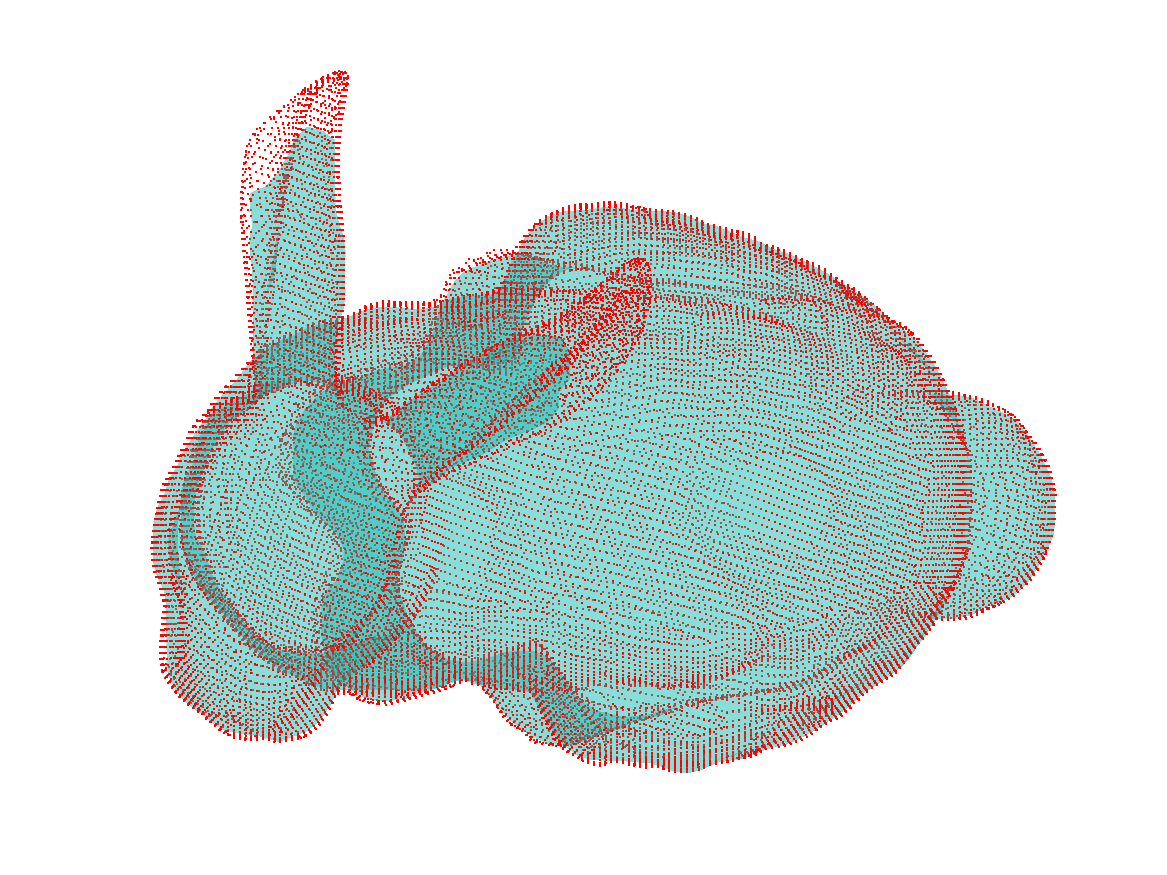}  \ \ 
\includegraphics[width = 0.2\textwidth, clip, trim = 4cm 1.5cm 2.5cm 1cm]{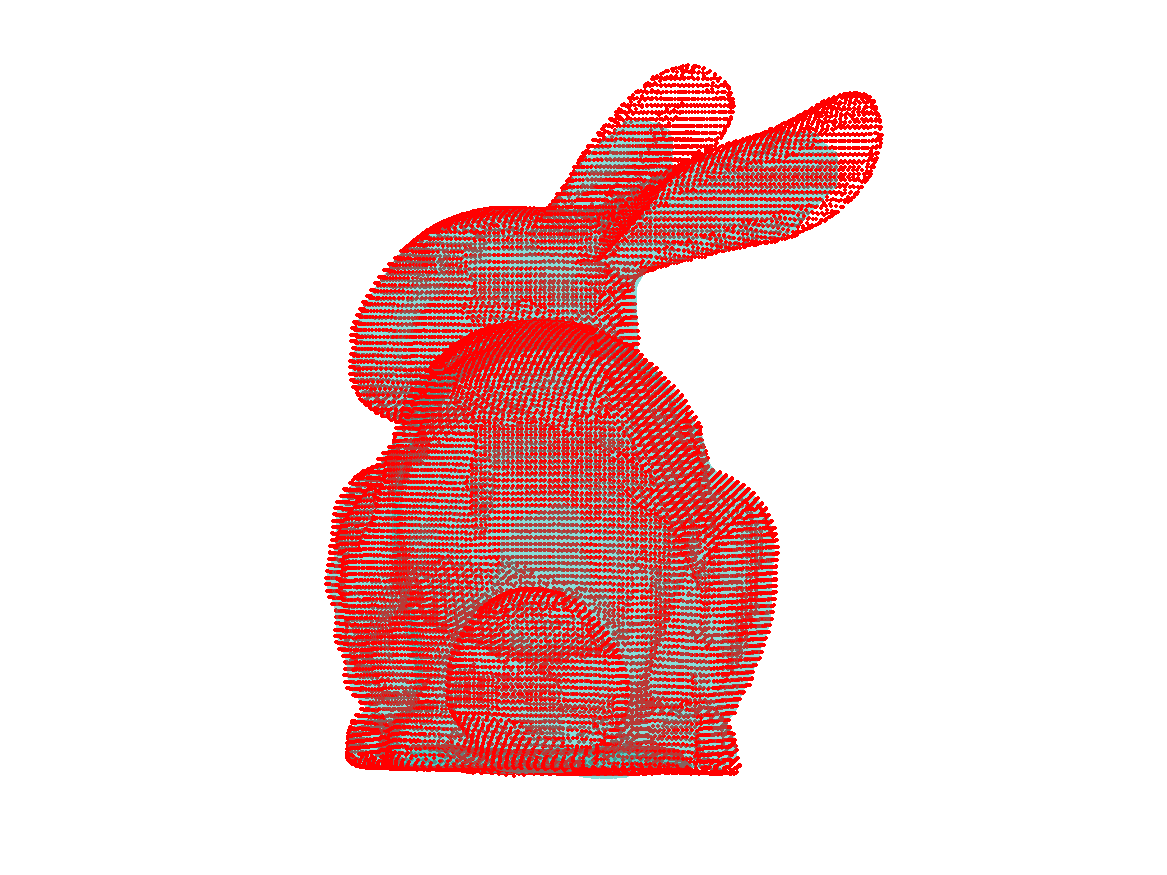} \  \
\includegraphics[width = 0.2\textwidth, clip, trim = 3cm 1.5cm 1.5cm 1cm]{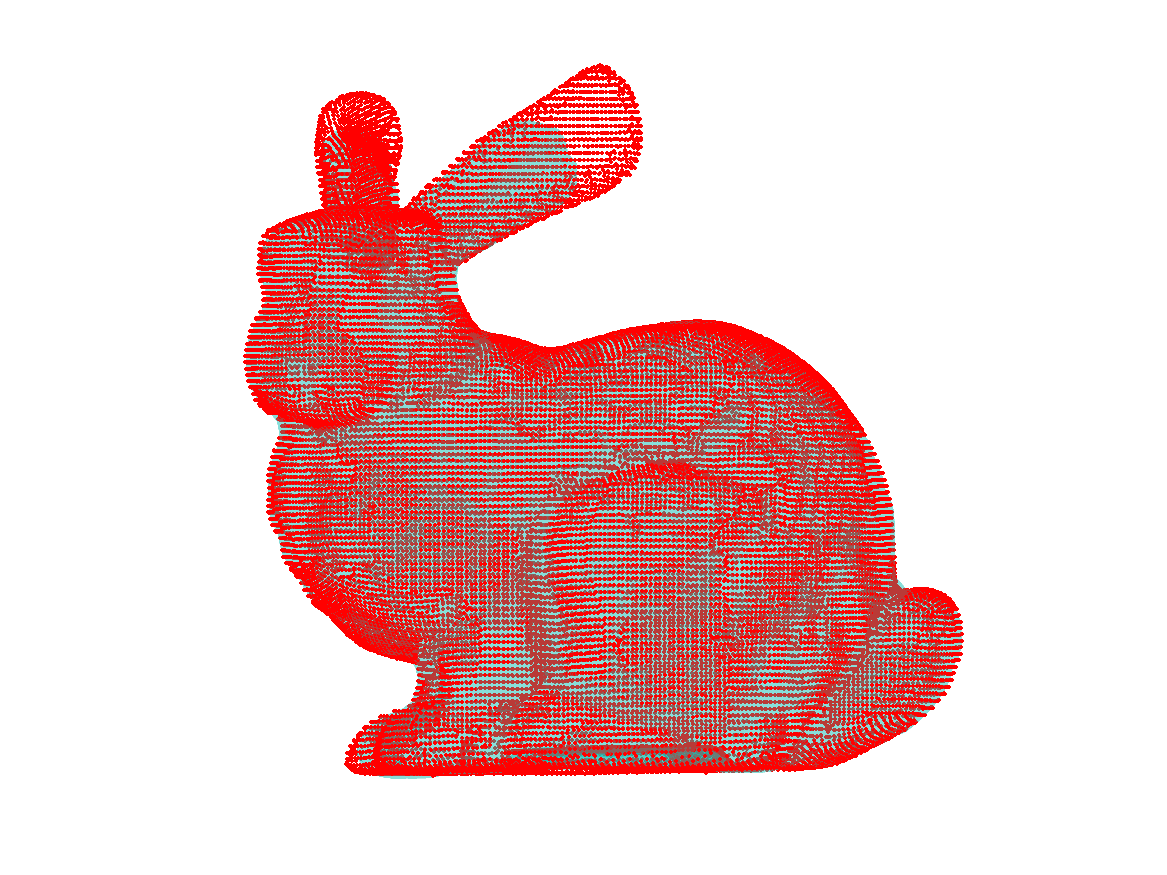} \\
\includegraphics[width = 0.2\textwidth, clip, trim = 3cm 2.5cm 1.5cm 2cm]{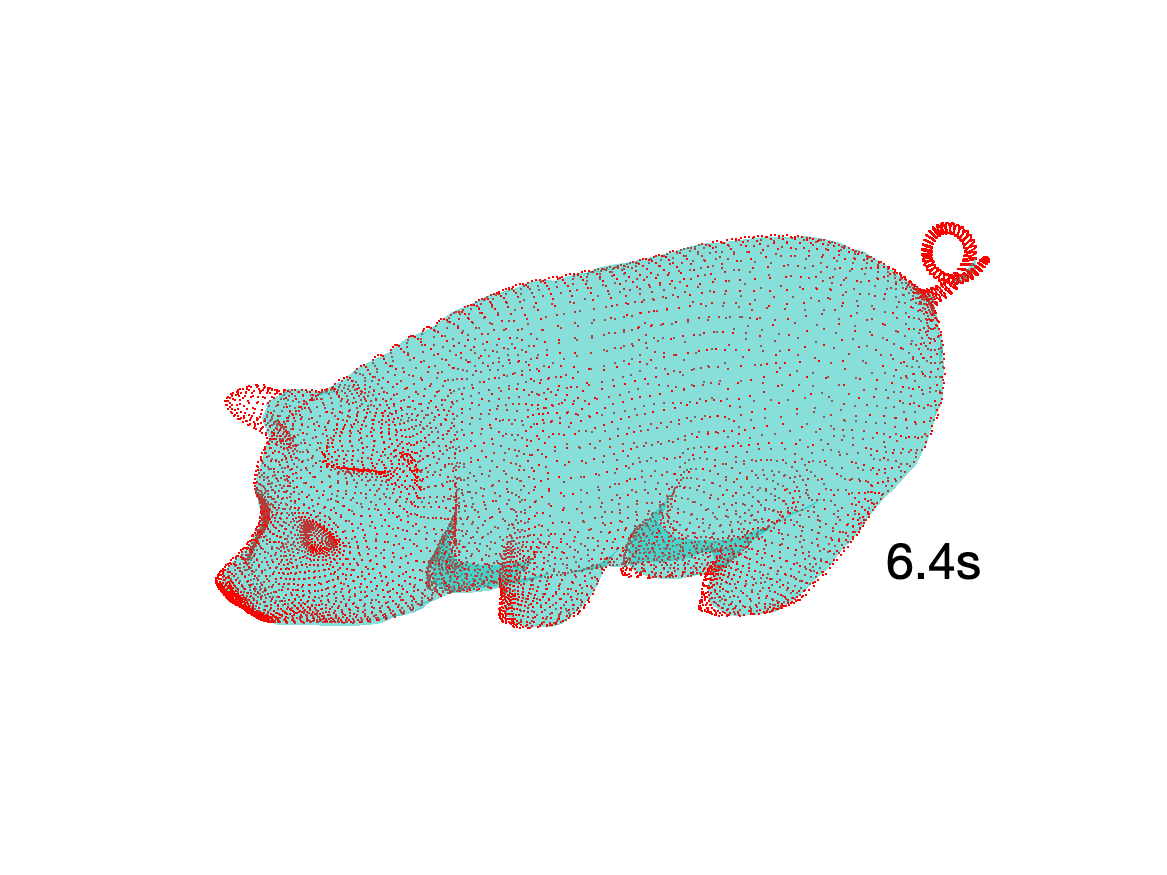} \  \
\includegraphics[width = 0.2\textwidth, clip, trim = 2cm 1.5cm 1.5cm 1cm]{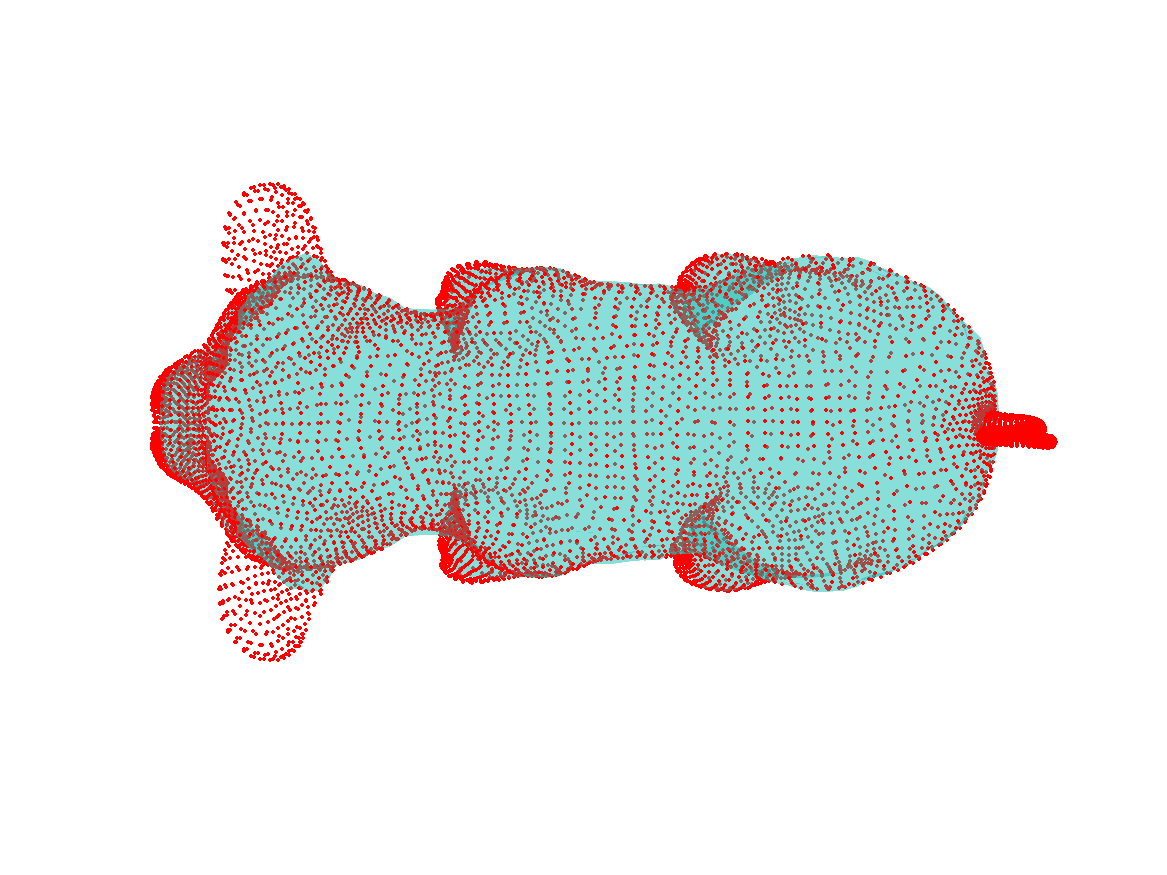} \  \
\includegraphics[width = 0.2\textwidth, clip, trim = 3cm 1.5cm 1.5cm 1cm]{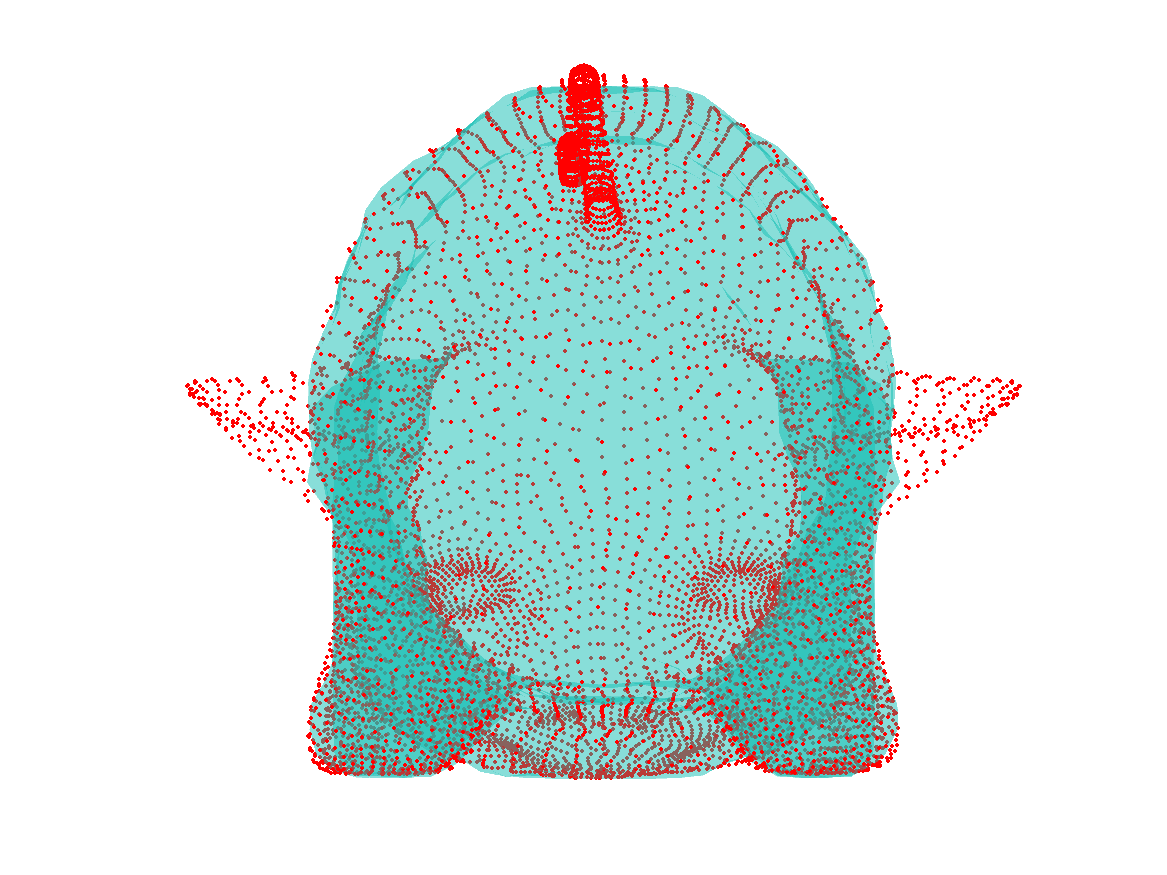} \  \
\includegraphics[width = 0.2\textwidth, clip, trim = 2cm 1.5cm 1.5cm 1cm]{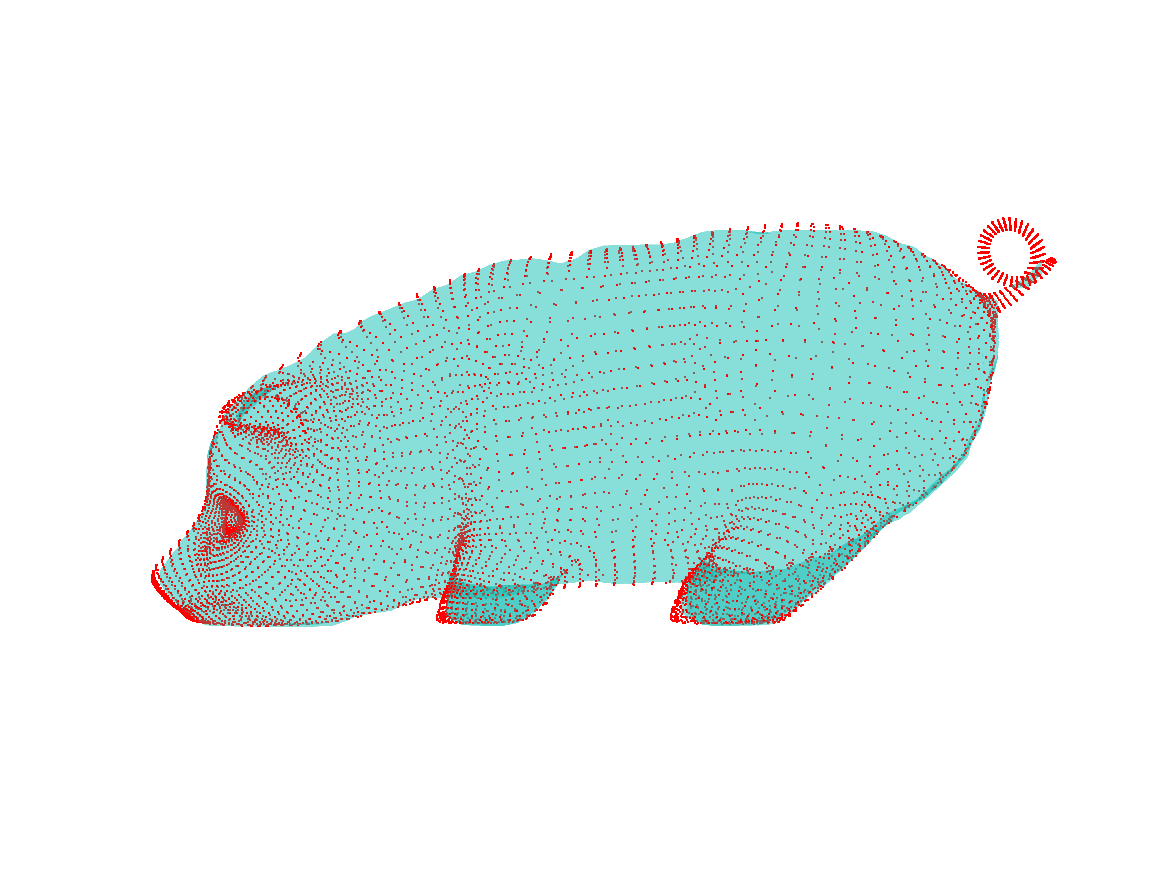} \\
\caption{Results obtained from Algorithm~\ref{a:MBO3}\_\ref{a:MBO2}. {\bf Left to right:} Reconstructed surface from point cloud, $xy$-view, $yz$-view, and $xz$-view. See Section~\ref{sec:3dex}.} \label{fig:3d}
\end{figure}

\section{Conclusion and discussions} \label{sec:con}
In this paper, we developed a novel iterative method to minimize an objective energy functional to reconstruct codimension-1 surfaces from point clouds in both 2- and 3-dimensional Euclidean spaces. The method is simple and unconditional stable in the sense of energy decaying. We carefully checked the properties and efficiency using a variety of numerical experiments. The proposed algorithms show great advantages than the level set approaches. 

From numerical experiments, we observe that thin parts with large curvature are difficult to reconstruct, especially for 3-dimensional point clouds. We expect this could be improved by considering more terms in the objective functional and this will be investigated and reported in the future. As for the Algorithm~\ref{a:MBO}, even all numerical results we performed so far imply the unconditional stability, the theoretical proof is still needed which we believe requires new mathematical tools. Since the surfaces are represented by indicator functions, the accuracy is dependent on the resolution of the discretization mesh. One could use a multi-scale strategy to refine the mesh after obtaining the results on a coarse mesh and set the results as the initial condition of the refined mesh. In addition, for those applications requiring finer design, the proposed methods can be used as an efficient tool for initialization which can be directly cooperated with a high order accurate method.

The CPU time for both 2- and 3-dimensional experiments have clearly shown the fast convergence of the algorithm. One could directly extend all proposed algorithms to higher dimensional problems. 

\subsection*{Acknowledgements}
The author would like to thank Wei Hu, Hao Liu, and Jun Ma for helpful discussions. The author would also like to thank Xiao-Ping Wang and Braxton Osting for constant support, help and encouragement.

\bibliography{refs}

\end{document}